\documentclass[%
	DIV=10, 
	english, 
	titlepage=false, 
	fontsize=10pt,
]{scrartcl}

\usepackage[utf8]{inputenc}
\usepackage[T1]{fontenc}
\usepackage{babel}
\usepackage{lmodern}
\usepackage{csquotes}
\usepackage{enumitem}
\usepackage{amsmath}
\usepackage{amssymb}
\usepackage{amsthm}
\usepackage{amsfonts}
\usepackage{dsfont}
\usepackage{mathtools}
\usepackage{esint}
\usepackage{cite}
\usepackage{array}
\usepackage{authblk}
\usepackage{nccmath}
\usepackage{bold-extra}
\usepackage{caption}
\usepackage{siunitx}

\usepackage{xcolor}
\definecolor{LinkColor}{rgb}{0,0,1}
\definecolor{LinkColor2}{rgb}{0,0.5,0}
\definecolor{lg}{rgb}{.5,.5,.5}

\definecolor{rosso}{rgb}{0.75,0,0}

\usepackage[%
left=1.25in,
right=1.25in,
bottom=1in,
top=1in,
footskip=0.35in,
]{geometry}

\usepackage[%
pdftitle={Titel},% 
pdfauthor={Autor},% 
pdfcreator={LaTeX, LaTeX with hyperref and KOMA-Script},% Genutzte Programme
pdfsubject={Betreff}, % Betreff
pdfkeywords={Keywords}
]{hyperref} 
\hypersetup{%
	colorlinks	=true,% Definition der Links im PDF File
	linkcolor	=LinkColor,%
	citecolor	=LinkColor2,%
	urlcolor	=LinkColor,%
}

\setkomafont{section}{\centering\large\normalfont\scshape}
\setkomafont{subsection}{\normalfont\bfseries}
\numberwithin{equation}{section}  
 
\makeatletter
\renewcommand{\@seccntformat}[1]{\csname the#1\endcsname.\hspace{1ex}}
\makeatother

\usepackage{tocloft}

\newtheorem{Thm}{Theorem}[section]
\newtheorem{Lem}[Thm]{Lemma}
\newtheorem{Pro}[Thm]{Proposition}

\newtheorem{Def}[Thm]{Definition}

\theoremstyle{definition}
\newtheorem{Rem}[Thm]{Remark} 
\newtheorem{Claim}[Thm]{Claim} 

\makeatletter
\renewcommand\paragraph{\@startsection{paragraph}{4}{\z@}%
    {1ex \@plus1ex \@minus.2ex}%
    {-1em}%
    {\normalfont\normalsize\bfseries}}
\renewcommand\subparagraph{\@startsection{paragraph}{4}{\z@}%
    {1ex \@plus1ex \@minus.2ex}%
    {-1em}%
    {\normalfont\normalsize\itshape}}
\makeatother

\newlength{\bibitemsep}
\newlength{\bibparskip}
\let\oldthebibliography\thebibliography
\renewcommand\thebibliography[1]{%
  \oldthebibliography{#1}%
  \setlength{\parskip}{\bibitemsep}%
  \setlength{\itemsep}{\bibparskip}%
}

\newcommand{\HH}{\mathcal{H}}

\newcommand{\abs}[1]{\left\vert#1\right\vert}

\newcommand{\bs}{\B{s}}

\newcommand{\B}[1]{\boldsymbol{#1}}

\newcommand{\BE}[2]{\left\langle\mathcal{E}\left(#1\right),\mathcal{E}\left(#2\right)\right\rangle_{\mathbb{C}(\B{\varphi})}}

\newcommand{\tG}{\tilde{\B{\mathcal{G}}}^{\B{m}}}

\newcommand{\Hd}{H^1_D(\Omega;\mathbb{R}^d)}

\newcommand{\HL}{H^1(\Omega;\mathbb{R}^N)\cap L^{\infty}(\Omega;\mathbb{R}^N)}

\newcommand{\HN}{H^1(\Omega;\mathbb{R}^N)}

\newcommand{\LN}[1]{L^{#1}(\Omega;\mathbb{R}^N)}

\newcommand{\LuN}{L^{\infty}(\Omega;\mathbb{R}^N)}

\newcommand{\LzS}{L^2_{T\Sigma}(\Omega;\mathbb{R}^N)}

\newcommand{\norm}[1]{\left\Vert #1\right\Vert}

\newcommand{\pd}{\B{\varphi}_{\delta}}

\newcommand{\rsp}[3]{\left(#1,#2\right)_{\rho(\B{#3})}}

\newcommand{\bphi}{{\B\varphi}}

\newcommand{\bx}{{\B x}}

\newcommand{\bw}{{\B w}}

\newcommand{\bW}{\mathbf{W}}

\newcommand{\R}{\mathbb{R}}

\newcommand{\N}{\mathbb{N}}

\newcommand{\E}{\mathcal{E}}

\newcommand{\C}{\mathbb{C}}

\newcommand{\eps}{\varepsilon}

\newcommand{\bnorm}[1]{\big\| #1 \big\|}

\newcommand{\babs}[1]{\ensuremath\big|#1 \big|}

\newcommand{\suchthat}{\;\ifnum\currentgrouptype=16 \middle\fi|\;}

\DeclareMathOperator{\trace}{tr}

\DeclareMathOperator{\sign}{sign}

\begin{document}

%%%%%%%%%%%%%%%%%%%%%%%%%%%%%%%%%%%%
%%%%%%%%%% TITLEPAGE %%%%%%%%%%%%%%%
%%%%%%%%%%%%%%%%%%%%%%%%%%%%%%%%%%%%
\begin{titlepage}
	\begin{addmargin}{0.25in}
		\begin{center}
		\LARGE{\scshape 
		Sharp-Interface Limit
		\\of a Multi-Phase Spectral 
		\\Shape Optimization Problem 
		\\for Elastic Structures
		}\\
		\rmfamily\mdseries
		\vspace{0.03\paperheight} 
		%%%%%
		\normalsize
		\textsc{Harald Garcke$^1$, Paul H\"uttl$^1$, Christian Kahle$^2$ and Patrik Knopf$^1$}\\
		\vspace{0.02\paperheight} 
		\begin{minipage}[t]{0.75\textwidth}
		\scriptsize\flushleft
		$^1$\textit{Fakul\"at f\"ur Mathematik, Universit\"at Regensburg, 93053 Regensburg, Germany},\\
		$\quad$
		\href{mailto:harald.garcke@ur.de}{harald.garcke@ur.de},
		\href{mailto:paul.huettl@ur.de}{paul.huettl@ur.de},
		\href{mailto:patrik.knopf@ur.de}{patrik.knopf@ur.de}.\\[1ex]
		$^2$\textit{Mathematisches Institut, Universit\"at Koblenz, 56070~Koblenz, Germany},\\
		$\quad$
		\href{mailto:christian.kahle@uni-koblenz.de}{christian.kahle@uni-koblenz.de}.
		\end{minipage}
		\\
		\vspace{0.02\paperheight} 
		\scriptsize
		\color{white}
		{
			\textit{This is a preprint version of the paper: Please cite as:} \\  
			H. Garcke, P. H\"uttl, Christian Kahle and P. Knopf, [Journal] (2023) \\ 
			\texttt{https://doi.org/...}
		}
		\end{center}
		\vspace{0.02\paperheight} 
		\footnotesize
		\textsc{Abstract.}
		\normalfont
		We consider an optimization problem for the eigenvalues of a multi-material elastic structure that was previously introduced by \textsc{Garcke et al.}
		[\textit{Adv.~Nonlinear Anal.}~11 (2022), no.~1, 159--197]. There, the elastic structure is represented by a vector-valued phase-field variable,
		and a corresponding optimality system consisting of a state equation and a gradient inequality was derived. In the present paper, we pass to the sharp-interface limit in this optimality system by the technique of formally matched asymptotics. Therefore, we derive suitable Lagrange multipliers to formulate the gradient inequality as a pointwise equality.
		Afterwards, we introduce inner and outer expansions, relate them by suitable matching conditions and formally pass to the sharp-interface limit by comparing the leading order terms in the state equation and in the gradient equality. Furthermore, the relation between these formally derived first-order conditions and results of \textsc{Allaire \& Jouve} [\textit{Comput.~Methods Appl.~Mech.~Engrg.}, 194 (2005), pp.~3269--3290] obtained in the framework of classical shape calculus is discussed. Eventually, we provide numerical simulations for a variety of examples. In particular, we illustrate the sharp-interface limit and also consider a joint optimization problem of simultaneous compliance and eigenvalue optimization.
        \\[1ex]
		\\[1ex]
	\end{addmargin}
\end{titlepage}

%%%%%%%%%%%%%%%%%%%%%%%%%%%%%%%%%%%%
%%%%%%%%%% DOCUMENT %%%%%%%%%%%%%%%%
%%%%%%%%%%%%%%%%%%%%%%%%%%%%%%%%%%%%

\bigskip
\normalsize
\setlength{\parskip}{1ex}
\setlength{\parindent}{0ex}
% \allowdisplaybreaks
\normalsize

%%%%%%%%%%%%%%%%%%%%%%%%%%%%%%%%%%%%
%%%%%%%%%%%%INTRODUCTION%%%%%%%%%%%%
%%%%%%%%%%%%%%%%%%%%%%%%%%%%%%%%%%%

\section{Introduction}\label{SEC:Intro}

The goal of structural shape and topology optimization is to find the optimal distribution of materials in a prescribed region, the so-called design domain.
Here, in addition to pure shape optimization, also the topology of the structure is to be optimized. This includes the formation of holes (void regions) in the structure as well as the merging and splitting of connected material components.
In many applications, certain properties of the materials (such as their elastic properties) as well as additional side conditions (e.g., volume constraints or support conditions) need to be taken into account within the optimization problem.

Besides the optimization of shape and topology, the optimization of eigenvalues is an important task in engineering science to make structures robust against vibrations. 
It has been observed that structures are less susceptiple against vibrations if their principal eigenvalue is large, see \cite[Section 2]{Bendsoe}, \cite{AllaireEig} and also \cite{Garcke} for concrete examples and further references.
Heuristically, this can be explained by the fact that larger principal eigenvalues are associated with higher temporal frequencies which correspond to smaller wavelengths of the oscillations.

The traditional mathematical tool to handle shape optimization problems is the calculus of shape derivatives based on boundary variations (see, e.g., \cite{Allaire-Jakabcin,Allaire-Dapogny,Delfour1,Murat-Simon,Simon,Sokolowski}). 
However, frequent remeshing leads to high computational costs and it cannot deal with topological changes, see also \cite{Oudet} for a comprehensive discussion. In some situations, it is possible to handle topology changes by means of homogenization methods (see, e.g., \cite{AllaireBook}) or variants of this approach such as the SIMP method (see, e.g., \cite{Bourdin,Bendsoe}). 
A drawback of this method occuring in applications to spectral problems is the phenomenon of so-called \textit{localized eigenmodes} 
(also often referred to as \textit{spurious eigenmodes}),
see \cite{BucurMartinetOudet,Bendsoe,Pedersen,Allaire}. In this context localized eigenmodes are eigenfunctions which are supported only in the void regions and pollute the spectrum with low eigenvalues. 
Especially in recent times, the level-set method has become a popular approach for topology optimization problems. 
After the method was developed in \cite{Osher-Sethian}, it has been used extensively in the literature (see, e.g., \cite{Burger,Osher-Santosa,Oudet,Allaire,AntunesOudet,KaoOsher}). Although the level-set method is capable of dealing with topological changes, difficulties can arise if voids are to be created.

In this paper, we consider an optimization problem that was introduced in \cite{Garcke}. There, the authors employed a different method to optimize the shape and the topology as well as a finite selection of eigenvalues of an elastic structure, namely the so-called \textit{(multi-)phase-field approach}. This method for shape and topology optimization was first developed in \cite{Bourdin} and subsequently used frequently in the literature. We refer the reader to \cite{Auricchio2,Blank,Blank2,Blank3,Bourdin-Chambolle,Burger-Stainko,Carraturo,Dede,Ebeling-Rump,Ebeling-Rump2,Penzler,Marino,DondlPohRumpfSimon}
to at least mention some of the various contributions.

In \cite{Garcke}, an elastic structure consisting of $N-1$ materials is described by a \textit{multi-phase-field variable}. This is a vector-valued function $\bphi:\Omega\to \R^N$ whose components $\varphi^1,...,\varphi^{N-1}$ represent the volume fractions of the materials, and $\varphi^N$ represents the void (i.e., the region where no material is present). In particular, the components of $\bphi$ are restricted to attain their values only in the interval $[0,1]$.
In most parts of the design domain, the materials are expected to appear in their pure form, meaning that the corresponding component of the multi-phase-field $\bphi$ attains the value one, whereas all other components are zero. 
These regions are separated by \textit{diffuse interfaces}, which are thin layers between the pure phases  
whose thickness is proportional to a small parameter $\eps>0$. 
In particular, $\bphi$ is expected to exhibit a continuous transition between the values zero and one at these diffuse interfaces. The main advantage of the phase-field approach in the context of shape and topology optimization is that topological changes (such as merging or splitting of material components or the creation of holes) during the optimization process can be handled without any problems.
The optimization problem in \cite{Garcke} is formulated as a minimization problem for an objective functional which involves a selection of eigenvalues as well as a Ginzburg--Landau type penalisation term for the phase-field.
For this problem, the existence of at least one global minimizer was established and a first-order necessary optimality condition for local minimizers was derived. A detailed mathematical formulation of the optimization problem from \cite{Garcke} will be presented in Section~\ref{SEC:Formulation}.

The main goal of this paper is to derive the \textit{sharp-interface limit} of the aforementioned optimization problem from \cite{Garcke}.
This means that we want to send the parameter $\eps$, that is related to the thickness of the diffuse interface, to zero. 
In this way, we can relate the diffuse-interface approach from \cite{Garcke} to the physically reasonable scenario of sharp-interfaces.
In particular, one of our key goals is to show that minimizers of the problem in the diffuse-interface framework converge to minimizers of a corresponding sharp-interface optimization problem. 

Qualitatively, there are two ways to deal with this passage to the limit: the rigorous investigation of the \textit{$\Gamma$-limit} of the involved cost functional, and the formal method of \textit{matched asymptotic expansions}.
For a rigorous discussion of the sharp-interface limit of diffuse-interface models describing elastic systems, we refer the reader to \cite{Blank-SharpInt,AlmiStefanelli}. There, the void is modeled as a further material having low but non-degenerate stiffness, which is crucial for the analysis.
Up to now, to the best of our knowledge, there is no rigorous $\Gamma$-limit analysis for spectral problems in the case of degenerating stiffness in the void regions.

As a first step towards the task of dealing with this delicate problem,
the sharp-interface $\Gamma$-limit for an optimization problem involving a selection of eigenvalues of the Dirichlet Laplacian was rigorously established in \cite{GarHKK_OptLapEV}. A relation between the minimization of the principal eigenvalue of the Dirichlet Laplacian on the phase-field level and the Faber--Krahn inequality on the sharp-interface level was discussed in \cite{HKL}. 

In order to understand the sharp interface limit, we thus intend to apply the technique of \textit{formally matched asymptotic expansions} on the optimization problem from \cite{Garcke}. This technique has already been employed on different phase-field models (especially of Allen--Cahn or Cahn--Hilliard type), see, e.g., \cite{Abels, Abels-Liu, Barrett, Blank, Bronsard, Hilhorst, GLSS, GLNS,GHHKL}. For comprehensive overviews of this technique, we refer to \cite{Fife, Eck, Kevorkian}.

The basic strategy of this formal approach is as follows: We assume that the phase-field as well as the corresponding eigenvalues and eigenfunctions each possess an \textit{inner asymptotic expansion} and an \textit{outer asymptotic expansion}, both given by a power series with respect to the interface parameter $\varepsilon$. 
The inner expansions approximate the aforementioned quantities ``close'' to the diffuse interface where the phase-transition takes place, whereas the outer expansions approximate these quantities in regions that are ``far'' away from the interface where only the pure phases are present.

Plugging the outer expansions into the eigenvalue equation on the diffuse-interface level, a comparison of the leading order terms leads to limit eigenvalue equations on the sharp-interface level.
At this point we will include a discussion about localized eigenmodes. As also mentioned above, in numerical simulations the formation of eigenmodes that are supported only in void areas and produce eigenvalues which pollute the low part of the spectrum (which we are interested in) is a major problem. We will see that our asymptotic approach is able to deal with such localized eigenmodes. More precisely, we will see that if such modes appear, then the corresponding eigenvalues will diverge to infinity as $\eps\to 0$. Thus, if $\eps>0$ is sufficiently small, localized eigenmodes do not affect the lower part of the spectrum that is considered in our optimization problem.

The inner expansions are used to describe the aforementioned quantities in tubular neighborhoods around the interfaces. The distinction between inner and outer regions is needed as we expect the phase-field to change its values rapidly in regions close to the interface. This is because the diffuse interface will become infinitesimally thin as $\eps\to 0$. Here, the main idea is to introduce a rescaled coordinate system which takes the $\eps$-scaling of this region into account.

After studying these two forms of expansions separately, it is crucial to match both expansions in a so-called intermediate region. This means we compare both extensions by exploiting the two different coordinate systems they are formulated in. Plugging these relations into the optimality system and comparing the leading order terms, we obtain boundary conditions for the previously obtained limit eigenvalue equations. We observe that the boundary condition on the free boundary will essentially be of homogeneous Neumann type. Furthermore, we use the inner expansions to derive a limit equality from the strong formulation of the gradient inequality.
The limit eigenvalue equations together with this gradient equality will then constitute the optimiality system of a corresponding sharp-interface optimization problem. This will be justified from the viewpoint of classical shape calculus (see, e.g., \cite{Allaire}) by relating the limit of the gradient inequality to the shape derivative of the associated cost functional.

However, in order to apply the technique of formally matched asymptotics, we first need to reformulate the gradient inequality on the diffuse-interface level as a pointwise gradient equality by introducing suitable Lagrange multipliers.
Under an additional regularity assumption on the involved eigenfunctions, this is achieved by employing a regularization technique following the ideas of \cite{Sarbu}, in which we eventually pass to the limit. A key benefit of this strategy is that it provides an explicit construction of the Lagrange multipliers arising from the constraints of our optimization problem. This specific knowledge about the Lagrange multipliers will turn out to be essential for the asymptotic analysis. As a byproduct, we also prove that the phase-field variable $\bphi$ solving the original gradient inequality is actually $H^2$-regular under the aforementioned assumption of suitably regular eigenfunctions.

The present paper is structured as follows. In Section~\ref{SEC:Formulation}, we first introduce the theory that is necessary to formulate the diffuse-interface optimization problem along with its first-order necessary optimality condition. The derivation of the strong formulation of the gradient inequality will then be performed in Section~\ref{SEC:Ana}.
Using the outer expansions, we derive the state equations of the limit problem in Section~\ref{SEC:outer}.
In order to construct inner expansions we first analyze in Section~\ref{SEC:inner} how the involved differential operators are reformulated in a suitable rescaled coordinate system. Suitable matching conditions connecting the outer expansions with the inner expansions are then derived in Section~\ref{SEC:Match}. In Section~\ref{SEC:Leading}, we use the inner expansions to derive boundary conditions on the free boundary in the sharp-interface setting as well as the sharp-interface limit of the gradient inequality. Then, in Section~\ref{SEC:limitpr}, we comprehensively state the limit optimality system, and in Section~\ref{SEC:Rel}, the first-order necessary optimality condition on the sharp-interface level is related to classical shape calculus. Eventually, in Section~\ref{sec:num}, we present several numerical solutions for concrete optimization problems on the diffuse interface-level. In this context, we also discuss suitable choices of the model parameters. In particular, we observe that our results using the phase-field approach compare very well to similar numerical results obtained in \cite{Allaire} by means of the level-set method combined with classical shape calculus on the sharp interface level.

%%%%%%%%%%%%%%%%%%%%%%%%%%%%%%%%%%%
%%%%%%%%%%%%FORMULATION%%%%%%%%%%%%
%%%%%%%%%%%%%%%%%%%%%%%%%%%%%%%%%%%

\section{Formulation of the problem}\label{SEC:Formulation}
In this section, we recall the framework introduced in \cite{Garcke} in order to formulate and understand the optimality system our analysis is based on. Therefore, we first introduce the key assumptions which shall hold throughout the paper.
\subsection{General assumptions}
\begin{enumerate}[label = (A\arabic*), leftmargin=*]
    \item \label{ASS:A:1} The \textit{design domain} $\Omega\subset \mathbb{R}^d$ is a bounded Lipschitz domain with $d\in \N$ and outer unit normal vector field $\B{n}$.
    Its boundary is split into two disjoint parts: A homogeneous Dirichlet boundary $\Gamma_D$ with strictly positive $(d-1)$-dimensional Hausdorff measure and a homogeneous Neumann boundary $\Gamma_0$. We define
    \begin{align*}
        H^1_D(\Omega;\R^d)\coloneqq
        \left\{\B{\eta}\in H^1(\Omega;\R^d)\suchthat \B{\eta}=\B{0}\text{ a.e.~on }\Gamma_D\right\}.
    \end{align*}
    \item \label{ASS:A:2} 
    The potential $\psi: \mathbb{R}^N \to \mathbb{R}_+\cup \left\{+\infty\right\}$ attains exactly $N$ global minima of value $0$ attained at the points $\B{e}_i$, i.e.,
    \begin{equation*}
    	\min\psi = \psi(\B{e}_i) = 0 \quad\text{for all $i\in\{1,...,N\}$},
    \end{equation*}
    where $\B{e}_i\in \R^N$ denotes the $i$-th standard basis vector in $\R^N$.
    Additionally, we assume $\psi$ to be decomposed into $\psi(\B{\varphi})=\psi_0(\B{\varphi})+I_{\B{G}}(\B{\varphi})$
    with $\psi_0\in C^1(\R^N,\R)$ and
    the indicator functional
    \begin{equation*}
        I_{\B{G}}(\B{\varphi})=
            \begin{cases}
                0&\text{if } \B{\varphi}\in\B{G},\\
                +\infty& \text{otherwise},
        \end{cases}
    \end{equation*}
    where $\B{G}:=\mathbb{R}^N_{+}\cap \Sigma^N$ with
    \begin{align}
        \label{DEF:SIGMA}
        \Sigma^N&:=\left\{\B{\xi}\in \mathbb{R}^N \;\left|\;\, \sum_{i=1}^{N}\xi^{i}=1\right.\right\},
        \\
        \mathbb{R}^N_{+}
        \label{DEF:RNP}
        &:=
            \left\{
                \left.\B{\xi}\in \mathbb{R}^N \,\right|\, \forall i\in \left\{1,\dots, N\right\}:\; \xi^{i}\ge 0 
            \right\}.
    \end{align}
    The set $\B{G}$ is referred to as the \textit{Gibbs simplex}.
    A prototype example for the continuous part $\psi_0$ would be $\psi_0(\bphi)=\frac{1}{2}(1-\bphi\cdot\bphi)$ (cf.~\cite{Blank}).
    \item \label{ASS:A:3} 
    The function $\Psi: \left(\mathbb{R}_{>0}\right)^l\to \mathbb{R}$
    is continuously differentiable.
\end{enumerate}
\subsection{The phase-field variable}\label{Sec:Phase}
To describe the material distribution of $(N-1)$ different materials in the design domain $\Omega$, we introduce the phase-field $\B{\varphi}:\Omega\to \mathbb{R}^N$. Its components $\varphi^i$, $i=1,...,N-1$ represent the materials, whereas $\varphi^N$ represents the void. We expect $\B{\varphi}$ to \textit{continuously} change its values at the diffuse interface. From a physical point of view, this means that the considered materials can be mixed at the interfacial region. 
In order for the phase-field to behave in a physically reasonable way we impose suitable constraints.
First of all, we fix the total amount of each material by the mean value constraint
\begin{align}\label{Volcon}
\fint_{\Omega}\B{\varphi}\mathrm{\,d}x=\B{m}=\big(m^i\big)_{i=1}^{N},
\end{align}
with $m^i\in (0,1)$ and $\B{m}\in \Sigma^N$ (cf.~\eqref{DEF:SIGMA}).
The constraint $\B{m}\in \Sigma^N$ is a consequence of the physical assumption that the amount of the individual volume fractions $\varphi^i$ needs to sum up to $1$ at each point in the domain. 
Furthermore, it is physically reasonable to assume that each volume fraction shall attain its values only in the interval $[0,1]$. This property is incorporated by assuming that any $\bphi$ belongs to the \textit{set of admissible phase-fields}
\begin{equation*}
    \B{\mathcal{G}}^{\B{m}}=\left\{\B{\varphi}\in \B{\mathcal{G}}\left|\,\fint_{\Omega}\B{\varphi}\mathrm{\,d}x=\B{m}\right.\right\},
\end{equation*}
with
\begin{align*}
    \B{\mathcal{G}}
        :=\left\{
                \left.\B{\varphi}\in \HN \right|\,\B{\varphi}(\bx)\in \B{G}\;\;\text{for almost all $\bx\in\Omega$}
            \right\}.
\end{align*}
Here, $\B{G}$ is the Gibbs simplex that was introduced in \ref{ASS:A:2}.

\subsection{The Ginzburg--Landau energy} \label{SECT:GL}
In order to make our optimzation problem well-posed, we need to include a regularizing term for the phase-field in the cost functional. For this purpose, we use the so-called \textit{Ginzburg--Landau} energy 
\begin{equation}     
\label{eq:prob:GinzLan}
        E^{\varepsilon}(\B{\varphi})
    =
        \int_{\Omega}
        \left(
            \frac{\varepsilon}{2}\abs{\nabla\B{\varphi}}^2+\frac{1}{\varepsilon}\psi(\B{\varphi})
        \right) \,\mathrm dx, 
\end{equation}
for all $\bphi\in H^1(\Omega;\R^N)$.
Here, the parameter $\varepsilon>0$ is related to the the thickness of the diffuse-interface and therefore, it is usually chosen very small. In the sharp-interface limit, we intend to (formally) send this parameter to zero. Due to assumption \ref{ASS:A:2}, the potential $\psi$ enforces the phase-field $\B{\varphi}$ to attain its values only in the Gibbs simplex. However, as we already include the the Gibbs simplex constraint in the set of admissible phase-fields, it suffices to merely consider the regular part $\psi_0$ of the potential $\psi$ in the Ginzburg--Landau energy as long as $\bphi\in\B{\mathcal{G}}$. This means that
\begin{equation*}
    E^{\eps}(\bphi)=
            \int_{\Omega}
        \left(
            \frac{\varepsilon}{2}\abs{\nabla\B{\varphi}}^2+\frac{1}{\varepsilon}\psi_0(\B{\varphi})
        \right) \,\mathrm dx
\end{equation*}
for all $\bphi\in\B{\mathcal{G}}$.

\subsection{The elasticity tensor and the density function}
As we intend to consider an elastic structure, we next introduce the two tensors of linear elasticity, which will be used to formulate the state equation.
The \textit{strain tensor} of a vector-valued function $\B{u}\in \Hd$ is given as
\begin{align*}
    \mathcal{E}(\B{u})\coloneqq\left(\nabla \B{u}\right)^{\text{sym}}=\frac{1}{2}\left(\nabla \B{u}+\nabla\B{u}^T\right).
\end{align*}
The \textit{elasticity tensor} $\mathbb{C}: \R^N\to \R^{d\times d\times d\times d}$ is a fourth order tensor with the following properties.
\begin{enumerate}[label = (B\arabic*), leftmargin=*]
    \item \label{ASS:B:1} $\mathbb{C}_{ijkl}\in C^{1,1}_\text{loc}(\mathbb{R}^N;\mathbb{R})$.
    \item \label{ASS:B:2} $\mathbb{C}$ is symmetric, i.e.,%
    \begin{equation*}
        \mathbb{C}_{ijkl}=\mathbb{C}_{jikl}=\mathbb{C}_{ijlk}=\mathbb{C}_{klij}\,,
    \end{equation*}
    for $i,j,k,l=1,\dots,d$.
    \item \label{ASS:B:3} $\mathbb{C}$ is coercive for any fixed $\eps>0$, i.e., there exists $\theta_\eps>0$ such that
    \begin{align*}
    \theta_\eps \abs{\mathcal{B}}^2\le \mathbb{C}(\bphi)\, \mathcal{B}:\mathcal{B}\,,
    \end{align*}
    for all $\bphi\in \mathbb{R}^N$ and all symmetric matrices $\mathcal{B}\in \mathbb{R}^{d\times d}$.
    For two matrices $\mathcal{A},\mathcal{B}\in \mathbb{R}^{d\times d}$ this product is defined as
    \begin{equation*}
    \mathcal{A}:\mathcal{B}\coloneqq 
    \sum_{i,j=1}^{d}\mathcal{A}_{ij}\mathcal{B}_{ij}\;.
    \end{equation*}
\end{enumerate}

The component specific densities are modeled by a density function $\rho: \R^N\to \R$ with the following properties.
\begin{enumerate}[label = (C\arabic*), leftmargin=*]
    \item \label{ASS:C:1} $\rho\in C^{1,1}_\text{loc}(\mathbb{R}^N;\mathbb{R})\,$.
    \item \label{ASS:C:2} $\rho$ is uniformly positive for any fixed $\eps>0$, i.e., there is a constant $\rho_{0,\eps}>0$ such that $\rho(\bphi)\ge \rho_{0,\eps}$ for all $\bphi\in \R^N$.
\end{enumerate}
 
As in \cite{Blank}, we want $\mathbb{C}$ and $\rho$ to possess a decomposition that reflects the material specific elasticity and density of the $N-1$ materials. Therefore, for $\bphi\in \B{G}$, we set
\begin{align}\label{elT}
\begin{aligned}
	\mathbb{C}(\B{\varphi})
	&=\overline{\mathbb{C}}(\B{\varphi})
	    +\tilde\C^N \eps^k \alpha_V(\varphi^N)
	=\sum_{i=1}^{N-1}\mathbb{C}^{i}\alpha_M(\varphi^i)
	    +\tilde\C^N \eps^k \alpha_V(\varphi^N),\\
	\rho(\B{\varphi})
	&=\overline{\rho}(\B{\varphi})+\tilde\rho^N \eps^l \beta_V(\varphi^N)
	=\sum_{i=1}^{N-1}\rho^{i}\beta_M(\varphi^i)+\tilde\rho^N \eps^l \beta_V(\varphi^N),
\end{aligned}
\end{align}
for any $k,l\in\N$ and any $\alpha_M,\alpha_V,\beta_M,\beta_V \in C^{1,1}_{\text{loc}}(\R)$ with
\begin{align}\label{eq:def_inter}
    \begin{split}
    &\alpha_M(0)=\alpha_V(0)=
    \beta_M(0)=\beta_V(0)=0,\\
    &\alpha_M(1)=\alpha_V(1)=
    \beta_M(1)=\beta_V(1)=1,\\
    &\alpha_V(s), \beta_V(s), \alpha_M(s),\beta_M(s) \in (0,1) \qquad \text{ for all } s\in (0,1).
   \end{split}
\end{align}
 In this way, we have 
 \begin{align}\label{eq:interpol}
    \begin{aligned}
    \mathrm{supp}\, \alpha_M(\varphi^i) &= \mathrm{supp}\, \beta_M(\varphi^i) = \mathrm{supp}\, \varphi^i
    \qquad \text{for $i=1,\dots,N-1$},\\
    \mathrm{supp}\, \alpha_V(\varphi^N) &= \mathrm{supp}\, \beta_V(\varphi^N) = \mathrm{supp}\, \varphi^N.
    \end{aligned}
 \end{align}
The positivity conditions are necessary to satisfy the assumptions \ref{ASS:B:3} and \ref{ASS:C:2}. The global non-negativity of $\alpha_M,\beta_M$ will be essential in avoiding spurious eigenmodes, see Section~\ref{intermezzo}.

This means, for $i\in\{1,...,N-1\}$, we choose component specific but constant elasticity tensors $\mathbb{C}^i\in \R^{d\times d\times d\times d}$ and densities $\rho^i >0$. As the void obviously has neither a stiffness nor a density, we approximate the void components by some fixed elasticity tensor $\tilde{\mathbb{C}}^N\in \R^{d\times d\times d\times d}$ and density $\tilde{\rho}^N > 0$ that are multiplied by the small interface parameter $\eps$ that was introduced in Section~\ref{SECT:GL} in the context of the Ginzburg--Landau energy. Of course, these constant prefactors need to be chosen such that the assumptions \ref{ASS:B:2}, \ref{ASS:B:3} and \ref{ASS:C:2} are satisfied, see \cite{Garcke}.

Even though an adequate scaling of the void components $\tilde{\C}^N$ and $\tilde{\rho}^N$ with respect to $\eps$ combined with an appropriate choice of interpolation functions $\alpha_M,\alpha_V,\beta_M,\beta_V$ will be crucial for the numerical simulations in order to avoid spurious eigenmodes, see also Section~\ref{intermezzo}, we emphasize that our formal analysis works for any kind of decomposition as in \eqref{elT} as long as the void components are scaled with  $\eps^p$ for some  $p\in \N$. Thus, in terms of our analysis, we will work with the general decomposition in \eqref{elT}, but we will also justify in the framework of asymptotic expansions how a \emph{specific} choice of $k,l$ and $\alpha_M,\alpha_V,\beta_M,\beta_V$ in \eqref{elT}  is capable of dealing with localized eigenmodes, see Section~\ref{intermezzo}.

As in \cite{Blank} and \cite{Garcke}, we extend the definition \eqref{elT} to the whole hyperplane $\Sigma^N$ by introducing a cut-off function for a small parameter $\omega>0$. We define
\begin{align}\label{cutoff}
    \sigma_{\omega}:\mathbb{R}\to \mathbb{R},\quad
    s\mapsto
    \begin{cases}
        -\omega \quad&\text{if\;}s\le -\omega,\\
        a_{\omega}\quad&\text{if\;}-\omega<s<0,\\
        s\quad&\text{if\;}0\le s\le 1,\\
        b_{\omega}\quad&\text{if\;} 1<s<1+\omega,\\
        1+\omega \quad&\text{if\;} s\ge 1+\omega,
    \end{cases}
\end{align}
where $a_{\omega}$ and $b_{\omega}$ are monotonically increasing $C^{1,1}$ functions that are constructed in such a way that $\sigma_{\omega}$ is also a $C^{1,1}$ function. Then we consider the extensions
\begin{align}\label{exrho}
\begin{aligned}
&\overline{\rho}:\R^N\to \mathbb{R},\quad
\B{\varphi}\mapsto \sum_{i=1}^{N-1}\rho^{i}\big(\sigma_{\omega}\circ\beta_M\big)([P_{\Sigma}(\bphi)]^i),
\\
&\rho:\R^N\to \mathbb{R},\quad
\B{\varphi}\mapsto \overline{\rho}(\B{\varphi})+ 
\tilde{\rho}^N\varepsilon^l \big(\sigma_{\omega}\circ\beta_V\big)([P_{\Sigma}(\bphi)]^N),
\end{aligned}
\end{align}
where 
\begin{align*}
    P_{\Sigma}: \R^N\to \Sigma^N,\quad
    \B{\varphi}\mapsto \underset{\B{v}\in \Sigma^N}{\arg\min}\frac12\norm{\B{\varphi}-\B{v}}_{\ell^2}
\end{align*}
denotes the $\ell^2$ projection of $\mathbb{R}^N$ onto the convex set $\Sigma^N$. Note that in order for \ref{ASS:C:2} to be satisfied, the demanded positivity of the interpolation functions in \eqref{eq:interpol} is in general not enough, as we allow $\alpha_V,\beta_V$ to become negative outside the unit interval. However, the special choice 
\begin{align}\label{eq:num_choice}
    \begin{aligned}
    \beta_M(\varphi^i)&=(\varphi^i)^2,\\
    \beta_V(\varphi^N)&=-(\varphi^N-1)^2+1,
    \end{aligned}
\end{align}
for $\varphi\in\Sigma^N$, which will be used in the numerical simulations, see also \eqref{eq:num_quad}, satisfies \ref{ASS:C:2} due to the fact that the $\ell^1$ and $\ell^2$ norms are equivalent on $\R^N$. In our analysis we will stick to the general decomposition \eqref{elT} for full generality.
The tensor $\mathbb{C}$ is dealt with analogously.

To conclude this subsection, let us introduce some further notation. For $\bphi\in \LuN$, we define a weighted scalar product on $L^2(\Omega;\R^d)$ by 
\begin{align*}
        (\B{f},\B{g})_{\rho(\bphi)}
    \coloneqq 
        \int_{\Omega}\rho(\bphi)\B{f}\cdot \B{g}\mathrm{\,d}x\quad \text{for all } \B{f},\B{g}\in L^2(\Omega;\R^d),
\end{align*}
and a weighted scalar product on $H^1_D(\Omega;\R^d)$ by
\begin{align*}
    \langle\mathcal{E}(\B{u}),\mathcal{E}(\B{v})\rangle_{\mathbb{C}(\bphi)}\coloneqq 
    \int_{\Omega}\mathbb{C}(\bphi)\mathcal{E}(\B{u}):\mathcal{E}(\B{v})\mathrm{\,d}x\quad \text{for all }\B{u},\B{v}\in H^1_D(\Omega;\R^d).
\end{align*}
In the following, we write $L^2_{\bphi}(\Omega;\R^d)$ in order to emphasize the fact that we equip $L^2(\Omega;\R^d)$ with the scalar product $(\cdot,\cdot)_{\rho(\bphi)}$.
\subsection{The state equation}\label{SUB:Se}
We now introduce the system of equations describing the elastic structure, which will be referred to as the \textit{state equation}. It reads as
\begin{align}\tag{$SE^\eps$}\label{state}
    \begin{cases}
        \begin{array}{rll}
                -\nabla\cdot\left[\mathbb{C}(\B{\varphi})\mathcal{E}(\B{w}^{\eps,\B{\varphi}})\right]
            &=
                \lambda^{\eps,\B{\varphi}}
                \rho(\B{\varphi})\B{w}^{\eps,\B{\varphi}}
            &\quad
                \text{in }\Omega,\\
                \B{w}^{\eps,\B{\varphi}}
            &=
                \B{0}
            &\quad
                \text{on }\Gamma_D,\\
                \left[
                    \mathbb{C}(\B{\varphi})\mathcal{E}(\B{w}^{\eps,\B{\varphi}})
                \right]\B{n}
            &=
                \B{0}
            &\quad
                \text{on }\Gamma_0,
        \end{array}
    \end{cases}
\end{align}
and its weak formulation is given by
\begin{align}\label{WWP}
    \BE{\B{w}^{\eps,\bphi}}{\B{\eta}}=\lambda^{\eps,\B{\varphi}}\rsp{\B{w}^{\eps,\bphi}}{\B{\eta}}{\B{\varphi}}
\end{align}
for all $\B{\eta}\in \Hd$. 
In \cite{Garcke}, using classical spectral theory, it was shown that for any $\bphi\in L^\infty(\Omega,\R^N)$, there exists a sequence of eigenvalues (with multiple eigenvalues being repeated according to their multiplicity) which can be ordered as
\begin{align}\label{lamSeq}
0<\lambda_1^{\eps,\B{\varphi}}
\le\lambda_2^{\eps,\B{\varphi}}
\le\lambda_3^{\eps,\B{\varphi}}
\le\cdots \to \infty.
\end{align}
This comprises all eigenvalues of \eqref{WWP}.
Moreover, the corresponding eigenfunctions $$\{\B{w}_1^{\eps,\bphi},\B{w}_2^{\eps,\bphi},...\}\subset H^1_D(\Omega;\R^d)$$ can be chosen as an orthonormal basis of $L^2_\bphi(\Omega;\R^d)$, meaning that
\begin{align}\label{EveNorm}
   (\B{w}_i,\B{w}_j)_{\rho(\bphi)} = \int_{\Omega}\rho(\bphi)\, \B{w}_i\cdot \B{w}_j \mathrm{\,d}x= \delta_{ij}
\end{align}
for all $i,j\in\N$.
This property will be crucial when considering the formal asymptotics of the eigenfunctions.
In the following, when we talk about eigenvalues and eigenfunctions, we will always refer to the pairs $(\lambda_i^{\eps,\bphi},\bw_i^{\eps,\bphi})$ with $i\in \N$, which have the aforementioned properties.

\subsection{The optimization problem and the gradient inequality}
Finally, we are in a position to state the optimization problem 
\begin{align}\tag{$\mathcal{P}^{\varepsilon}_{l}$}\label{Pepsla}
    \left\{
        \begin{array}{ll}
            \min&	
                J^{\varepsilon}_l(\B{\varphi}),\\
            \text{over}&
                \B{\varphi}\in \mathcal{\B{\mathcal{G}}}^{\B{m}},\\
            \text{s.t.}& 
                \lambda^{\eps,\B{\varphi}}_{n_1},\dots, \lambda^{\eps,\B{\varphi}}_{n_l}
                \text{ are eigenvalues of } \eqref{WWP},
           \end{array}
    \right.
\end{align}
with
\begin{align*}
        J_{l}^{\varepsilon}(\B{\varphi})
    \coloneqq
        \Psi(\lambda_{n_1}^{\eps,\B{\varphi}},\dots, \lambda_{n_l}^{\eps,\B{\varphi}})
        +
        \gamma E^{\varepsilon}(\B{\varphi}),
\end{align*}
for some $l\in \mathbb{N}$, where $n_1,\dots,n_l\in \mathbb{N}$ indicate a selection of eigenvalues.
Here, $\gamma>0$ is a fixed constant related to surface tension.

\begin{Rem}
    It is worth mentioning that we do not need any boundedness assumption on $\Psi$ in order to prove the existence of a minimizer to \eqref{Pepsla} in the same way as in \cite[Theorem 6.1]{Garcke}. In analogy to \cite[Lemma 3.7]{GarHKK_OptLapEV}, one can show that there are constants $C_{1,\eps},\,C_{2,\eps} > 0$ depending only on the choice of $\C_{\eps}$ and $\rho_{\eps}$ such that
    \begin{align*}
        C_{1,\eps}\lambda_k^M\le \lambda_k^{\eps,\bphi}\le C_{2,\eps} \lambda_k^M
        \quad\text{for all $\bphi\in \B{\mathcal{G}}$.}
    \end{align*}
    Here, $\lambda_k^M$ denotes the $k$-th eigenvalue of the problem \eqref{WWP} with $\mathbb{C}\equiv \mathrm{Id}$ and $\rho\equiv 1$.
    Qualitatively speaking, $\lambda_k^M$ denotes an eigenvalue in the situation when the whole design domain is occupied by one material.
\end{Rem}

In \cite[Theorem 6.2]{Garcke}, the following first-order necessary optimality conditions was derived.
\begin{Thm}
Let $\bphi\in \mathcal{\B{\mathcal{G}}}^{\B{m}}$ be a local minimizer of \eqref{Pepsla}, i.e., there exists $\delta>0$ such that $J_{l}^{\varepsilon}(\B{\varphi}) \le J_{l}^{\varepsilon}(\B{\zeta})$ for all $\B{\zeta} \in \mathcal{\B{\mathcal{G}}}^{\B{m}}$ with $\norm{\B{\zeta}-\bphi}_{H^1(\Omega;\R^N)\cap L^\infty(\Omega;\R^N)}<\delta$.
We further assume that the eigenvalues $\lambda_{n_1}^{\eps,\bphi},\dots,\lambda_{n_l}^{\eps,\bphi}$ are simple.
Then the gradient inequality 
\begin{alignat}{2}\tag{$GI^\eps$}\label{GIp}
    \begin{aligned}
        &\sum_{r=1}^{l}
        \Bigg\{
        [\partial_{\lambda_{n_r}}\hspace{-0.7ex}\Psi]
        \big(
        \lambda^{\eps,\B{\varphi}}_{n_1},\dots,\lambda^{\eps,\B{\varphi}}_{n_l}
        \big)
        \\
        &\qquad\cdot
        \Big(
        \langle
        \E(\B{w}^{\eps,\B{\varphi}}_{n_r}):\E(\B{w}^{\eps,\B{\varphi}}_{n_r})
        \rangle_{\C^\prime(\bphi)(\tilde{\B{\varphi}}-\bphi)}
        -\lambda^{\eps,\B{\varphi}}_{n_r}
        \int_{\Omega}
        \rho^{\prime}(\B{\varphi})\big(\tilde{\B{\varphi}}-\B{\varphi}\big)
        \big|\B{w}^{\eps,\B{\varphi}}_{n_r}\big|^2
        \textup{\,d}x
        \Big)
        \Bigg\}\\
        &\;\;\;+
        \gamma\varepsilon
        \int_{\Omega}
        \nabla\B{\varphi}:\nabla(\tilde{\B{\varphi}}-\B{\varphi})
        \textup{\,d}x
        +\frac{\gamma}{\varepsilon}
        \int_{\Omega}
        \psi_0^{\prime}(\B{\varphi})(\tilde{\B{\varphi}}-\B{\varphi})
        \textup{\,d}x
        \;\ge\; 0
    \end{aligned}
\end{alignat}
holds for all $\tilde{\B{\varphi}}\in \B{\mathcal{G}}^{\B{m}}$.
\end{Thm}

The upcoming sharp-interface analysis will be concerned with passing to the limit in the state equation \eqref{state} as well as in the gradient inequality \eqref{GIp}.

%%%%%%%%%%%%%%%%%%%%%%%%%%%%%%%%%%%
%%Analysis of the gradient inequality
%%%%%%%%%%%%%%%%%%%%%%%%%%%%%%%%%%%

\section{Analysis of the gradient inequality}\label{SEC:Ana}
In this section, we will show under a suitable regularity assumption on the eigenfunctions involved in \eqref{GIp} that there exists a solution of the above gradient inequality possessing even the regularity $\B{\varphi}\in H^2(\Omega; \mathbb{R}^N)$.
This will be carried out by applying a regularization process to the non-smooth potential $\psi$, which was employed in a similar fashion in \cite{Blowey,BloweyEll, Elliott, Sarbu}.
Our approach mainly follows the ideas of \cite{Sarbu}.

We regularize the gradient inequality in order to deal with the indicator functional $I_{\B{G}}$ contained in the definition of the potential $\psi$. This will yield a sequence of $H^2$-regular approximating phase-fields $(\pd)_{\delta>0}$ solving regularized equations and converging to the desired phase-field $\B{\varphi}$. Another convenient aspect of this procedure is that it will generate Lagrange multipliers that will allow us to transform the gradient inequality into an equality. This strong formulation of \eqref{GIp} will be the starting point for our asymptotic analysis in Section~\ref{SEC:Leading}.

\subsection{Regularization of the potential \texorpdfstring{$\psi$}{x} and rewriting the constraints}
We notice that $\B{\varphi}\in \B{\mathcal{G}}^{\B{m}}$ needs to satisfy the constraint
\begin{align*}
    \varphi^i(x)\ge 0
\end{align*}
for almost every $x\in \Omega$ and $i=1,\dots,N$. To deal with this constraint we regularize the potential appearing in the Ginzburg-Landau energy which was initially given as
\begin{align*}
    \psi(\bphi)=\psi_0(\bphi)+I_{\B{G}}(\bphi).
\end{align*}
\begin{Def}\label{DEF:Reg}
    For $\delta>0$ we define the regularized potential
    \begin{align}\label{defreg}
        \psi_{\delta}:\mathbb{R}^N\to\mathbb{R},\quad
        \psi_{\delta}(\bphi)=\psi_{0}(\bphi)+\frac{1}{\delta}\hat{\psi}(\bphi),
    \end{align}
    where
    \begin{equation}\label{defHat}
        \hat{\psi}(\bphi)\coloneqq\sum_{i=1}^{N}\big(\min(\varphi^i,0)\big)^2.
    \end{equation}
\end{Def}
\begin{Rem}
 We see that the regularization now approximates the indicator functional $I_{\mathbb{R}^N_{+}}$ by the function $\frac{1}{\delta}\hat{\psi}$.
 For $\delta \searrow 0$, exactly the negative parts of the components of $\B{\varphi}$ are penalized. 
\end{Rem}
 To deal with the remaining constraints hidden in $\B{\mathcal{G}}^{\B{m}}$ namely the integral constraint $\fint_{\Omega}\B{\varphi}\mathrm{\,d}x=\B{m}$ and the sum constraint $\sum_{i=1}^{N}\varphi^i=1$ a.e.~in $\Omega$, we introduce linear orthogonal projections.
\begin{Def}\label{DEF:Proj}
Let us define the linear orthogonal projections
\begin{align}\label{defoint}
    \begin{aligned}
        P_{\int}: L^2(\Omega;\mathbb{R}^N)&\to L^2_0(\Omega;\mathbb{R}^N),\\
        \B{u}&\mapsto \B{u}-\fint_{\Omega}\B{u}\mathrm{\,d}x
    \end{aligned}    
\end{align}
with $L^2_0(\Omega;\mathbb{R}^N)\coloneqq \big\{\B{u}\in L^2(\Omega;\mathbb{R}^N)\,\big|\, \int_{\Omega}\B{u}\mathrm{\,d}x=\B{0}\big\}$
and
\begin{align}\label{defosum}
    \begin{aligned}
        P_{T\Sigma}: L^2(\Omega;\mathbb{R}^N)&\to L^2_{T\Sigma}(\Omega;\mathbb{R}^N),\\
        \B{u}&\mapsto \B{u}-\left(\frac{1}{N}\sum_{i=1}^{N}{u}^i\right)\B{1},
    \end{aligned}    
\end{align}
where $\B{1}=(1,\dots,1)^T\in \mathbb{R}^N$ and 
\begin{align*}
L^2_{T\Sigma}(\Omega;\mathbb{R}^N)\coloneqq \left\{\B{u}\in L^2(\Omega;\mathbb{R}^N)\;\left|\; \sum_{i=1}^{N}u^i=0 \text{ a.e.~in } \Omega\right.\right\}.
\end{align*}
To simplify the notation, we further define the composition $P\coloneqq P_{T\Sigma}\circ P_{\int}=P_{\int}\circ P_{T\Sigma}$.
\end{Def}

\begin{Rem}
    Note that for the constraint $\varphi(\bx)\in \R^N_{+}$, we cannot introduce a linear orthogonal projection as there is no vector space corresponding to this constraint. Thus, the approximation of the indicator function in Definition~\ref{DEF:Reg} is actually necessary.
\end{Rem}

\subsection{Smoothness assumption and rewriting the gradient inequality}\label{Sub:Ext}
In order to obtain a suitable regularization of the gradient inequality, we need to find a way to test \eqref{GIp} with arbitrary functions in $\HN$ and not only in $\HL$. For this reason and to obtain higher regularity of the phase-field, we will need to assume higher regularity of the corresponding eigenfunctions.

We fix a parameter $\eps>0$ as well as a solution $\bphi^\eps\in\B{\mathcal{G}}^{\B{m}}\subset \HL$ of \eqref{GIp}.
For a cleaner presentation we omit the superscript $\eps$ in the eigenvalues and eigenfunctions.

A priori, the term
\begin{equation*}
        \big\langle
              \mathcal{E}\big(\B{w}_{n_r}^{\bphi}\big),
              \mathcal{E}\big(\B{w}_{n_r}^{\bphi}\big)
         \big\rangle_{\mathbb{C}^{\prime}(\bphi)\B{\eta}}
     =
        \int_{\Omega}
            \big[\mathbb{C}^{\prime}(\bphi)\B{\eta}\big]
            \mathcal{E}\big(\B{w}_{n_r}^{\bphi}\big):
            \mathcal{E}\big(\B{w}_{n_r}^{\bphi}\big)
        \mathrm{\,d}x
\end{equation*}
is well defined only for $\B{\eta}\in L^{\infty}(\Omega;\mathbb{R}^N)$ as the expression $\mathcal{E}(\bw_{n_r}):\mathcal{E}(\bw_{n_r})$
merely belongs to $L^1(\Omega)$. However, in order to consider a suitable regularized problem associated to \eqref{GIp}, we need this term to be an element in $L^2(\Omega)$. For this purpose, we require the regularity $\B{w}_{n_r}\in W^{1,4}(\Omega)$.

Therefore, we now make the following crucial regularity assumption which shall hold for the rest of this paper.
\begin{enumerate}[label = (R), leftmargin=*]
 \item\label{ASS:RegEig} For $r=1,\dots,l$, let the eigenfunctions $\B{w}_{n_r}$ involved in \eqref{GIp} belong to $W^{1,4}(\Omega;\mathbb{R}^d)$.
\end{enumerate}
\begin{Rem}
    Note that there exists a regularity theory for the equations of linear and nonlinear elasticity, see, e.g. \cite{Shi,Herzog}. However, due to the fact that the coefficient $\mathbb{C}(\bphi)$ is merely essentially bounded, we could only prove the existence of an (in general arbitrarily small) parameter $\iota>0$ such that
    \begin{equation}\label{regEF}
        \mathcal{E}(\bw_{n_r})\in L^{2+\iota}(\Omega).
    \end{equation}
    Note that there exist counterexamples going back to De Giorgi for linear systems of elliptic PDEs (see, e.g., \cite[Section 4.1]{Beck}) providing \textit{unbounded} solutions $\B{u}\in W^{1,2}(B;\R^d)$ for $d\ge 3$ to a system of the form
    \begin{align*}
        \text{div}(\mathcal{A}(x)D\B{u}(x))=\B{0} \textup{ in } B\subset \mathbb{R}^d,
    \end{align*}
    where $\mathcal{A}$ is bounded and coercive and $B$ denotes the unit ball.
    In particular, in the physically relevant case $d=3$ where $W^{1,4}(\Omega;\mathbb{R}^d)\hookrightarrow
    C^0(\overline{\Omega};\mathbb{R}^d)$, the condition $\bw_{n_r}\in W^{1,4}(\Omega;\mathbb{R}^d)$ seems to be a real assumption as unbounded eigenfunctions might exist.
\end{Rem}

In the following, let $(\cdot,\cdot)$ denote the classical scalar product on $L^2(\Omega;\mathbb{R}^N)$. 
Recalling 
\begin{align*}
    \mathbb{C}^{\prime}(\bphi)\B{\eta}=
    \left(
        \sum_{m=1}^{N}\partial_m\mathbb{C}_{ijkl}(\bphi)\eta^m
    \right)_{i,j,k,l=1}^{d}
\end{align*}
for $\B{\eta}\in \LN{2}$,
we have
\begin{align*}
    \langle
        \mathcal{E}\big(\B{w}_{n_r}^{\bphi}\big),
        \mathcal{E}\big(\B{w}_{n_r}^{\bphi}\big)
    \rangle_{\mathbb{C}^{\prime}(\bphi)\B{\eta}}
    &=
    \int_{\Omega}
    \Big(
      \sum_{m=1}^{N}
        [\partial_m\mathbb{C}(\bphi)]\eta^m
    \Big)
        \mathcal{E}\big(\B{w}_{n_r}^{\bphi}\big):
        \mathcal{E}\big(\B{w}_{n_r}^{\bphi}\big)
        \mathrm{\,d}x\\     
    &=
     \int_{\Omega}
      \sum_{m=1}^{N}
     \Big(
        [\partial_m\mathbb{C}(\bphi)]
        \mathcal{E}\big(\B{w}_{n_r}^{\bphi}\big):
        \mathcal{E}\big(\B{w}_{n_r}^{\bphi}\big)
     \Big)\eta^m
        \mathrm{\,d}x\\ 
    &=
     \int_{\Omega}
      \sum_{m=1}^{N}
     \Big[\Big(
        \mathbb{C}^{\prime}(\bphi)
        \mathcal{E}\big(\B{w}_{n_r}^{\bphi}\big):
        \mathcal{E}\big(\B{w}_{n_r}^{\bphi}\big)
     \Big)\Big]_m\eta^m
        \mathrm{\,d}x
    \\[1ex]
    &=\left(
        \mathbb{C}^{\prime}(\bphi)
        \mathcal{E}\big(\B{w}_{n_r}^{\bphi}\big):
        \mathcal{E}\big(\B{w}_{n_r}^{\bphi}\big),
        \B{\eta}
        \right).
\end{align*}
Note that the term in the last line is to be understood as
\begin{align*}
    \mathbb{C}^{\prime}(\bphi)
    \mathcal{E}\big(\B{w}_{n_r}^{\bphi}\big):
    \mathcal{E}\big(\B{w}_{n_r}^{\bphi}\big)
    =
    \Big([\partial_m\mathbb{C}(\bphi)]
    \mathcal{E}\big(\B{w}_{n_r}^{\bphi}\big):
    \mathcal{E}\big(\B{w}_{n_r}^{\bphi}\big)\Big)_{m=1}^{N}    
    \in \LN{2}.
\end{align*}
Thus, the projection of this term is well defined and the $L^2$ regularity of this object is ensured by the assumptions \ref{ASS:RegEig} and \ref{ASS:B:1}. For later purposes, we point out that a straightforward computation reveals
\begin{align*}
    P_{T\Sigma}\left[\mathbb{C}^{\prime}(\bphi)
    \mathcal{E}\big(\B{w}_{n_r}^{\bphi}\big):
    \mathcal{E}\big(\B{w}_{n_r}^{\bphi}\big)\right]
    =
    \left[\left(P_{T\Sigma}\left[\mathbb{C}_{ijkl}^{\prime}(\bphi)\right]\right)_{i,j,k,l=1}^{d}\right]
    \mathcal{E}\big(\B{w}_{n_r}^{\bphi}\big):
    \mathcal{E}\big(\B{w}_{n_r}^{\bphi}\big),
\end{align*}
where
\begin{align*}
    \mathbb{C}_{ijkl}^{\prime}(\bphi)=
    \left(\partial_m\C_{ijkl}\right)_{m=1}^N\in 
    L^2(\Omega;\R^N).
\end{align*}
To have a more concise notation, we will write
\begin{align*}
    \langle \mathcal{E}\big(\B{w}_{n_r}^{\bphi}\big):
    \mathcal{E}\big(\B{w}_{n_r}^{\bphi}\big)\rangle_{P_{T\Sigma}\left[\mathbb{C}^{\prime}(\bphi)\right]}
    :=
    P_{T\Sigma}\left[\mathbb{C}^{\prime}(\bphi)
    \mathcal{E}\big(\B{w}_{n_r}^{\bphi}\big):
    \mathcal{E}\big(\B{w}_{n_r}^{\bphi}\big)\right].
\end{align*}
Analogously, we use the notation
\begin{align*}
    \left(
        \B{w}_{n_r}^{\bphi},\B{w}_{n_r}^{\bphi}
    \right)_{\rho^{\prime}(\bphi)\B{\eta}} 
    =
    \left(
    \rho^{\prime}(\bphi)
       \B{w}_{n_r}^{\bphi}\cdot
        \B{w}_{n_r}^{\bphi},
    \B{\eta}    
    \right)
\end{align*}
for the density term.
To reformulate the gradient inequality \eqref{GIp}, we further define the function
\begin{align}
\label{DEF:FPHI}
\begin{aligned}
    \B{f}^{\B{\varphi}}\coloneqq&   
        -\sum_{r=1}^{l}\Big\{
                    [\partial_{\lambda_{i_j}}\hspace{-0.7ex}\Psi]\big(\lambda_{n_1}^{\bphi},\dots,\lambda_{n_l}^{\bphi}\big)
                    \Big(
                        \mathbb{C}^{\prime}(\bphi) \mathcal{E}\big(\B{w}_{n_r}^{\bphi}\big):\mathcal{E}\big(\B{w}_{n_r}^{\bphi}\big)
        \\
        &\qquad\qquad
                        -\lambda_{n_r}^{\bphi}
                           \rho^{\prime}(\bphi)   \B{w}_{n_r}^{\bphi}\cdot\B{w}_{n_r}^{\bphi}
                    \Big)
                \Big\}
        % \\
        % &\quad
        -\frac{\gamma}{\varepsilon}\psi_{0}^{\prime}(\bphi) \,.
\end{aligned}    
\end{align}
By means of assumption \ref{ASS:RegEig}, we infer $\B{f}^{\bphi}\in \LN{2}$. This is crucial for the subsequent analysis, especially for the absorption argument at the end of Lemma~\ref{lem:phi_apriori}, and therefore, assumption \ref{ASS:RegEig} cannot be waived. As $\bphi\in \B{\mathcal{G}}^{\B{m}}$ is fixed, we write $\B{f}=\B{f}^{\bphi}$ in the following.
Using this notation, we obtain:

\begin{Pro}

The gradient inequality \eqref{GIp} is equivalent to
    \begin{align}\label{GIvar}
            \gamma\eps\left(\nabla\bphi,\nabla(\tilde{\bphi}-\bphi)\right)_{L^2}
        \ge
            \left(\B{f},\tilde{\bphi}-\bphi\right)_{L^2}
        \quad \text{for all $\tilde{\bphi}\in \B{\mathcal{G}^{\B{m}}}$}.
    \end{align}
\end{Pro}

\subsection{The regularized problem and its limit}
Now that we have introduced the regularized potential and suitable orthogonal projections, and have made the necessary regularity assumption, we can formulate a regularized problem which will approximate our initially fixed solution $\bphi\in \B{\mathcal{G}}^{\B{m}}$ of \eqref{GIp} in order to provide the desired $H^2$-regularity of $\bphi$.

Using all the previously introduced notation, we are now in a position to state the so-called regularized problem.
\begin{Def}
Let 
\begin{align}
    \label{DEF:GMT}
    \tilde{\B{\mathcal{G}}}^{\B{m}}\coloneqq \left\{ \tilde{\bphi}\in \;\HN\left|\; \fint_{\Omega}\tilde{\bphi}=\B{m}\;\,\text{and}\;\;\sum_{i=1}^{N}\tilde{\varphi}^i = 1 \text{\,a.e.~in }\Omega\right.\right\}.
\end{align}
We say that $\pd\in \tilde{\B{\mathcal{G}}}^{\B{m}}$ is a solution to the regularized problem if it solves

\begin{align*}\tag{$RE$}\label{reg}
        \gamma\varepsilon \left(\nabla\pd,\nabla \B{\eta}\right)
        +\frac{\gamma}{\delta\varepsilon}\left(P[\hat{\B{\phi}}(\pd)],\B{\eta}\right)=
        \left(P\B{f},\B{\eta}\right)\quad \text{for all $\B{\eta}\in \HN$}.
\end{align*}
\end{Def}

Before proving the existence of a solution to \eqref{reg}, we recall some properties proven in \cite{Sarbu} that will be important for the upcoming analysis.

\pagebreak[2]

\begin{Pro}\label{propsi}
Let $\hat{\psi}$ be as defined in \eqref{defHat}. Then the following properties hold true.

\begin{enumerate}[label = \textnormal{(\alph*)}, leftmargin=*]
        \item The weak derivative fulfills
          \begin{align}\label{prepsid}
             \nabla\hat{\psi}=\hat{\B{\phi}}
          \end{align}
          where, for $\B{\xi}\in \mathbb{R}^N$,
         $\hat{\phi}^i(\B{\xi})\coloneqq\hat{\phi}^i({\xi}^i)\coloneqq2\left[\xi^i\right]_{-}$ with $[{s}]_{-}\coloneqq \min(s,0)$ for all $s\in \mathbb{R}$.
         \item \textbf{Monotonicity:} $\hat{\B{\phi}}$ is non-decreasing in each component, i.e.,
         \begin{align}\label{phmon}
            0\le \left(\hat{\phi}^i(r)-\hat{\phi}^i(s)\right)(r-s)
         \end{align}
         for all $r,s\in\mathbb{R}$ and $i=1,\dots,N$.\\
         \item \textbf{Convexity:} $\hat{\psi}$ is convex, i.e.,
        \begin{align}\label{gpsip}
            (\B{\xi}-\B{\eta})\cdot \hat{\B{\phi}}(\B{\eta})
            \le \hat{\psi}(\B{\xi})-\hat{\psi}(\B{\eta})
        \end{align}
        for all $\B{\xi},\B{\eta}\in\mathbb{R}^N$.
\end{enumerate}
\end{Pro}
Using these results, we now prove the well-posedness result for the regularized problem. 
In order to show $H^2$-regularity of the solution $\pd$, we need the following regularity assumption on the design domain which shall hold for the rest of the paper.

\begin{enumerate}[label = (D), leftmargin=*]
\item \label{ass:Omega} In addition to \ref{ASS:A:1}, we assume that $\Omega$ has at least one of the following properties:
\begin{enumerate}[label = (\roman*)]
    \item The boundary $\partial\Omega$ is of class $C^{1,1}$.
    \item $\Omega$ is convex.
\end{enumerate}
\end{enumerate}

The well-posedness result for the regularized problem \eqref{reg} reads as follows.

\begin{Lem}\label{Exreg}
    For $\delta>0$ there exists a unique solution $\pd \in \tilde{\B{\mathcal{G}}}^{\B{m}}\subset \HN$ of \eqref{reg}. 
    The solution possesses the regularity $\pd \in H^2(\Omega;\mathbb{R}^N)$ and it holds
    \begin{alignat}{3}\label{RegPw}\tag{$PRE$}
    \begin{aligned}
        -\Delta \pd&=
        -\frac{1}{\delta\varepsilon^2}P[\hat{\B{\phi}}(\pd)]
        +\frac{1}{\gamma\varepsilon}P\B{f}\quad &&\text{a.e.~in }\Omega\\
        \nabla\pd \, \B{n}&=\B{0}\quad &&\text{a.e.~on\;}\partial\Omega,
    \end{aligned}
    \end{alignat}
\end{Lem}
\begin{proof}
First of all, we want to show that there exists at most one solution to \eqref{reg}. To this end, we assume that there are two solutions $\bphi_{\delta,1},\bphi_{\delta,2}\in\tilde{\B{\mathcal{G}}}^{\B{m}}$. Then, by subtracting the corresponding equations, we obtain
\begin{align*}
    \gamma\varepsilon \left(\nabla[\bphi_{\delta,1}-\bphi_{\delta,2}],\nabla \B{\eta}\right)+
    \frac{\gamma}{\varepsilon\delta}\left(P[\hat{{\phi}}(\bphi_{\delta,1})-\hat{{\phi}}(\bphi_{\delta,2})],\B{\eta}\right)=0
\end{align*}
for all $\B{\eta}\in \HN$. Testing with $\bphi_{\delta,1}-\bphi_{\delta,2}\in \LzS\cap L_0^2(\Omega;\mathbb{R}^N)$, we can drop the projection $P$ in the second term. Using the monotonicity property \eqref{phmon}, we infer
\begin{align*}
    \gamma\varepsilon \left(\nabla[\bphi_{\delta,1}-\bphi_{\delta,2}],\nabla [\bphi_{\delta,1}-\bphi_{\delta,2}]\right)\le 0.
\end{align*}
This yields $\bphi_{\delta,1}=\bphi_{\delta,2}$ as these functions have identical mean value.

In order to prove the existence of a solution, we consider a suitable minimization problem. 
Therefore, we define the functional
\begin{alignat}{1}
\begin{aligned}\label{defId}
    I_{\delta}(\B{\xi})\coloneqq& 
    \frac{\gamma\varepsilon}{2}\int_{\Omega}\abs{\nabla \B{\xi}}^2\mathrm{\,d}x
    +\frac{\gamma}{\varepsilon\delta}\int_{\Omega}{\hat{\psi}}(\B{\xi})\mathrm{\,d}x
    -\int_{\Omega}\B{f}\cdot\B{\xi}\mathrm{\,d}x
\end{aligned}
\end{alignat}   
for all $\B{\xi}\in \HN$.
If we can now show that there exists a $\pd\in \tilde{\B{\mathcal{G}}}^{\B{m}}$
that solves the minimization problem
\begin{align}\label{mind}
\underset{\B{\xi}\in \tG}{\min}I_{\delta}(\B{\xi}),
\end{align}
the existence result is proven since then the G\^{a}teaux derivative of $I_{\delta}'(\pd)$, which is given by
\begin{alignat}{2}\label{RegMin}
        I_{\delta}^{\prime}(\pd)\B{\eta}
    &=&&               
    \gamma\varepsilon \left(\nabla\pd,\nabla \B{\eta}\right)
    +\frac{\gamma}{\varepsilon\delta}\left(\hat{\B{\phi}}(\pd),\B{\eta}\right)
    -(\B{f},\B{\eta}),
\end{alignat}
for all directions $\B{\eta}\in \HN\cap L^2_{T\Sigma}(\Omega;\mathbb{R}^N)\cap L^2_{0}(\Omega;\mathbb{R}^N)$, vanishes. By applying the projections $P_{T\Sigma}$ and $P_{\int}$ to any $\B{\eta}\in \HN$ and then switching them to the other component in the $L^2$ scalar product, it follows that solving \eqref{mind} is equivalent to solving \eqref{reg}.

Note that there is no need to project the gradient term. This is justified as follows. 
By construction, we have
\begin{align*}
    P_{T\Sigma}\B{\eta}=\B{\eta}-\left[\frac{1}{N}\sum_{k=1}^N\eta^k\right]\B{1}.
\end{align*}
On the other hand, we compute
\begin{align*}
    \nabla \left[\sum_{k=1}^N\eta^k\B{1}\right]
    =\left(\left[\sum_{k=1}^N\partial_1\eta^k\right]\B{1},\dots,\left[\sum_{k=1}^N\partial_d\eta^k\right]\B{1}\right),
\end{align*}
and therefore, the entries in each column are identical. Now, we compute
\begin{align}\label{gradvan}
    \nabla\pd :\nabla \left[\sum_{k=1}^N\eta^k\B{1}\right]
    =\sum_{i=1}^N \left\{\left[\sum_{k=1}^N\partial_i\eta^k\right]\sum_{j=1}^N\partial_i\varphi_{\delta}^j\right\}.
\end{align}
We see that this term vanishes as by construction, as $\sum_{i=1}^{N}\varphi_{\delta}^i=1$ a.e.~in $\Omega$ because $\pd\in \tilde{\B{\mathcal{G}}}^{\B{m}}$. In other words, $\partial_i \pd\in L^2_{T\Sigma}(\Omega;\mathbb{R}^N)$.

 As the gradient term is invariant under addition of constants we can also omit the projection $P_{\int}$.

It remains to show that there exists a minimizer of \eqref{mind}.
By construction, $\hat{\psi}\ge 0$.
Furthermore, using Young's inequality, we find a constant $C>0$ such that
\begin{align}\label{YAbs}
    \int_{\Omega}\abs{\nabla \B{\xi}}^2\mathrm{\,d}x+\int_{\Omega}\B{f}\cdot \B{\xi}\mathrm{\,d}x\ge -C 
    \quad\text{for all $\B{\xi}\in\tilde{\B{\mathcal{G}}}^{\B{m}}$}.
\end{align}
This is obtained by absorbing the quantity $\norm{\B{\xi}}_{L^2}^2$ by the term $\norm{\nabla \B{\xi}}_{L^2}^2$ which controls the whole $\HN$-norm as all $\B{\xi}\in\tilde{\B{\mathcal{G}}}^{\B{m}}$ have a fixed mean value.
Hence, $I_{\delta}$ is bounded from below on $\tilde{\B{\mathcal{G}}}^{\B{m}}$ and thus,
the infimum exists. Consequently, we find a minimizing sequence $(\bphi_{\delta,k})_{k\in \mathbb{N}}\subset \tG$ such that
\begin{align*}
    \underset{k\to\infty}{\lim}I_{\delta}(\bphi_{\delta,k})=\underset{\bphi\in \tG}{\inf}I_{\delta}(\bphi).
\end{align*}
In particular, $\|\bphi_{\delta,k}\|_{H^1(\Omega;\R^N)}$ remains bounded and thus,
there exists $\pd\in H^1(\Omega;\mathbb{R}^N)$ and such that the convergences
\begin{alignat*}{2}
    \bphi_{\delta,k}&\rightharpoonup \pd &&\quad\text{in\;}\HN,\\
    %\bphi_{\delta,k}&\rightharpoonup \pd &&\quad\text{in\;}L^{p'}(\Omega;\mathbb{R}^N),\\
    \bphi_{\delta,k}&\to\pd &&\quad \text{in\;}L^2(\Omega;\mathbb{R}^N),\\
    \bphi_{\delta,k}&\to\pd &&\quad \text{a.e.~in }\Omega
\end{alignat*}
hold along a non-relabeled subsequence.
From these convergences, we deduce $\pd\in\tilde{\B{\mathcal{G}}}^{\B{m}}$.
Using the estimate $\hat{\psi}(\bphi_{\delta,k})\le |\bphi_{\delta,k}|^2$ a.e.~in $\Omega$ along with the generalized majorized convergence theorem of Lebesgue (see \cite[p.~1015]{Zeidler}), and further employing the weak lower semincontinuity of norms, we conclude
\begin{align*}
    I_{\delta}(\pd)\le \underset{k\to\infty}{\liminf}\,I_{\delta}(\bphi_{\delta,k})
    =\underset{\bphi\in \tG}{\inf}I_{\delta}(\bphi).
\end{align*}
This implies that $\pd\in \tilde{\B{\mathcal{G}}}^{\B{m}}$ is a minimizer.

Now that we have shown the existence of a solution $\pd\in \HN$ to \eqref{reg}, it remains to prove that it possesses the desired regularity $H^2(\Omega;\mathbb{R}^N)$.
Since $\pd$ is a weak solution of \eqref{reg}, it can be interpreted as a weak solution of 
\begin{alignat}{2}
\left\{
\begin{aligned}\label{strongreg}
    -\Delta \pd&=\B{\mathcal{F}}\quad &&\text{in\;}\Omega,\\
    \nabla\pd \, \B{n}&=\B{0}\quad &&\text{on\;}\partial\Omega,
\end{aligned}  
\right.
\end{alignat}
 with
\begin{align}\label{rhs}
\B{\mathcal{F}}=\B{\mathcal{F}}(\pd)=
-\frac{1}{\delta\varepsilon^2}P[\hat{\B{\phi}}(\pd)]
%-p'\abs{\pd}^{p'-2}P_{T\Sigma}P_{\int}\pd
+\frac{1}{\gamma\varepsilon}P\B{f}
    \in L^2(\Omega;\mathbb{R}^N).
\end{align}
Due to assumption \ref{ass:Omega}, elliptic regularity theory (see e.g., \cite[Theorem~2.4.2.7]{Grisvard} in the case \ref{ass:Omega}(i) or \cite[Theorem~3.2.3.1]{Grisvard} in the case \ref{ass:Omega}(ii), respectively) yields $\pd\in H^2(\Omega;\mathbb{R}^N)$.
In particular, using this regularity, we conclude from \eqref{reg} that $\pd$ satisfies \eqref{RegPw}.
\end{proof}
As we want to pass to the limit in the regularized equation, we need some uniform bounds to apply classical compactness results.
\begin{Lem}\label{lem:phi_apriori}
Let $\pd\in H^2(\Omega;\mathbb{R}^N)$ be the solution of \eqref{reg}. Then there exist a constant $C>0$ such that
\begin{align}
    \bnorm{\pd}_{H^2(\Omega;\mathbb{R}^N)}&\le C,\label{boundH2}\\
    \bnorm{\big[\pd\big]_{-}}_{L^2(\Omega;\mathbb{R}^N)}&\le C\delta^{\frac{1}{2}},\label{pdmin}\\
    \frac{1}{\delta}\bnorm{P[\hat{\B{\phi}}(\pd)]}_{L^2(\Omega;\mathbb{R}^N)}&\le C\label{boundpsi},
\end{align}
for all $\delta>0$.
\end{Lem}
\begin{proof}
    By the previous lemma, we know that $\pd$ minimizes $I_{\delta}$ (see \eqref{defId}) over $\tilde{\B{\mathcal{G}}^{\B{m}}}$ (see \eqref{DEF:GMT}). Thus, we have
    \begin{align*}
        I_{\delta}(\pd)\le I_{\delta}(\B{\xi}), 
        \quad \text{for all $\B{\xi}\in \tilde{\B{\mathcal{G}}}^{\B{m}}$}.
    \end{align*}
    If we now choose any $\B{\xi}\in\B{\mathcal{G}}^{\B{m}}\subset \tilde{\B{\mathcal{G}}}^{\B{m}}$, we know that it is additionally componentwise non-negative and therefore $\hat{\psi}(\B{\xi})=0$ a.e.~in $\Omega$. In view of definition \eqref{defId}, this yields
    \begin{align}\label{boundpd}
    \begin{aligned}
            &\frac{\gamma\varepsilon}{2}\int_{\Omega}\babs{\nabla{\pd}}^2\mathrm{\,d}x
             +\frac{\gamma}{\varepsilon\delta}\int_{\Omega}{\hat{\psi}}({\pd})
         -\int_{\Omega}\B{f}\cdot{\pd}\mathrm{\,d}x\\
        &\quad \le
            \frac{\gamma\varepsilon}{2}\int_{\Omega}\abs{\nabla{\B{\xi}}}^2\mathrm{\,d}x
            -\int_{\Omega}\B{f}\cdot{\B{\xi}}\mathrm{\,d}x\\
        &\quad \le
            C,
    \end{aligned}        
    \end{align}
 where $C>0$ is a constant independent of $\delta$. Recalling the absorption trick \eqref{YAbs}, we obtain
 \begin{align}\label{pdH1}
    \bnorm{\pd}_{\HN}\le C,
 \end{align}
which will be needed in the end of the proof. 
Furthermore, using the definition of $\hat{\psi}$ (see \eqref{defHat}), we deduce that
 \begin{align*}
    \sum_{i=1}^{N}\bnorm{\big[\varphi_{\delta}^i\big]_{-}}_{L^2}^2\le C\delta,
 \end{align*}
 which directly leads to \eqref{pdmin}.

We notice that $\frac{1}{\delta}\hat{\B{\phi}}(\pd)$ is weakly differentiable (cf.~\cite[Lemma~7.6]{Gilbarg}) and belongs to $\HN$.
In order to prove \eqref{boundpsi}, we test \eqref{reg} with $\B{\eta}=\frac{1}{\delta}\hat{\B{\phi}}(\pd)$. 
We obtain
\begin{alignat}{2}\label{testphi}
    \begin{aligned}
        &\frac{\gamma}{\varepsilon\delta}\left(\nabla\pd,\nabla 
            \hat{\B{\phi}}(\pd)\right)
            +\frac{\gamma}{\delta^2\varepsilon}\int_{\Omega}\babs{P[\hat{\B{\phi}}(\pd)]}^2\mathrm{\,d}x\\
        &\quad=
                \frac{1}{\delta}\int_{\Omega}\B{f}\cdot 
                P[\hat{\B{\phi}}(\pd)]
                \mathrm{\,d}x.
    \end{aligned}
\end{alignat}
Applying \cite[Lemma 7.6]{Gilbarg} to $\hat{\B{\phi}}$, we further deduce
\begin{align*}
\frac{\gamma}{\varepsilon\delta}\left(\nabla\pd,\nabla 
         \hat{\B{\phi}}(\pd)   
            \right)\ge 0
\end{align*}
since for a.e. $x$ in $\Omega$ either $\nabla \hat{\B{\phi}}(\pd)(x)=\B{0}$ or $\nabla \hat{\B{\phi}}(\pd)(x)=\nabla \pd(x)$.

Applying Hölder's inequality in \eqref{testphi} and an absorption argument, which crucially requires $\B{f}\in L^2(\Omega;\R^N)$, we thus infer
\begin{align*}
\frac{\gamma}{\delta^2\varepsilon}\bnorm{P[\hat{\B{\phi}}(\pd)]}_{L^2}^{2}\le
\frac{C}{\delta}\bnorm{P[\hat{\B{\phi}}(\pd)]}_{L^{2}},
\end{align*}
and thus,
 \begin{align*}
    \frac{1}{\delta}
    \bnorm{P\big[\hat{\B{\phi}}(\pd)\big]}_{L^2}
  \le
     C.
 \end{align*}

As we now have bounded both the right-hand side of \eqref{RegPw} and $\pd$ itself in $\LN{2}$ uniformly in $\delta$ (see \eqref{pdH1}), we can again apply elliptic regularity theory (see \cite[Theorem 2.3.1.5]{Grisvard} or \cite[Theorem~3.2.3.1]{Grisvard}) to deduce \eqref{boundH2}.
\end{proof}

In order to reformulate \eqref{RegPw} by means of Lagrange multipliers that are expected to converge in the weak sense, we need to get rid of the projection in \eqref{boundpsi}. This is done analogously to \cite[Theorem 2.1]{Sarbu} and therefore, we only present the statement of the result without a proof.

\begin{Lem} \label{LEM:PHIHAT}
There exists a constant $C>0$ such that
\begin{align}\label{boundhp}
    \frac{1}{\delta}\bnorm{\hat{\B{\phi}}(\pd)}_{L^2}\le C
    \quad\text{for all $\delta>0$.}
\end{align}
\end{Lem}

Now, we introduce suitable Lagrange multipliers and pass to the limit in the the regularized equation.
\begin{Thm}\label{Thm:Lagr}
The initially chosen solution $\bphi\in \B{\mathcal{G}}^{\B{m}}$ of \eqref{GIp} possesses the regularity $\bphi\in H^2(\Omega;\mathbb{R}^N)$. Furthermore, there are Lagrange multipliers $\B{\Lambda}, \B{\mu}\in L^2(\Omega;\R^N)$ and $ \B{\vartheta}\in \R^N$ such that
\begin{alignat}{3}\label{PhStr}\tag{$GS^\eps$}
     \begin{aligned}
             -\gamma\varepsilon \Delta \bphi
         &=
             \frac{1}{\varepsilon}
             (\B{\Lambda}+\B{\vartheta}+\B{\mu})+P_{T\Sigma}\B{f}^{\bphi}
         &&\quad\text{a.e.~in }\Omega,\\
         \nabla\bphi\B{n}&=\B{0}
         &&\quad \text{on\;}\partial\Omega.
     \end{aligned}    
     \end{alignat}
    with
    \begin{alignat}{2}
        \Lambda^i&=\Lambda^j
        &&\quad\text{for all $i,j\in\{1,\dots,N\}$},\label{LamEqComp}\\
        \mu^i\ge 0\text{ and }\mu^i\varphi^i&=0 
        &&\quad\text{a.e.~in $\Omega$ for all $i\in\{1,\dots,N\}$},\label{muphiVan0}\\
        \sum_{i=1}^{N} \vartheta^i&=0.\label{thetaSum0}
    \end{alignat}
\end{Thm}
\begin{proof}
    In this proof, we will again use the notation $\B{f} = \B{f}^{\B{\varphi}}$.
    From \eqref{boundH2} we deduce the existence of a function $\overline{\bphi}\in H^2(\Omega;\mathbb{R}^N)$ such that
    \begin{alignat}{2}\label{konvpd}
    \begin{aligned}
        \pd&\rightharpoonup \overline{\bphi} &&\quad\text{in }H^2(\Omega;\mathbb{R}^N),\\
        \pd&\to \overline{\bphi} &&\quad\text{in }H^1(\Omega;\mathbb{R}^N),\\
        \pd&\to \overline{\bphi}&&\quad\text{a.e.~in }\Omega,\\
        \hat{\B{\phi}}(\pd) & \to 0 &&\quad\text{in }L^2(\Omega;\mathbb{R}^N),
    \end{aligned}    
\end{alignat}
as $\delta\to 0$ along a non-relabeled subsequence.
This directly implies that $\overline{\bphi}(x)\in\mathbb{R}^N_{+}$ for almost all $x\in \Omega$. Hence, since $\pd\in\tilde{\B{\mathcal{G}}}^{\B{m}}$, we know that $\overline{\bphi}\in\B{\mathcal{G}}^{\B{m}}$.

Recalling the definition of $\B{f}$ in \eqref{DEF:FPHI}, we now define the Lagrange multipliers of the regularized problem as
\begin{alignat}{1}\label{Lagrd}
    \begin{aligned}
            \B{\Lambda}_{\delta}
        &\coloneqq
           \frac1N\sum_{i=1}^{N}\left(
          \frac{\gamma}{\delta}\hat{\phi}^i(\pd)\right)\B{1},\\
        \B{\vartheta}_{\delta}
        &\coloneqq
        \fint_{\Omega}P_{T\Sigma}
        \left(
        \frac{\gamma}{\delta}
        \hat{\B{\phi}}(\pd)-\eps\B{f}
        \right)    
        \mathrm{\,d}x,\\
            \B{\mu}_{\delta}
        &\coloneqq 
            -\frac{\gamma}{\delta}\hat{\B{\phi}}(\pd).
    \end{aligned}
\end{alignat}
The reason why we do not reformulate the projection term $P_{T\Sigma}\B{f}$ by means of a Lagrange multiplier is that this is a term depending on $x$, which will produce terms of order $\mathcal{O}(\frac{1}{\eps^2})$ when we consider the inner expansions in Section~\ref{SEC:Leading} due to the involved derivative of eigenfunctions.    
Recalling Definition~\ref{DEF:Proj}, we have
 \begin{align*}
       \B{\Lambda}_\delta+\B{\vartheta}_{\delta}+\B{\mu}_{\delta}
    =
        -\frac{\gamma}{\delta}P_{T\Sigma}[\hat{\B{\phi}}(\pd)]-\eps\fint_{\Omega}P_{T\Sigma}\B{f}\mathrm{\,d}x.
 \end{align*}

Hence, we can write \eqref{reg} as
\begin{align}\label{ThmRE}
        \gamma\varepsilon \left(\nabla\pd,\nabla \B{\eta}\right)
        -\frac{1}{\varepsilon}\left(
                \B{\Lambda}_{\delta}+
                \B{\vartheta}_{\delta}
                +\B{\mu}_{\delta},\B{\eta}
            \right)=
        \left(P_{T\Sigma}\B{f},\B{\eta}\right)
        \quad \text{for all $\B{\eta}\in \HN$}.
\end{align}
We point out that the Lagrange multipliers are constructed in such a way that the factor $\frac{1}{\eps}$ corresponding to the scaling of the original potential $\psi$ is still present. This will be important in the next sections for the sharp-interface asymptotics.

We know from Lemma~\ref{LEM:PHIHAT}
that $\B{\Lambda}_{\delta},\B{\mu}_{\delta}\in L^2(\Omega;\mathbb{R}^N)$ and $\B{\vartheta}_{\delta}\in \R^N$ are bounded uniformly in $\delta$. 
Hence, we find a subsequence and $\B{\Lambda},\B{\mu}\in L^2(\Omega;\mathbb{R}^N)$ and $\B{\vartheta}\in \R^N$ such that
\begin{alignat}{2}\label{wkoLa}
\begin{aligned}
    \B{\Lambda}_{\delta}&\rightharpoonup \B{\Lambda} &&\quad\text{in\;}L^2(\Omega;\mathbb{R}^N),\\
    \B{\vartheta}_{\delta}&\to\B{\vartheta}&&\quad\text{in\;}\mathbb{R}^N,\\
    \B{\mu}_{\delta}&\rightharpoonup\B{\mu}&&\quad\text{in\;}L^2(\Omega;\mathbb{R}^N)
\end{aligned}    
\end{alignat}
as $\delta\to 0$. We additionally know from the definition of $\hat{\B{\phi}}$ in \eqref{prepsid} that $\B{\mu}\ge 0$ componentwise as weak convergence in $L^2(\Omega;\mathbb{R}^N)$ preserves non-negativity.
Furthermore, from the construction in \eqref{Lagrd} we directly deduce \eqref{LamEqComp} and \eqref{thetaSum0}.

Passing to the limit in \eqref{ThmRE}, we infer 
\begin{align}\label{wLim}
        \gamma\varepsilon \left(\nabla\overline{\bphi},\nabla \B{\eta}\right)
        -\frac{1}{\varepsilon}
        \left(
                \B{\Lambda}+\B{\vartheta}
                +\B{\mu},\B{\eta}
            \right)=
        \left(P_{T\Sigma}\B{f},\B{\eta}\right), 
        \quad \text{for all $\B{\eta}\in \HN$}.
\end{align}
Thus, the regularity $\overline{\bphi}\in H^2(\Omega;\mathbb{R}^N)$ and 
integration by parts yield the equation
\begin{alignat}{3}
\begin{aligned}
     -\gamma\varepsilon \Delta \overline{\bphi}
 &=
     \frac{1}{\varepsilon}
     (\B{\Lambda}+\B{\vartheta}+\B{\mu})+P_{T\Sigma}\B{f}
 &&\quad\text{a.e.~in }\Omega.
\end{aligned}    
\end{alignat}
If we can now show that for our initially fixed solution $\bphi\in\B{\mathcal{G}}^{\B{m}}$ of \eqref{GIp} it holds $\bphi=\overline{\bphi}$, the proof is complete.

Let us consider the test function $\B{\eta}\coloneqq\overline{\bphi}-\bphi\in \HL$. 
Due to \eqref{LamEqComp}, we have $\big(\B{\Lambda}_{\delta},\B{\eta}\big)=0$, as $\sum_{i=1}^{N}\eta^i=0$ because of $\overline{\bphi},\bphi\in \B{\mathcal{G}}^{\B{m}}$. 
In view of  \eqref{thetaSum0} we know that $\big(\B{\vartheta},\B{\eta}\big)=0$, because by construction $\int_{\Omega} \B{\eta}\mathrm{\,d}x=0$.

As already mentioned, we have $\B{\mu}_{\delta}\ge 0$. Hence, using the monotonicity \eqref{phmon}, we infer
\begin{align*}
\left(\B{\mu}_{\delta},\pd\right)=-\frac{1}{\delta}\left(\hat{\B{\phi}}(\pd),\pd\right)\le 0.
\end{align*}
Using the convergences \eqref{konvpd} and \eqref{wkoLa}, we deduce $(\B{\mu},\overline{\bphi})\le 0$. Recalling $\B{\mu}\ge 0$ 
and that $\overline{\bphi}\in  \B{\mathcal{G}}^{\B{m}}$ is component wise non-negative, we already deduce $(\B{\mu},\overline{\bphi})=0$.

As also $\bphi\in \B{\mathcal{G}}^{\B{m}}$ and $\bphi$ is component wise non-negative, we have $(\B{\mu},\bphi)\ge 0$.
Combining these results and testing \eqref{wLim} with our particular choice $\B{\eta}=\overline{\bphi}-\bphi$, we get
\begin{align*}
      \gamma\varepsilon \left(\nabla\overline{\bphi},\nabla[\overline{\bphi}-\bphi]\right)
     =
     -\frac{1}{\varepsilon}(\B{\mu},\bphi)+(P_{T\Sigma}\B{f},\overline{\bphi}-\bphi)
     \le (P_{T\Sigma}\B{f},\overline{\bphi}-\bphi).
\end{align*} 
Considering the gradient inequality \eqref{GIvar} tested with $\tilde{\bphi}=\overline{\bphi}\in \B{\mathcal{G}}^{\B{m}}$, we have
\begin{align*}
        \gamma\varepsilon 
        \left(\nabla\bphi,\nabla[\overline{\bphi}-\bphi]\right)
     \ge
        (\B{f},\overline{\bphi}-\bphi)
        =(P_{T\Sigma}\B{f},\overline{\bphi}-\bphi).
\end{align*}
Hence, by subtracting both inequalities, we infer
\begin{align*}
    \gamma\varepsilon 
    \left(\nabla[\overline{\bphi}-\bphi],\nabla[\overline{\bphi}-\bphi]\right)\le 0.
\end{align*}
As $\int_{\Omega}\overline{\bphi}-\bphi\mathrm{\,d}x=0$, this gives us the desired identity $\bphi=\overline{\bphi}\in H^2(\Omega;\mathbb{R}^N)$.

From the previous reasoning we know
    \begin{align*}
    \sum_{i=1}^{N}\int_{\Omega}\mu^i\varphi^i\mathrm{\,d}x=(\B{\mu},\bphi)_{L^2}=0,
    \end{align*}
    Furthermore, we know that $\B{\mu},\bphi\ge 0$ componentwise and thus, each summand in above equality has to be identical to $0$. This verifies \eqref{muphiVan0}.
\end{proof}

In the following, we use the above knowledge to show that our asymptotic expansions will produce a state equation and a gradient equality in the sharp-interface limit.

%%%%%%%%%%%%%%%%%%%%%%%%%%%%%%%%%%%%
%%%%%%%%%%%ASYMPTOTIC EXPANSIONS%%%%
%%%%%%%%%%%%%%%%%%%%%%%%%%%%%%%%%%%%

\section{Asymptotic expansions}
As mentioned above, we will now perform the procedure of sharp-interface asymptotics. Therefore, we start by analyzing outer and inner expansions approximating the quantities involved in our problem. The outer expansions are used to approximate these quantities in regions far away from the interfacial layers. They will be used to derive the state equation in the sharp-interface limit. The inner expansions are used in regions close to the interfacial layers where the phase transition takes place. They will provide boundary conditions for the equations obtained in the sharp-interface limit. As these layers are expected to scale proportionally to $\eps$, a rescaling is needed here. By comparing the leading order equations, we will obtain jump conditions at the phase interfaces within the design domain and a sharp-interface version of the gradient equality \eqref{PhStr}.

In the following, we choose $\left(\bphi^{\varepsilon}\right)_{\varepsilon>0}\subset \B{\mathcal{G}}^{\B{m}}$ as a sequence of minimizers of the optimization problem \eqref{Pepsla}. For $r=1,\dots,l$, $\left(\bw_{n_r}^{\eps},\lambda_{n_r}^{\eps}\right)_{\eps>0}\subset \Hd\times \R$ denotes the corresponding sequence of $L^2_{\bphi}(\Omega;\R^d)$-normalized eigenfunctions and eigenvalues, which are non-trivial solutions of the state equation \eqref{state} involved in the optimization problem \eqref{Pepsla}.

\subsection{Outer expansions}\label{SEC:outer}
As in \cite{Blank}, we first consider the asymptotic expansions in regions ``far'' away from the interface. Therefore, we assume expansions of the form
\begin{alignat}{2}
\begin{aligned}\label{ExpPw}
    \B{\varphi}^{\varepsilon}(x)
    &=\sum_{k=0}^{\infty}\varepsilon^k\B{\varphi}_k(x),\\
    \lambda^\eps_{n_r}
    &=
    \sum_{k=0}^{\infty}\varepsilon^{k}\lambda_{k,n_r},\\
    \B{w}_{n_r}^{\varepsilon}(x)
    &= \sum_{k=0}^{\infty}\varepsilon^{k}\B{w}_{k,n_r}(x)
\end{aligned}
\end{alignat}
for all $x\in\Omega$.
Furthermore, we demand for all $x\in\Omega$ that
${\B{\varphi}_0(x)\in \B{G}}$, $\B{\varphi}_k(x)\in T\Sigma$, $\fint \B{\varphi}_0=\B{m}$ and $\fint \B{\varphi}_k=\B{0}$
for $k\ge 1$, in order to be compatible with the constraints on the phase-field formulated in Section~\ref{Sec:Phase}. As we are concerned with a formal limit process, we assume all the appearing quantities to possess a suitable regularity such that we can write the state equation \eqref{state} in its strong formulation. 

Using standard arguments relying on the $\Gamma$-convergence of the Ginzburg--Landau energy in \cite{Baldo}, we can partition the domain as
\begin{align}
    \label{DEC:OMEGA}
    \Omega=\bigcup_{i=1}^N\Omega_i\cup \mathcal{N} 
    \quad\text{with}\quad
    \Omega_i\coloneqq \left\{\B{\varphi}_0=\B{e}_i\right\},
\end{align}
where $\mathcal{N}\subset \Omega$ is a Lebesgue null set. In general, the sets $\Omega_i$ are only  finite perimeter sets. This follows from the boundedness of the Ginburg--Landau energy, the inequality \cite[(3.1)]{Baldo} and \cite[Proposition 2.2]{Baldo}. Nevertheless, for our asymptotic analysis we assume them to be smooth enough.

With this knowledge, we are in a position to derive the limit state equation resulting from \eqref{state} in the framework of outer asymptotic expansions.

\begin{Claim}\label{Thm:Lam-1}
   Recall the scaling of $\mathbb{C}$ and $\rho$ in \eqref{elT}, i.e.,
    \begin{align}
    \label{EXP:CRHO}
    \begin{aligned}
	\mathbb{C}(\B{\varphi})
	&=\overline{\mathbb{C}}(\B{\varphi})
	    +\tilde\C^N \eps^k \alpha_V(\varphi^N)
	=\sum_{i=1}^{N-1}\mathbb{C}^{i}\alpha_M(\varphi^i)
	    +\tilde\C^N \eps^k \alpha_V(\varphi^N),\\
	\rho(\B{\varphi})
	&=\overline{\rho}(\B{\varphi})+\tilde\rho^N \eps^l \beta_V(\varphi^N)
	=\sum_{i=1}^{N-1}\rho^{i}\beta_M(\varphi^i)+\tilde\rho^N \eps^l \beta_V(\varphi^N),
    \end{aligned}
    \end{align}
    for $\bphi\in \B{G}$. 
     Then, for $r\in \left\{1,\dots,l\right\}$, we obtain
            that the pair $(\lambda_{0,n_r},\bw_{0,n_r})$ fulfills the eigenvalue equations in the material regions
    \begin{align}\label{state0}\tag{${SE}^{i}_0$}
    \begin{cases}
        \begin{array}{rll}
                -\nabla\cdot\left[\mathbb{C}^{i}\mathcal{E}(\B{w}_{0,n_r})\right]
            &=
                \lambda_{0,n_r}
                \rho^i\B{w}_{0,n_r}
            &\quad
                \text{in }\Omega_i,\\
                \B{w}_{0,n_r}
            &=
                \B{0}
            &\quad
                \text{on }\Gamma_D\cap \partial\Omega_i,\\
                \left[
                    \mathbb{C}^{i}\mathcal{E}(\B{w}_{0,n_r})
                \right]\B{n}
            &=
                \B{0}
            &\quad
                \text{on }\Gamma_0\cap \partial\Omega_i,
        \end{array}
    \end{cases}
    \end{align}
    for $i=1,\dots, N-1$. Furthermore, the normalization condition \eqref{EveNorm} is transferred to the limit eigenfunction $\bw_{0,n_r}$ meaning that
    \begin{align}\label{EVNorm0}
        1
        =\int_{\Omega}\overline{\rho}(\bphi_0)\abs{\bw_{0,n_r}}^2\mathrm{\,d}x
        =\sum_{i=1}^{N-1}\int_{\Omega_i}\overline{\rho}(\bphi_0)\abs{\bw_{0,n_r}}^2\mathrm{\,d}x.
    \end{align}
    In particular, the eigenfunction $\bw_{0,n_r}$ is non-trivial in $\Omega_i$ for at least one index $i\in\{1,\dots,N-1\}$. Thus, $\bw_{0,n_r}$ cannot be a localized eigenmode as it is not supported only in the void region~$\Omega_N$.
\end{Claim}    

 \begin{Rem}\label{REM:Korn}
 \begin{enumerate}[label=(\alph*), leftmargin =*]
     \item\label{KornA} Of course, the eigenvalue $\lambda_{n_r}^\eps$ could degenerate in the limit, i.e., $\lambda_{0,n_r}=0$. This is no contradiction to the normalization \eqref{EVNorm0} because $\bw_{0,n_r}$ could potentially be a non-trivial constant in each material region $\Omega_i$. If each material region $\Omega_i$ shares a sufficiently nice part of the boundary with $\Gamma_D$, one can use Korn's inequality (see, e.g., \cite[Theorem 6.15-4]{CiarletFA} or \cite[Theorem 62.13]{Zeidler}) to deduce that $\bw_{0,n_r}=0$ in each $\Omega_i$, which would then indeed contradict \eqref{EVNorm}. The inner expansions will provide us with boundary conditions that allow us to refine this statement, see Section~\ref{SEC:limitpr} (especially Remark~\ref{REM:Korn2}).
     
     \item\label{KornB} In the case $\lambda_{0,n_r}>0$, even though the limit eigenvalue equations \eqref{state0} hold for all $i\in\{1,\dots,N-1\}$, the eigenfunction $\bw_{0,n_r}$ could potentially be non-trivial only in one particular material region $\Omega_i$ but vanish in all other material regions $\Omega_j$ with $j\in\{1,\dots,N-1\}\setminus\{i\}$. This means that a non-trivial equation might hold only in one single material region.
 \end{enumerate}

 \end{Rem}
Let us show the Claim~\ref{Thm:Lam-1} assuming that outer expansions of the form \eqref{ExpPw} exist.
            For the sake of a cleaner presentation, we will now fix the index $n_r\in \N$ and in the following, we omit the subscript $n_r$.
            In the spirit of formal asymptotics, we consider the state equation \eqref{state}, i.e.,
            \begin{align*}
                -\nabla\cdot\left[\mathbb{C}(\B{\varphi}^\eps)\mathcal{E}(\B{w}^\eps)\right]
                =
                \lambda^\eps
                \rho(\B{\varphi}^\eps)\B{w}^\eps
                \quad
                \text{a.e.~in }\Omega,  
            \end{align*}
            and the normalization condition 
            \begin{align}\label{EVNorm}
                1=\int_{\Omega} \rho(\bphi^\eps)\abs{\bw^\eps}^2\mathrm{\,d}x
            \end{align}
            resulting from \eqref{EveNorm}.
            Then, we plug in the asymptotic expansions \eqref{ExpPw} and consider each resulting order in $\eps$ separately.
            
We deduce that \eqref{EVNorm} reads to order $\mathcal{O}(1)$ as
\begin{align*}
    1&=
    \int_{\Omega}\overline{\rho}(\bphi_0)\abs{\bw_0}^2\mathrm{\,d}x=
    \sum_{i=1}^{N-1}\int_{\Omega_i}\rho^i\abs{\bw_0}^2\mathrm{\,d}x,
\end{align*}
which proves \eqref{EVNorm0}.
As a consequence, $\bw_0$ has to be non-trivial in in $\Omega_i$ for at least one index $i\in\{1,\dots,N-1\}$. 

Eventually, we compare the contributions of order $\mathcal{O}(1)$ in the state equation. We obtain
\begin{align}\label{O0Out}
    -\nabla\cdot\left[\overline{\C}(\bphi_0)\E(\bw_0)\right]=\lambda_0\overline{\rho}(\bphi_0)\bw_0\quad\text{a.e.~in }\Omega,
\end{align}
which reads for each phase
\begin{align*}
    -\nabla\cdot\left[\C^i\E(\bw_0)\right]=\lambda_0\rho^i\bw_0\quad\text{a.e.~in }\Omega_i
\end{align*}
for $i=1,\dots,N-1$.
The remaining boundary conditions on the \textit{outer} boundary $\Gamma$ follow directly by plugging in the asymptotic expansion into \eqref{state}. 
This completes the argumentation.  

\subsection{Intermezzo on spurious eigenmodes}\label{intermezzo}    
       
As already mentioned in the introduction, we now want to analytically justify the model that will be chosen for the numerical computations in order to avoid spurious eigenmodes. As we have seen in the above reasoning, assuming outer expansions of the form \eqref{ExpPw} and the general decomposition of $\C$ and $\rho$ as in \eqref{EXP:CRHO}, we recover the desired limit system. 
In this subsection, in order to discuss spurious eigenmodes, we consider \eqref{EXP:CRHO} with the
\emph{specific} choices $k=1$, $l=2$, 
% but general interpolation functions $\alpha_M,\alpha_V,\beta_M,\beta_V$ as introduced in \eqref{elT}.
and in addition to \eqref{elT} and \eqref{eq:def_inter}, we choose the interpolation functions $\alpha_M$ and $\beta_M$ such that \begin{align}\label{eq:global_nonneg}
    \alpha_M(s)>0, \quad \beta_M(s)>0 \quad \text{for all $s\in\R\setminus\{0\}$}.
\end{align}
In numerical simulations the phenomenon of spurious eigenmodes is a serious issue, see \cite{BucurMartinetOudet,Pedersen,Bendsoe,Allaire}. The problem is that if the model parameters are not chosen correctly, eigenmodes that are supported only in the void region can actually emerge. Of course, the associated eigenvalues are  \textit{unphysical} as the void should not contribute to the resonance behaviour of the structure. Nevertheless, even though spurious eigenmodes might not be avoided in numerical simulations,  
they do not pose any problem if their associated eigenvalues are large since then, they do not affect the part of the spectrum that is involved in the optimization problem \eqref{Pepsla}. 
For this reason, as also observed in the aforementioned literature, the key idea is to choose the scaling and interpolation in \eqref{EXP:CRHO} in such a way that spurious eigenmodes will only produce large eigenvalues or more precisely, eigenvalues $\lambda^\eps$ with $\lambda^\eps\to\infty$ as $\eps\to 0$. In particular, this means that by using an adequate interplay of scaling and interpolation, spurious eigenmodes will not enter the sharp interface limit as their eigenvalues leave the considered part of the spectrum.

In order to allow for spurious eigenmodes in our asymptotic expansions, we have to include terms of \textit{negative} order in $\eps$.
  
\begin{Claim}\label{Thm:NumScale}
  Assume the following outer asymptotic expansions
  \begin{alignat}{2}
  	\begin{aligned}\label{ExpNum}
  		\B{\varphi}^{\varepsilon}(x)
  		&=\sum_{k=0}^{\infty}\varepsilon^k\B{\varphi}_k(x),\\
  		\lambda^\eps_{n_r}
  		&=\sum_{k=-m}^{\infty}\varepsilon^{k}\lambda_{k,n_r},\\
  		\B{w}_{n_r}^{\varepsilon}(x)
  		&=\sum_{k=-m}^{\infty}\varepsilon^{k}\B{w}_{k,n_r}(x),
  	\end{aligned}
  \end{alignat}
  for an arbitrary $m\in\N$.
 Let $\mathbb{C}$ and $\rho$ be given as in \eqref{EXP:CRHO} with $k=1$, $l=2$, i.e.,
  \begin{align}
  	\label{QuadExp}
  	\begin{split}
  		\mathbb{C}(\B{\varphi})
  		=\overline{\mathbb{C}}(\B{\varphi})
  		+\tilde{\mathbb{C}}^N\eps\alpha_V(\varphi^N)
  		&=\sum_{i=1}^{N-1}\mathbb{C}^{i}\alpha_M({\varphi}^i)
  		+\tilde{\mathbb{C}}^N\eps\alpha_V(\varphi^N),\\
  		\rho(\B{\varphi})
  		=\overline{\rho}(\B{\varphi})+\tilde{\rho}^N\eps^2\beta_V(\varphi^N)
  		&=\sum_{i=1}^{N-1}\rho^{i}\beta_M({\varphi}^i)+\tilde{\rho}^N\eps^2\beta_V(\varphi^N),
  	\end{split}
  \end{align}
  for $\bphi\in \B{G}$. 
  Then, for $r\in \left\{1,\dots,l\right\}$, we obtain $\bw_{k,n_r}=\B{0}$ and $\lambda_{k,n_r}=0$ for $k<-1$ and
  the pair $(\lambda_{-1,n_r},\bw_{-1,n_r})$ fulfills 
  \begin{align}\label{DirVoid}\tag{${SE}^{N}_0$}
  	\begin{cases}
  		\begin{array}{rll}
  			-\nabla\cdot\left[\tilde{\C}^N\mathcal{E}(\B{w}_{-1,n_r})\right]
  			&=
  			\lambda_{-1,n_r}
  			\left[\tilde{\rho}^N+\overline{\rho}(\bphi_1)\right]\B{w}_{-1,n_r}
  			&\quad
  			\text{in }\Omega_N,\\
  			\B{w}_{-1,n_r}
  			&=
  			\B{0}
  			&\quad
  			\text{on }\Gamma_D\cap\partial\Omega_N,\\
            \left[
			\tilde{\mathbb{C}}^{N}\mathcal{E}(\B{w}_{-1,n_r})
			\right]\B{n}
  			&=
  			\B{0}
  			&\quad
  			\text{on }\Gamma_0\cap\partial\Omega_N.
  		\end{array}
  	\end{cases}
  \end{align}
\end{Claim}    
\begin{Rem}
The asymptotic analysis in the argumentation of this subsection is crucially based on the interplay of the \emph{non-negative} interpolation functions $\alpha_M$ and $\beta_M$, see \eqref{eq:global_nonneg}, and the different $\eps$-scaling of the void components in \eqref{EXP:CRHO}.
Note that these two features are also important for our numerical experiments in Section~\ref{sec:num}, 
where the quadratic interpolation of $\C$ and $\rho$ as well as the relatively lower scaling in $\eps$ of the void contribution of $\rho$ compared to the void contribution of $\C$ are crucial to obtain meaningful results. It has also already been observed in the literature that a relatively lower scaling of mass compared to stiffness is an appropriate choice to deal with localized eigenmodes, see \cite{Allaire,Pedersen,BucurMartinetOudet}.
\end{Rem}	

  We now argue why Claim~\ref{Thm:NumScale} is true. Therefore, we consider the state equation \eqref{state} and the normalization \eqref{EVNorm}. First of all, we note that plugging in the asymptotic expansion of $\bphi^\eps$ into \eqref{QuadExp} yields
  \begin{alignat}{2}
  	\begin{aligned}\label{Taylor}
  		\C(\bphi^\eps)&=\overline{\C}(\bphi_0)+\eps\tilde{\C}^N\alpha_V(\varphi_0^N)+\eps^2\overline{\C}(\bphi_1)+\mathcal{O}(\eps^3)\\
  		\rho(\bphi^\eps)&=\overline{\rho}(\bphi_0)+\eps^2(\overline{\rho}(\bphi_1)+\tilde{\rho}^N\beta_V(\varphi_0^N))+\mathcal{O}(\eps^3).
  	\end{aligned}
  \end{alignat}
  As a first step let us show that $\bw_{k}=\B{0}$ in $\Omega$ for $k=-m,-m+1,\dots,-2$.
  Therefore let us start with the contribution of lowest order $\mathcal{O}(\eps^{-2m})$ in \eqref{EVNorm}, which reads as
  \begin{align}\label{Localized}
  	0=\int_{\Omega}\overline{\rho}(\bphi_0)\abs{\bw_{-m}}^2\text{\,d}x.
  \end{align}
This implies that $\bw_{-m}=0$ in $\Omega\backslash\Omega_N$, or in other words, $\bw_{-m}$ is localized in the void region. Now, we consider \eqref{EVNorm} to the order $\mathcal{O}(\eps^{-2m+2})$. We have 
  \begin{align}\label{O-m+2}
  	0=\int_{\Omega}\overline{\rho}(\bphi_0)\abs{\bw_{-m+1}}^2+2\overline{\rho}(\bphi_0)\bw_{-m}\cdot \bw_{-m+2}+(\overline{\rho}(\bphi_1)+\tilde{\rho}^N\beta_V(\varphi_0^N))\abs{\bw_{-m}}^2\text{\,d}x.
  \end{align}
  Here, we used that $-2m+2<0$.
  As $\bw_{-m}$ is localized in the void we infer
  \begin{align*}
  	0=\int_{\Omega}2\overline{\rho}(\bphi_0)\bw_{-m}\cdot \bw_{-m+2}\text{\,d}x.
  \end{align*}
  Thus, due to the non-negativity of the first summand in \eqref{O-m+2} we deduce
  \begin{align}
  	0=\int_{\Omega}(\overline{\rho}(\bphi_1)+\tilde{\rho}^N\beta_V(\varphi_0^N))\abs{\bw_{-m}}^2\text{\,d}x.
  \end{align}			
  Using the crucial global non-negativity of $\beta_M$ as assumed in \eqref{eq:global_nonneg}, 
  we have $\overline{\rho}(\bphi_1)\ge 0$, see \eqref{QuadExp}. Moreover, $\varphi_0=\B{e}_N$ in $\Omega_N$ and we thus deduce from \eqref{eq:def_inter}
  \begin{align}
  	0=\int_{\Omega_N}\tilde{\rho}^N\abs{\bw_{-m}}^2\text{\,d}x.
  \end{align}
Hence, since $\tilde{\rho}^N$ is positive, we infer $\bw_{-m}=\B{0}$ in  $\Omega$. These steps can now be repeated until the critical order $\mathcal{O}(1)$ is reached because up to this order, the normalization equation \eqref{EVNorm} possesses a trivial left hand side. This shows $\bw_{k}=0$ for $k=-m,-m+1,\dots,-2$.
As in \eqref{Localized}, we additionally conclude that $\bw_{-1}=\B{0}$ in $\Omega\backslash\Omega_N$.
  
  With this knowledge, we are in a position to show $\lambda_{k}=0$ for $k=-m,-m+1,\dots,-2$. Therefore let us consider the energy associated with \eqref{state}, i.e.,
  \begin{align}\label{EVEnergy}
  	\lambda^\eps=\int_{\Omega}\C(\bphi^\eps)\E(\bw^\eps):\E(\bw^\eps)\text{\,d}x.
  \end{align}
Due to the fact that $\bw_{k}=\B{0}$ in $\Omega$ for $k=-m,-m+1,\dots,-2$ and $\bw_{-1}=\B{0}$ in $\Omega\backslash\Omega_N$, we deduce that the right hand side is of leading order $\mathcal{O}(\eps^{-1})$. This directly implies $\lambda_{k}=0$ for $k=-m,-m+1,\dots,-2$ as well as
  \begin{align*}
  	\lambda_{-1}=\int_{\Omega_N}\tilde{\C}^N\E(\bw_{-1}):\E(\bw_{-1})\text{\,d}x.
  \end{align*}
  It remains to show that $(\lambda_{-1},\bw_{-1})$ solves the desired limit eigenvalue problem.
  Therefore we consider the state equation \eqref{state} to order $\mathcal{O}(1)$
  \begin{align*}
	-\nabla\cdot\left[
		\tilde{\C}^N\alpha_V(\varphi_0^N)\E(\bw_{-1})+\overline{\C}(\bphi_0)\E(\bw_0)\right]=
		&\phantom{+}\lambda_{1}\overline{\rho}(\bphi_0)\bw_{-1}+
		\lambda_0\overline{\rho}(\bphi_0)\bw_0+\\
		&
		\lambda_{-1}\overline{\rho}(\bphi_0)\bw_1+\lambda_{-1}(\tilde{\rho}^N\beta_V(\varphi_0^N)+\overline{\rho}(\bphi_1))\bw_{-1}.
  \end{align*}
  In $\Omega_N$ this simplifies to
  \begin{align*}
  	-\nabla\cdot\left[
  	\tilde{\C}^N\E(\bw_{-1})
  	\right]=
  	\lambda_{-1}(\tilde{\rho}^N+\overline{\rho}(\bphi_1))\bw_{-1}\quad\text{in }\Omega_N.
  \end{align*}
  
Summing up this intermezzo, we have now seen that even if spurious eigenmodes are not excluded, the appropriate choice of the model parameters will force the associated eigenvalues to leave the spectrum in the limit $\eps\to 0$. Hence, the spurious modes do not affect our optimization problem as they leave the considered part of the spectrum.

\subsection{Inner expansions}\label{SEC:inner}
In the interfacial regions, i.e., in layers separating two outer regions, we need to rescale our coordinate system in order to take into account that $\B{\varphi}^{\varepsilon}$ changes rapidly in directions perpendicular to the interface.

Therefore, for all $i,j=1,\dots,N$, we write $\Gamma=\Gamma_{ij}$ to denote the sharp-interface separating $\Omega_i$ and $\Omega_j$. Moreover, let $\B{n}_{\Gamma_{ij}}$ denote the unit normal vector field on $\Gamma$ pointing from $\Omega_i$ to $\Omega_j$.
In the following, we omit these indices to provide a cleaner presentation.

We now introduce a suitable coordinate system that fits the geometry of the interface. 
The following discussion can be found, e.g., in \cite{Abels} and thus we only give the key steps needed for our analysis.
Let us choose a local parametrization 
\begin{align}
    \label{DEF:GAMMA}
	\B{\gamma}: U\to \mathbb{R}^d,
	\quad
	\B{\gamma}(U) \subseteq \Gamma
\end{align}
of $\Gamma$, where $U$ is an open subset of $\mathbb{R}^d$. We further define $\B{\nu} := \B{n}_\Gamma \circ \B{\gamma}$.

As we want to describe a whole neighborhood surrounding the local part of the interface ${\B{\gamma}(U)\subset\Gamma}$, we introduce the signed distance function relative to $\Omega_i$ which satisfies $d(x)>0$ if $x\in \Omega_j$ and $d(x)<0$ if $x\in \Omega_i$. For more details concerning the signed distance function we refer the reader to \cite[Sec. 14.6]{Gilbarg}. 
By introducing the rescaled distance coordinate $z(x)\coloneqq \frac{1}{\eps}d(x)\in \R$ we define for fixed $z\in \R$ and sufficiently small $\eps>0$ the $(d-1)$-dimensional submanifold
\begin{align*}
  \Gamma_{\varepsilon z}\coloneqq \big\{\B{\gamma}(\bs)+\varepsilon z \B{\nu}(\bs)\,\big|\, \bs\in U \big\},
\end{align*}
which describes a translation of $\Gamma$ in the direction $\B{\nu}$. Here, for $x$ belonging to a sufficiently thin tubular neighbourhood around $\gamma(U)$, $\bs(x)$ is the unique point in $U$ such that $\gamma(\bs)\in \Gamma$ is the orthogonal projection of $x$ onto $\Gamma$. The
summand
        \begin{align*}
            \varepsilon z(x)\B{\nu}\big(s(x)\big) = d(x) \B{n}_\Gamma\big(\B{\gamma}\big(\bs(x)\big)\big)
        \end{align*}
shifts the point $\B{\gamma}\big(\bs(x)\big)$ back onto $x$.
Hence, a sufficiently thin tubular neighborhood around $\B{\gamma}(U)$ can be expressed by the coordinate system $(\bs,z)$.\\
Now we can express the transformation of differential operators with respect to the coordinate transformation $x\mapsto (\B{s}(x),z(x))$. Therefore, let us consider an arbitrary scalar function
\begin{align*}
    b(x)=\hat{b}(\bs(x),z(x)).
\end{align*}
It holds
\begin{align}\label{gradx}
  	\nabla_{x}b
  =
  	\nabla_{\Gamma_{\varepsilon z}}\hat{b}
    +\frac{1}{\varepsilon}\left(\partial_{z}\hat{b}\right)\B{\nu}
  =\frac{1}{\varepsilon}\left(\partial_{z}\hat{b}\right)\B{\nu}+\nabla_{\Gamma}\hat{b}+\mathcal{O}(\eps),	
\end{align}
where $\nabla_{\Gamma_{\eps z}}$ stands for the surface gradient on $\Gamma_{\eps z}$.
Proceeding analogously, we deduce that the divergence of a vector-valued function $\B{j}(x)=\hat{\B{j}}(\bs(x),z(x))$ can be expressed as 
\begin{align}\label{divx}
  	\nabla_{x}\cdot \B{j}
  =
  	\nabla_{\Gamma_{\varepsilon z}}\cdot \hat{\B{j}}
  +
  	\frac{1}{\varepsilon}\partial_{z}\hat{\B{j}}\cdot \B{\nu}
  =
    \frac{1}{\varepsilon}\partial_{z}\hat{\B{j}}\cdot \B{\nu}+
    \nabla_{\Gamma}\cdot \hat{\B{j}}+\mathcal{O}(\eps),
\end{align}
where $\nabla_{\Gamma_{\eps z}}\cdot \hat{\B{j}}$ stands for the surface divergence on $\Gamma_{\eps z}$.
Furthermore the full gradient of a vector-valued function is given by
\begin{align}\label{Jb}
  	\nabla_{x}\B{j}=\frac{1}{\varepsilon}\partial_{z}\hat{\B{j}}\otimes\B{\nu}+\nabla_{\Gamma}\hat{\B{j}}+\mathcal{O}(\eps),
\end{align}
where $\otimes$ denotes the dyadic product that is defined as
$\B{a}\otimes \B{b}=(a_ib_j)_{i,j=1}^{d}$ 
for all $\B{a},\B{b}\in \R^d$.
Analogously, for a matrix-valued function 
$$\mathcal{A}(x)=\left(a_{ij}(x)\right)_{i,j=1}^{d}=\hat{\mathcal{A}}(\bs(x),z(x)),$$ 
we apply formula \eqref{divx} to each component of the row-wise defined divergence $\nabla_{x}\cdot \mathcal{A}$. We obtain
\begin{align}\label{divA}
  \nabla_{x}\cdot\mathcal{A}=\nabla_{\Gamma}\cdot \hat{\mathcal{A}}+\frac{1}{\varepsilon}\partial_{z}\hat{\mathcal{A}}\B{\nu}+\mathcal{O}(\eps).
\end{align}

For the Laplacian we obtain the representation
\begin{alignat}{2}
\begin{aligned}\label{Laplace}
  	\Delta_{x}b
  &=  	
  	\Delta_{\Gamma_{\varepsilon z}}\hat{b}
  	+\frac{1}{\varepsilon}\left(\Delta_xd\right)\partial_{z}\hat{b}
  	+\frac{1}{\varepsilon^{2}}\partial_{zz}\hat{b}
  &=\frac{1}{\varepsilon^2}\partial_{zz}\hat{b}
  -\frac{1}{\varepsilon}\left(\hat\kappa+\varepsilon z \big|\hat{\mathcal{W}}\big|^2\right)\partial_{z}\hat{b}+
  \Delta_{\Gamma}\hat{b}+
  \mathcal{O}(\eps).	
\end{aligned}
\end{alignat}
Here $\mathcal{W}$ denotes the \textit{Weingarten map} 
associated with $\Gamma$ that is given by
\begin{align}\label{Wein}
    \mathcal{W}(x)\coloneqq -\nabla _{\Gamma}\B{n}_{\Gamma}(x) \in \R^{d\times d}
    \quad\text{for all $x\in\Gamma$},
\end{align}
see, e.g., \cite[Appendix~B]{Eck}. Its non-trivial eigenvalues $\kappa_1,\dots,\kappa_{d-1}$ are the principal curvatures of $\Gamma$ and its spectral norm can be expressed as 
$$\abs{\mathcal{W}}=\sqrt{\kappa_1^2+\dots \kappa_{d-1}^2}.$$ 
Furthermore $\kappa$ denotes the mean curvature which is defined as the sum of the principal curvatures of $\Gamma$. Note that in view of \eqref{Wein}, $\kappa$ can be expressed as
\begin{align}\label{RepCurv}
    \kappa(x)=-\nabla_{\Gamma}\cdot \B{n}_{\Gamma}(x) \quad\text{for all $x\in\Gamma$},
\end{align}
which will be important for later purposes.

To conclude this section, we introduce the inner expansions that we will work with in the next section. Therefore, we make the ansatz
\begin{alignat}{1}\label{inner}
\begin{aligned}
  \B{w}_{n_r}^{\eps}(x) &=\sum_{k=0}^{\infty}\varepsilon^{k}\, \bW_{k,n_r}\big(\bs(x),z(x)\big),\\
  \B{\varphi}^{\eps}(x) &=\sum_{k=0}^{\infty}\varepsilon^{k}\, \B{\Phi}_k\big(\bs(x),z(x)\big),
\end{aligned}
\end{alignat}
where we assume $\B{\Phi}_0(\bs(x),z(x))\in \B{G}$ and
$\B{\Phi}_k(\bs(x),z(x))\in T\Sigma^N$  for all $k\ge 1$. 
In the next section, we will relate these inner expansions to the outer expansions that were introduced before.

\begin{Rem} \label{REM:EXP:LAMBDA}
Note that the eigenvalues $\lambda_{n_r}^\eps$ do not depend locally on $x\in \Omega$ and thus, their inner expansion simply equals their outer expansion.
\end{Rem}

%%%%%%%%%%%%%%%%%%%%%%%%%%%%%%%%%%%%
%%%%%%%%%%%% MATCHING %%%%%%%%%%%%%%
%%%%%%%%%%%%%%%%%%%%%%%%%%%%%%%%%%%%

\section{The matching conditions}\label{SEC:Match}
So far, we have constructed outer expansions which are supposed to hold inside the material regions $\Omega_i$ for $i=1,\dots, N$ as well as inner expansions which are supposed to hold in a tubular neighborhood around the sharp-interfaces $\Gamma_{ij}$. Note that due to the construction in the previous section, the thickness of this tubular neighborhood is proportional to $\varepsilon$. In order to be compatible, both expansions must match in a suitable intermediate region by suitable matching conditions. This region is approximately given by all points $x\in \Omega$ with the property $\text{dist}(x,\Gamma)\le  \varepsilon^{\theta}$ for some fixed $\theta\in (0,1)$. This means we stretch the tubular neighborhood the inner expansions were constructed on from a thickness proportional to $\eps$ to a thickness proportional to $\eps^{\theta}$ and relate both expansions in this region. These matching conditions will be expressed as limit conditions for the inner expansions when $\varepsilon\to 0$ or equivalently $z\to \pm \infty$ depending on which side we approach the interface from.
This procedure is again standard in the context of formally matched asymptotics and we only state the matching conditions, for the computations see \cite{Fife}.

Using the notation 
\begin{align}
    (\B{v})_j(x)\coloneqq \underset{\delta \searrow 0}{\lim}\;\B{v}\big(x\pm\delta\B{n}_\Gamma(x)\big)
\end{align}
for the lowest order term we have the matching condition
\begin{alignat}{2}\label{Ph0}
\begin{aligned}
		\B{\Phi}_0(\bs,z)&\to
	\begin{cases}
		(\bphi_0)_j(x)=\B{e}_j&\quad \text{as }z\to +\infty,\\
		(\bphi_0)_i(x)=\B{e}_i&\quad \text{as }z\to -\infty,\\
	\end{cases}\\
	\partial_{z}\B{\Phi}_0(\bs,z)&=0\quad \text{as }z\to \pm\infty.	
\end{aligned}
\end{alignat}
For the term of order $\mathcal{O}(\eps)$ we have
\begin{align}\label{Ph1}
		\B{\Phi}_1(\bs,z)\approx
	\begin{cases}
		(\bphi_1)_j(x)+\big(\nabla \B{\varphi}_0\big)_j(x)\,\B{n}_\Gamma(x)\, z&\quad \text{as }z\to +\infty,\\
		(\bphi_1)_i(x)+\big(\nabla \B{\varphi}_0\big)_i(x)\,\B{n}_\Gamma(x)\, z&\quad \text{as }z\to -\infty
	\end{cases}
\end{align}
for all $x = \B{\gamma}(\bs)\in \Gamma=\Gamma_{ij}$. Note that here the symbol $\,\approx\,$ means that the difference of the left-hand side and the right-hand side as well as all its derivatives with respect to $z$ tend to zero as $z\to\pm\infty$. In particular \eqref{Ph1} provides us with
\begin{align}\label{dzP1}
  		\partial_z\B{\Phi}_1(\bs,z)\to
  	\begin{cases}
      	(\nabla \B{\varphi}_0)_j(x) \, \B{n}_\Gamma(x)&\quad \text{as }z\to +\infty,\\
      	(\nabla \B{\varphi}_0)_i(x) \, \B{n}_\Gamma(x)&\quad \text{as }z\to -\infty.
  	\end{cases}
\end{align}
The analogous relations also hold true for the expansions of $\B{w}^\eps_{n_r}$. 

In the following, we will also see that the \textit{jump} across the interfaces $\Gamma_{ij}$ is an important quantity. It is defined by
 \begin{align}
    \label{DEF:JUMP}
 	[\B{v}]_i^j(x)\coloneqq \underset{\delta \searrow 0}{\lim}\; 
 	    \Big( \B{v}\big( x+\delta \B{n}_\Gamma(x) \big)-\B{v}\big(x-\delta \B{n}_\Gamma(x) \big) \Big),
 \end{align}
for any $x = \B{\gamma}(\bs)\in \Gamma=\Gamma_{ij}$.
 
Now, we have made all the necessary computations to analyze the state equations and the gradient equality near the interfaces $\Gamma_{ij}$. In particular, we are able to investigate their behavior as $\eps\to 0$.

%%%%%%%%%%%%%%%%%%%%%%%%%%%%%%%%%%%%
%%%%%%%%%%%%LEADING ORDER%%%%%%%%%%%
%%%%%%%%%%%%%%%%%%%%%%%%%%%%%%%%%%%%

\section{Comparison of the leading order terms}\label{SEC:Leading}

Now, we want to apply our knowledge about the inner and outer expansions to the optimality system consisting of \eqref{state} and \eqref{PhStr}. This means we apply the formulas for the differential operators discussed in Section~\ref{SEC:inner} to the optimality system, compare the terms with same orders in $\eps$ and apply the matching conditions. In this section, we will suppress the index $n_r$ to provide a clearer notation.

\subsection{Comparison of the leading order terms in the state equation}
We point out that our state equation \eqref{state} differs from the one in \cite{Blank} only in terms of the right-hand side. In contrast to \cite{Blank}, where the right-hand side is just a given function $\B{f}$, our right-hand is given by
\begin{align}
    \label{state:rhs}
    \lambda^{\eps,\B{\varphi}}
    \rho(\B{\varphi})\B{w}^{\eps,\B{\varphi}}.
\end{align}
In particular, it depends on the phase-field $\bphi$ as well as the corresponding eigenvalue $\lambda^{\eps,\bphi}$ and its associated eigenfunction $\B{w}^{\eps,\B{\varphi}}$. Recall that  the inner expansion of $\lambda^{\eps,\B{\varphi}}$ equals its outer expansion (cf.~Remark~\ref{REM:EXP:LAMBDA}) as the eigenvalue does not depend locally on $x\in\Omega$.
As no derivatives of $\rho$, $\B{w}^{\eps,\B{\varphi}}$ or $\B{\varphi}$ are involved, we conclude that the inner expansion of \eqref{state:rhs}
possesses only summands of non-negative orders in $\eps$. 
As the discussion of the left-hand side of the state equation works exactly as in \cite{Blank}, we can thus proceed in a completely analogous manner. We will therefore only summarize the most important results.

For the functions $\bW_0$ and $\bW_1$ involved in the inner expansion of the eigenfunction, we deduce the following relations:
\begin{alignat}{2}
    \label{W0match}
    \partial_z\bW_0(\bs,z)&\to \B{0} 
    &&\quad\text{as $z\to\pm\infty$},
    \\
    \label{Windz}
    \partial_z\bW_0&=\B{0} 
    &&\quad\text{around }\Gamma_{ij},
    \\
    \label{e-1}
	\partial_z
		\left[
			\overline{\mathbb{C}}(\B{\Phi}_0)
				\left(
					\partial_z\bW_1\otimes\B{\nu}+\nabla_{\Gamma}\bW_0
				\right)^{\text{sym}}\B{\nu}
		\right]
	&= \B{0}
	&&\quad\text{around }\Gamma_{ij},
\end{alignat}
\begin{align}
    \label{wzer}
		\bW_0(\bs,z)
	&\to
		\begin{cases}
			\left(\B{w}_0\right)_j(x)&\text{as } z\to +\infty,\\
			\left(\B{w}_0\right)_i(x)&\text{as } z\to -\infty,
		\end{cases}
	\\
	\label{W0W1}
		\nabla_{\Gamma}\bW_0(\bs,z)+\partial_z\bW_1(\bs,z)\otimes\B{\nu}(\bs)
	&\to
		\begin{cases}
			\left(\nabla_{x}\B{w}_0\right)_j(x)&\text{as }z\to +\infty,\\
			\left(\nabla_{x}\B{w}_0\right)_i(x)&\text{as }z\to -\infty,
		\end{cases}
\end{align}
for all $x = \B{\gamma}(\bs)\in \Gamma_{ij}$. Here and in the remainder of this paper, the expression ``around $\Gamma_{ij}$'' means that the statement is valid in a sufficiently thin tubular neighborhood around $\Gamma_{ij}$ where our inner expansions hold.
We thus arrive at the jump condition 
\begin{align}
    \left[\B{w}_0\right]_i^j=\B{0} \quad\text{for all $i,j=1,\dots, N$}.
\end{align}
However, we point out that the jump condition on an interface between a material region and a void region (i.e., $i=N$ or $j=N$) is negligible as we do not have any information about the behavior of $\B{w}_0$ in the void. In other words, we will obtain a closed system of PDEs forming the state equations of the sharp-interface problem in Section~\ref{SEC:limitpr} without needing this additional jump condition at the material-void boundary.

For the function $\overline{\mathbb{C}}(\B{\Phi}_0(z))$, where $\B{\Phi}_0$ is the lowest order term of the inner expansion of the phase-field, we obtain:
\begin{align}
    \label{Cmatch}
		\overline{\mathbb{C}}(\B{\Phi}_0(\bs,z))
	&\to
		\begin{cases}
			\mathbb{C}^j&\text{as } z\to +\infty,\\
			\mathbb{C}^i&\text{as } z\to -\infty,
		\end{cases}
	\quad\text{if $i,j \neq N$},
\end{align}
\begin{alignat}{2}
	\begin{aligned}\label{CVaneN}
	    	\overline{\mathbb{C}}(\B{\Phi}_0(\bs,z))
	        & \to 0 \quad\text{as } z\to +\infty,
	        \quad\text{if $j = N$},\\
	    	\overline{\mathbb{C}}(\B{\Phi}_0(\bs,z))
	        & \to 0 \quad\text{as } z\to -\infty,
	        \quad\text{if $i = N$}\\	        
	\end{aligned}
\end{alignat}
Here, the convergence \eqref{CVaneN} follows due to the additional factor $\varepsilon$ in the void contribution of $\overline{\mathbb{C}}(\bphi^\eps)$ (see~\eqref{EXP:CRHO}).

Eventually, we obtain that
\begin{align}\label{CEnu}
		\mathbb{C}^i\mathcal{E}_i(\B{w}_0)\B{n}_\Gamma
	=
		\begin{cases}
			\B{0}&\quad \text{if }j=N,\\
			\mathbb{C}^j\mathcal{E}_j(\B{w}_0)\B{n}_\Gamma &\quad \text{if }j\neq N,
		\end{cases}
% 	\quad\text{for $i\neq N$},
\end{align}
holds on each $\Gamma_{ij}$ with $i\neq N$, where 
\begin{align*}
    \mathcal{E}_i(\B{w}_0)\coloneqq \underset{\delta \searrow 0}{\lim}\; \mathcal{E}(\B{w}_0)(x-\delta \B{n}_{\Gamma})
    \quad\text{and}\quad
    \mathcal{E}_j(\B{w}_0)\coloneqq \underset{\delta \searrow 0}{\lim}\; \mathcal{E}(\B{w}_0)(x+\delta \B{n}_{\Gamma}).
\end{align*}

\subsection{Comparison of the leading order terms in the gradient equality}

Now, we want to analyse the gradient equality \eqref{PhStr}, which reads as

\begin{align}\label{GEeps}
    &\begin{aligned}
    \sum_{r=1}^{l}\Big\{
         [\partial_{\lambda_{n_r}}\hspace{-0.7ex}\Psi]\left(\lambda_{n_1}^{\eps},\dots,\lambda_{n_l}^{\eps}\right)
        &\Big[
             \langle
                 \mathcal{E}(\B{w}_{n_r}^{\eps}),\mathcal{E}(\B{w}_{n_r}^{\eps})
             \rangle_{ P_{T\Sigma}[\mathbb{C}^{\prime}(\bphi^\eps)]}   
    \\[-1ex]
        &\qquad -\lambda_{n_r}^{\eps}
                 \big(
                     \B{w}_{n_r}^{\eps},\B{w}_{n_r}^{\eps}
                 \big)_{P_{T\Sigma}[\rho^{\prime}(\bphi^\eps)]}
         \Big]
    \Big\}
    \end{aligned} 
    \notag\\[1ex]
    &\quad=\gamma\varepsilon \Delta\bphi^\eps
        +\frac{1}{\varepsilon}
        (\B{\Lambda}^\eps+\B{\vartheta}^\eps+\B{\mu}^\eps)       
        -\frac{\gamma}{\varepsilon}P_{T\Sigma}\left[
        \psi_{0}^{\prime}(\bphi^\eps)\right].
\end{align}

Here, we recall that the Lagrange multipliers were constructed in Theorem~\ref{Thm:Lagr} in such a way that their sum appearing in the gradient equality \eqref{GEeps} is scaled by the factor~$\tfrac 1\eps$. 
We now assume the Lagrange multipliers to have the following inner asymptotic expansions:
\begin{align}
\label{EXP:LAG}
    \B{\Lambda}^{\varepsilon}(x)=\sum_{k=0}^{\infty}\varepsilon^{k}\B{\Lambda}_k(\bs,z)\quad
    \B{\vartheta}^{\varepsilon}=\sum_{k=0}^{\infty}\varepsilon^{k}\B{\vartheta}_k,\quad
    \B{\mu}^{\varepsilon}(x)=\sum_{k=0}^{\infty}\varepsilon^{k}\B{\mu}_k(\bs,z),
\end{align}

Furthermore, in order to deal with the nonlinear terms in \eqref{GEeps} involving $\C^\prime$, $\rho^{\prime}, \psi_0^\prime, \partial_{\lambda_{n_r}} \Psi$, we perform a (componentwise) Taylor expansion around the leading order term $\B{\Phi}_0$ to obtain the inner expansions
\begin{align*}
     \C^\prime(\bphi^\eps)&={\C}^\prime(\B{\Phi}_0)+\mathcal{O}(\eps),\\
     \rho^\prime(\bphi^\eps)&={\rho}^\prime(\B{\Phi}_0)+\mathcal{O}(\eps),\\
    \psi_0^\prime(\bphi^\eps)&=\psi_0^\prime(\B{\Phi}_0)+\mathcal{O}(\eps),\\
    [\partial_{\lambda_{n_r}}\hspace{-0.5ex}\Psi]\big(\lambda_{n_1}^\eps,\dots,\lambda_{n_l}^\eps\big)&=
    [\partial_{\lambda_{n_r}}\hspace{-0.5ex}\Psi]\big(\lambda_{0,n_1},\dots,\lambda_{0,n_l}\big)+\mathcal{O}(\eps).
\end{align*}
We now take a closer look at the quantities $\C^\prime(\B{\Phi}_0)$ and $\rho^{\prime}(\B{\Phi}_0)$. To this end, we recall the definition of $\rho$ in \eqref{exrho}, which reads as
\begin{align*}
\rho:\R^N\to \mathbb{R},\quad
\B{\varphi}\mapsto \sum_{i=1}^{N-1}\rho^{i}\big(\sigma_{\omega}\circ\beta_M\big)(P_{\Sigma}(\bphi)^i)+
\tilde{\rho}^N\varepsilon^l \big(\sigma_{\omega}\circ\beta_V\big) (P_{\Sigma}(\bphi)^N).
\end{align*}
Thus, it is clear that $\rho' = \overline{\rho}' +  \mathcal{O}(\eps)$.
Note that we can write the projection $P_{\Sigma}$ as 
\begin{align*}
    P_{\Sigma}(\bphi)=
    \bphi-
    \left(\frac{1}{N}\sum_{i=1}^N\varphi^i\right)\mathbf{1}+
    \frac{1}{N}\mathbf{1},
\end{align*}
for $\bphi\in \R^N$, where $\mathbf{1} = (1,...,1)^T \in \R^N$.
For the partial derivatives with respect to $\varphi^j$ with $j\in\{1,...,N\}$, we thus obtain
\begin{align*}
    (\partial_jP_{\Sigma})(\bphi)=\B{e}_j-\frac{1}{N}\mathbf{1}
\end{align*}
and therefore,
\begin{align*}
    (\partial_j\overline{\rho})(\bphi)=
    \sum_{i=1}^{N-1}\rho^i \big(\sigma_{\omega}^\prime\circ\beta_M\big)(P_{\Sigma}(\bphi)^i)\beta_M^\prime(P_{\Sigma}(\bphi)^i)\left(\delta_{ij}-\frac{1}{N}\right),
\end{align*}
where $\delta_{ij}$ denotes the Kronecker delta. Inserting $\B{\Phi}_0$ (which belongs pointwise to $\B{G}\subset\Sigma^N$ and thus, no projection is necessary) and recalling that $\sigma_{\omega}$ is the identity on $[0,1]$ (cf.~\eqref{cutoff}), using \eqref{eq:def_inter} we arrive at
\begin{align}\label{rhoPPhi}
    \overline{\rho}^{\prime}(\B{\Phi}_0)=((\partial_j\overline{\rho})(\B{\Phi}_0))_{j=1}^{N}=
    \big(\rho^j\beta_M(\phi_0^j)\beta_M^\prime(\phi^j_0)(1-\delta_{Nj})-\frac1N\sum_{i=1}^{N-1}\rho^i\beta_M(\phi^i_0)\beta_M^\prime(\phi^i_0)\big)_{j=1}^{N}.
\end{align}
Thus, considering the inner expansion of $\rho^\prime(\bphi^\eps)$ to the lowest order $\mathcal{O}(1)$ gives
\begin{align}
    \begin{aligned}\label{ovlRPh}
    \rho^\prime(\bphi^\eps)=&\overline{\rho}^\prime(\B{\Phi}_0)\\
    =&\Bigg(\rho^1\beta_M(\phi_0^1)\beta_M^\prime(\phi^1_0)-\frac1N\sum_{i=1}^{N-1}\rho^i\beta_M(\phi^i_0)\beta_M^\prime(\phi^i_0),\;\dots\;,\\
    &\quad\rho^{N-1}\beta_M(\phi_0^{N-1})\beta_M^\prime(\phi^{N-1}_0)-\frac1N\sum_{i=1}^{N-1}\rho^i\beta_M(\phi^i_0)\beta_M^\prime(\phi^i_0),\\
    &\quad-\frac1N\sum_{i=1}^{N-1}\rho^i\beta_M(\phi^i_0)\beta_M^\prime(\phi^i_0)\Bigg)^T.
    \end{aligned}
\end{align}
Thus, it obviously holds $\overline{\rho}^\prime(\B{\Phi}_0)\in T\Sigma^N$. 
The function $\C^\prime(\B{\Phi}_0)$ can be expressed analogously.
Altogether, this allows us to drop the projection acting on the left-hand side in \eqref{GEeps} when considering only the lowest order contributions.

Now that we have considered all the quantities appearing in \eqref{GEeps}, we begin with our formal asymptotics.
First of all, applying formula \eqref{Jb} on the lowest order contribution $\bW_0$ of the inner expansion of $\B{w}^\eps$, we find that
\begin{equation}\label{symW}
 		\mathcal{E}(\B{w}^\eps)
 	=
 		\big(\nabla_x\B{w}^\eps\big)^{\mathrm{sym}}
 	=
 		\left(
             \nabla_{\Gamma}\bW_0+\frac{1}{\varepsilon}\partial_{z}\bW_0\otimes\B{\nu}
         \right)^{\mathrm{sym}} + \mathcal{O}(\eps).
\end{equation}
Comparing the contributions of order $\mathcal{O}(\eps^{-2})$ in \eqref{GEeps}, we use \eqref{symW} to obtain 
\begin{align*}
    \B{0}=
    \sum_{r=1}^l
    &[\partial_{\lambda_{n_r}}\hspace{-0.7ex}\Psi]\left(\lambda_{0,n_1},\dots,\lambda_{0,n_l}\right)
    \\
    &\cdot \Big[\overline{\C}^\prime(\B{\Phi}_0)
    \big(\partial_z\bW_{0,n_r}\otimes\B{\nu}\big)^{\text{sym}}:
	\big(\partial_z\bW_{0,n_r}\otimes\B{\nu}\big)^{\text{sym}}\Big],
	%\quad\text{around }\Gamma_{ij}.
\end{align*}
This equation is obviously fulfilled since $\partial_z\bW_0$ vanishes according to \eqref{Windz}. 

Let us now consider \eqref{GEeps} to order $\mathcal{O}(\eps^{-1})$. First of all, we infer from \eqref{Windz} and \eqref{symW} that the left-hand side has no contribution of order $\mathcal{O}(\eps^{-1})$. We thus have
\begin{alignat}{2}\label{defPh}
\begin{aligned}
    \B{0}=
    \gamma \partial_{zz}\B{\Phi}_0
    +(\B{\Lambda}_0+\B{\vartheta}_0+\B{\mu}_0)
    -\gamma P_{T\Sigma}\left[\psi_0^{\prime}(\B{\Phi}_0)\right],
\end{aligned}
\end{alignat}
where we used the formula \eqref{Laplace} to compute the Laplacian.
Multiplying \eqref{defPh} by $\partial_z\B{\Phi}_0$ and integrating with respect to $z$ from $-\infty$ to $\infty$, we deduce
\begin{alignat}{2}\label{0dzP}
\begin{aligned}
	&
    -\int_{-\infty}^{\infty}
        (\B{\Lambda}_0+\B{\vartheta}_0+\B{\mu}_0)\cdot\partial_{z}\B{\Phi}_0
        \mathrm{\,d}z\\
	&\quad =
            \gamma\int_{-\infty}^{\infty}
           \partial_{zz}\B{\Phi}_0\cdot\partial_{z}\B{\Phi}_0
             \mathrm{\,d}z
		-\gamma
            \int_{-\infty}^{\infty}
             P_{T\Sigma}\left[\psi_0^{\prime}(\B{\Phi}_0)\right]\partial_{z}\B{\Phi}_0
             \mathrm{\,d}z.
\end{aligned}             
\end{alignat}
Now, we consider each of the terms in \eqref{0dzP} separately. First of all, we see
\begin{align}\label{PIPh}
    \begin{aligned}
        &\int_{-\infty}^{\infty}
                \partial_{zz}\B{\Phi}_0\cdot\partial_{z}\B{\Phi}_0
            \mathrm{\,d}z
        =
            \int_{-\infty}^{\infty}
                \frac{1}{2}\frac{\text{d}}{\text{d}z}
                \left\vert\partial_{z}\B{\Phi}_0\right\vert^2
            \mathrm{\,d}z
        \\
        &\quad=
            \frac{1}{2}
            \left(\underset{z\to +\infty}{\lim}\partial_z\B{\Phi}_0(z)-
            \underset{z\to -\infty}{\lim}\partial_z\B{\Phi}_0(z)\right)
        =\B{0}, 
    \end{aligned}       
\end{align}
where the last equality follows from the matching condition \eqref{Ph0}.

\pagebreak[2]

As $\B{\Phi}_0\in \B{G}$ pointwise, we know that $\partial_{z}\B{\Phi}_0\in T\Sigma^N$ pointwise. Hence, we obtain
\begin{align}
\label{PTZPsi0}
\begin{aligned}
	&\int_{-\infty}^{\infty}
            P_{T\Sigma}\left[\psi_0^{\prime}(\B{\Phi}_0)\right]\partial_{z}\B{\Phi}_0
         \mathrm{\,d}z
	=
		\int_{-\infty}^{\infty}
          \psi_0^{\prime}(\B{\Phi}_0)\partial_{z}\B{\Phi}_0
        \mathrm{\,d}z
    \\
	&=
        \int_{-\infty}^{\infty}
            \frac{\text{d}}{\text{d}z}\left[\psi_0(\B{\Phi}_0)\right]
         \mathrm{\,d}z
	=\underset{z\to +\infty}{\lim}
             \psi_0(\B{\Phi}_0(z))
          -\underset{z\to -\infty}{\lim}
             \psi_0(\B{\Phi}_0(z))
    =\B{0}.
\end{aligned}
\end{align}
For the last equality, we used the fact that $\psi_0$ vanishes on $\B{e}_i$ for $i=1,\dots,N$ along with the matching condition \eqref{Ph0}. We have thus shown that the right-hand side of \eqref{0dzP} vanishes.

Recall from \eqref{LamEqComp} that $\B{\Lambda}^{\eps}$ is identical in each component. It is therefore natural to assume that every term in the inner expansion of $\B{\Lambda}^\eps$ also has this property. Thus, recalling that $\partial_{z}\B{\Phi}_0\in T\Sigma^N$ pointwise, we infer
\begin{equation}\label{LamPhi0}
        \int_{-\infty}^{\infty}
            \B{\Lambda}_0\cdot\partial_z\B{\Phi}_0
        \mathrm{\,d}z
    =
        \int_{-\infty}^{\infty}
            \Lambda_0\sum_{i=1}^{N}[\partial_z\B{\Phi}_0]^i
        \mathrm{\,d}z
   =0,    
\end{equation}
where $\Lambda_0$ denotes an arbitrary component of $\B{\Lambda}_0$.

Recall from Theorem~\ref{Thm:Lagr} that $\B{\vartheta}^{\eps}\in \R^N$ is constant. Thus, assuming that this property is transferred to the inner expansion, $\B{\vartheta}_0$ is independent of $z$, we infer by means of the matching condition \eqref{Ph0} that
\begin{align}\label{PImu}
    	\int_{-\infty}^{\infty}
            \B{\vartheta}_0\cdot\partial_{z}\B{\Phi}_0
         \mathrm{\,d}z
      =
    	\int_{-\infty}^{\infty}
            \frac{\text{d}}{\text{d}z}[\B{\vartheta}_0\cdot\B{\Phi}_0]
         \mathrm{\,d}z        
      =
        \B{\vartheta}_0\cdot\left(\B{e}_j-\B{e}_i\right).
\end{align}

Eventually, we want to justify that the remaining Lagrange multiplier fulfills
\begin{align}
	\label{mu0}
    \int_{-\infty}^{\infty}\B{\mu}_0\cdot\partial_{z}\B{\Phi}_0\mathrm{\,d}z=0.
\end{align}
Therefore, we recall \eqref{muphiVan0} which tells us for $i=1,\dots,N$ that 
\begin{align*}
    \mu_i^{\eps}=0 \quad\text{a.e.~in}\;\; \Omega_{i}^{+}
    =\big\{\bx\in \Omega \,\big|\, \varphi_i^{\eps}(\bx)>0\big\}
    =\Omega\backslash \big\{\bx\in \Omega\ \,\big|\, \varphi_i^{\eps}(\bx)=0\big \}.
\end{align*}
Using \cite[Lemma 7.7]{Gilbarg}, we infer that for all $i\in\{1,\dots,N\}$,
\begin{align}\label{GilVan}
    \mu_i^{\eps}\, \nabla_x\varphi_i^{\eps}=\B{0}. 
\end{align}
Using \eqref{gradx} and comparing the terms of order $\mathcal{O}(\eps^{-1})$, we deduce
\begin{align}
\label{EQ:MUDZPHI}
    \mu_0^i\, \partial_z\Phi^i_0\,\B{\nu}=\B{0} 
\end{align}
for all $i\in\{1,\dots,N\}$.
In particular, by multiplying with $\B{\nu}$ and integrating with respect to $z$ from $-\infty$ to $\infty$, we arrive at
\begin{align*}
    \int_{-\infty}^{\infty}\mu_0^i(z)\partial_z\Phi^i_0(z)\mathrm{\,d}z=0.
\end{align*}
for all $i\in\{1,\dots,N\}$.
This proves \eqref{mu0}.

Combining \eqref{PIPh}--\eqref{mu0}, we conclude from \eqref{0dzP} that
\begin{align*}
        \B{\vartheta}_0\cdot(\B{e}_j-\B{e}_i)=0,
\end{align*}
for all $i,j=1,\dots,N$, meaning that all components of $\B{\vartheta}_0$ are equal. Since $\B{\vartheta}^\eps\in T\Sigma^N$ in \eqref{thetaSum0}, we also assume $\B{\vartheta}_0\in T\Sigma^N$. This implies that $\B{\vartheta}_0 = \B{0}$ and thus, \eqref{defPh} can be rewritten as
\begin{align}\label{Pheq}
	\B{0}=-\gamma\partial_{zz}\B{\Phi}_0+\gamma P_{T\Sigma}\left[\psi_0^{\prime}(\B{\Phi}_0)\right]-\B{\Lambda}_0-\B{\mu}_0.
\end{align}
Let now $\tilde{z}\in \R$ be arbitrary. Multiplying \eqref{Pheq} by $\partial_{z}\B{\Phi}_0$ and integrating with respect to $\tilde{z}$ from $-\infty$ to $\tilde{z}$, we obtain 
\begin{align*}
\int_{0}^{\tilde{z}}\frac{1}{2}\frac{\text{d}}{\text{d}z}\left\vert\partial_{z}\B{\Phi}_0\right\vert^2\mathrm{\,d}z=
\int_{0}^{\tilde{z}}\frac{\text{d}}{\text{d}z}\left[\psi_0(\B{\Phi}_0)\right]\mathrm{\,d}z
-\frac{1}{\gamma}
    \int_{0}^{\tilde{z}}
        \left(\B{\Lambda}_0+\B{\mu}_0\right)\cdot\partial_{z}\B{\Phi}_0
    \mathrm{\,d}z.
\end{align*}
Here, the last equality holds because of \eqref{LamPhi0} and \eqref{mu0}. By the fundamental theorem of calculus, we thus have
\begin{align*}
   \left\vert\partial_{z}\B{\Phi}_0(\tilde{z})\right\vert^2-2\psi_0(\B{\Phi}_0(\tilde{z}))=
    \left\vert\partial_{z}\B{\Phi}_0(0)\right\vert^2-2\psi_0(\B{\Phi}_0(0))
\end{align*}
for all $\tilde{z}\in\R$. We further know from the matching condition \eqref{Ph0} that the left-hand side vanishes as $\tilde{z}\to\pm \infty$. This entails
\begin{align}\label{PhiProfile:0}
    \left\vert\partial_{z}\B{\Phi}_0(0)\right\vert^2-2\psi_0(\B{\Phi}_0(0))=0,
\end{align}
and thus, we arrive at
\begin{align}
    \label{PhiProfile}
    \left\vert\partial_{z}\B{\Phi}_0({z})\right\vert^2=2\psi_0(\B{\Phi}_0({z})) \quad\text{for all } z\in\R.   
\end{align}
In order to obtain further information, we next show that \eqref{Pheq} can be interpreted as the first-order optimality condition of a particular optimization problem that is similar to the minimization of the one-dimensional Ginzburg--Landau energy.
Therefore, we first assume that 
\begin{align}
\label{DEF:SIGMAIJ}
   \sigma_{ij}\coloneqq\inf\left\{
    \int_{-1}^{1}\sqrt{2\psi_0(\B{\theta}(t))}\abs{\B{\theta}^\prime(t)}\mathrm{\,d}t
    \suchthat 
    \begin{aligned}
    &\B{\theta}\in C^{0,1}([0,1];\R^N),\;
    \B{\theta}\in \B{G} \text{ pointwise},
    \\
    &\B{\theta}(1)=\B{e}_j
    \;\;\text{and}\;\;
    \B{\theta}(-1)=\B{e}_i 
    \end{aligned}
   \right\}
\end{align}
possesses a minimizer, which we call $\B{\theta}_{ij}$. This means that $\B{\theta}_{ij}$ is a geodesic with respect to the degenerate metric induced by the potential $\psi_0$ that connects the values $\B{e}_i$ and $\B{e}_j$. Now, proceeding as in \cite[proof of formula (15)]{Sternberg}, this geodesic can be used to construct a minimizer $\B{\Phi}$ of the problem
\begin{align}\label{OptTrans}
    \inf\left\{
        \int_{-\infty}^{+\infty}\abs{\partial_z \B{\Phi}}^2+2\psi_0(\B{\Phi})\mathrm{\,d}z
    \suchthat
    \begin{aligned}
    &\B{\Phi}\in C^{0,1}([0,1];\R^N),\;
    \B{\Phi}\in \B{G} \text{ pointwise},
    \\
    &\underset{z\to \infty}{\lim}\B{\Phi}(z)=\B{e}_j
    \;\;\text{and}\;\;
    \underset{z\to -\infty}{\lim}\B{\Phi}(z)=\B{e}_i
    \end{aligned}
    \right\}.
\end{align}
This means that $\B{\Phi}$ describes an optimal transition between the values $\B{e}_i$ and $\B{e}_j$.
As in \cite[proof of formula (15)]{Sternberg}, we further see that $\B{\Phi}$ solves \eqref{Pheq} and \eqref{PhiProfile}, where $\B{\Lambda}_0+\B{\mu}_0$ is the Lagrange multiplier for the Gibbs--Simplex constraint. 
Consequently, choosing $\B{\Phi}_0=\B{\Phi}$ we have found a solution of \eqref{Pheq} and \eqref{PhiProfile}. Moreover, \cite[formula (15)]{Sternberg} states that $2\sigma_{ij}$ is exactly the value of the minimum sought in~\eqref{OptTrans}. 

As the minimizer $\B{\Phi}_0=\B{\Phi}$ of \eqref{OptTrans} satisfies \eqref{PhiProfile}, we further conclude
\begin{align}\label{Phrel}
    \sigma_{ij}
	=	\int_{-\infty}^{\infty}\abs{\partial_z\B{\Phi}_0}^2\mathrm{\,d}z
	=
		2\int_{-\infty}^{\infty}\psi_0(\B{\Phi}_0)\mathrm{\,d}z<\infty,
\end{align}
which will be important for later purposes.
\pagebreak[2]

Finally, we now consider \eqref{GEeps} to the order $\mathcal{O}\left(1\right)$.
Using \eqref{Laplace} to reformulate the term $\gamma \varepsilon \Delta\B{\varphi}$, employing \eqref{Windz}, and recalling that \eqref{ovlRPh} holds analogously for $\overline{\C}^\prime(\B{\Phi}_0)$, we conclude

\begin{alignat}{2}\label{P1Eq}
	\begin{aligned}
	    &
        \frac{1}{\gamma}(\B{\Lambda}_1+\B{\vartheta}_1+\B{\mu}_1)
	 		+\partial_{zz}\B{\Phi}_1
	 		-P_{T\Sigma}\left[\psi_0^{\prime\prime}(\B{\Phi}_0)\B{\Phi}_1\right]
	 	\\
	 	&\quad=
	 		\hat{\kappa}\partial_{z}\B{\Phi}_0+
	 		\frac{1}{\gamma}
	 				\sum_{r=1}^{l}
	 					\Big\{[\partial_{\lambda_{n_r}}\hspace{-0.7ex}\Psi](\lambda_{0,n_1},\dots,\lambda_{0,n_l})
	 	\\
	 	&\qquad	            \cdot\Big[
	 								\overline{\mathbb{C}}^{\prime}(\B{\Phi}_0)
 									(\nabla_{\Gamma}\bW_{0,n_r}+\partial_z\bW_{1,n_r}\otimes\B{\nu})^{\text{sym}}
 									:(\nabla_{\Gamma}\bW_{0,n_r}+\partial_z\bW_{1,n_r}\otimes\B{\nu})^{\text{sym}}
	 	\\
	 	&\qquad\qquad
	 								-\lambda_{0,n_r} \overline{\rho\,}^{\,\prime}(\B{\Phi}_0)\abs{\bW_{0,n_r}}^2
	 						\Big]
	 					\Big\}.
	\end{aligned}
\end{alignat}
We now multiply this equation by $\partial_{z}\B{\Phi}_0$ and integrate with respect to $z$ from $-\infty$ to $\infty$. Let us consider each term of the resulting equation separately.
Analogously to \eqref{LamPhi0} and \eqref{PImu}, we obtain
\begin{align}
    \label{EQ:LAM:1}
    \int_{-\infty}^{\infty}\B{\Lambda}_1\cdot\partial_{z}\B{\Phi}_0\mathrm{\,d}z&=0\quad\text{and}
    \\
    \label{EQ:THETHA:1}
	\int_{-\infty}^{\infty}\B{\vartheta}_1\cdot\partial_{z}\B{\Phi}_0\mathrm{\,d}z
	&=\B{\vartheta}_1\cdot\left(\B{e}_j-\B{e}_i\right). 
\end{align}
Considering the Lagrange multiplier $\B{\mu}$, we recall \eqref{GilVan}        
\begin{align}
    \mu_i^{\eps}\, \nabla_x\varphi_i^{\eps}=\B{0}.
\end{align}    
Due to \eqref{gradx}, its contribution of leading order $\mathcal{O}(1)$ in inner coordinates is given by
\begin{align*}
    \mu_1^i\partial_z\Phi_0^i\B{\nu}+
    \mu_0^i\nabla_{\Gamma}\Phi_0^i+
    \mu_0^i\partial_z\Phi_1^i\B{\nu}
    =\B{0}
\end{align*}
for all $i\in\{1,\dots,N\}$.
Multiplying this identity by $\B{\nu}$ and integrating the resulting equation with respect to $z$, we infer 
\begin{align}
    \label{EQ:MU:1}
    \int_{-\infty}^{\infty}\B{\mu}_1\cdot\partial_{z}\B{\Phi}_0\mathrm{\,d}z
    &=-\int_{-\infty}^{\infty}\B{\mu}_0\cdot\partial_{z}\B{\Phi}_1\mathrm{\,d}z.
\end{align}

Furthermore, applying integration by parts twice and using that due to the matching condition \eqref{Ph0} all derivatives of $\B{\Phi}_0$ with respect to $z$ tend to $0$ as $z\to \pm\infty$, we obtain
\begin{align}
    \label{EQ:DZZPHI:1}
		\int_{-\infty}^{\infty}\partial_{zz}\B{\Phi}_1\cdot\partial_{z}\B{\Phi}_0\mathrm{\,d}z
	=
		\int_{-\infty}^{\infty}\partial_{zz}\left(\partial_{z}\B{\Phi}_0\right)\cdot\B{\Phi}_1\mathrm{\,d}z.
\end{align}
As $\partial_{z}\B{\Phi}_0$ attains its values only in $T\Sigma^N$, we deduce
\begin{align}
    \label{EQ:PTZPHI:1}
		\int_{-\infty}^{\infty}
			P_{T\Sigma}\left[\psi_0^{\prime\prime}(\B{\Phi}_0)\B{\Phi}_1\right]\cdot\partial_z\B{\Phi}_0
		\mathrm{\,d}z
	&=
		\int_{-\infty}^{\infty}
			\psi_0^{\prime\prime}(\B{\Phi}_0)\, \partial_z\B{\Phi}_0\cdot\B{\Phi}_1
		\mathrm{\,d}z
\end{align}
due to the symmetry of the Hessian matrix.
Moreover, recalling that $\bW_{0}$ is independent of $z$ due to \eqref{Windz}, a simple computation yields
\begin{align}
    \label{EQ:RHOPHI:1}
\begin{aligned}
        \int_{-\infty}^{\infty}
            \overline{\rho}^{\prime}(\B{\Phi}_0)\partial_z\B{\Phi}_0\abs{\bW_0}^2
        \mathrm{\,d}z
    &=
        \int_{-\infty}^{\infty}
            \left[\frac{\text{d}}{\text{d}z}\overline{\rho}(\B{\Phi}_0)\right]\abs{\bW_0}^2
        \mathrm{\,d}z\\
    &=
        \int_{-\infty}^{\infty}
            \frac{\text{d}}{\text{d}z}\left[\overline{\rho}(\B{\Phi}_0)\abs{\bW_0}^2\right]
        \mathrm{\,d}z.  
\end{aligned}
\end{align}
Furthermore, by the definition of the dyadic product, it holds
\begin{alignat*}{2}
	&
		\left(\nabla_{\Gamma}\bW_0+\partial_{z}\bW_1\otimes\B{\nu}\right)^{\text{sym}}
		\B{\nu} \cdot \partial_{zz}\bW_1\\
	&\quad = 
		\left(\nabla_{\Gamma}\bW_0+\partial_{z}\bW_1\otimes\B{\nu}\right)^{\text{sym}}:
		\left(\partial_{zz}\bW_1\otimes\B{\nu}\right)^{\text{sym}}.
\end{alignat*}
Now we use \eqref{Windz} (which directly entails $\partial_z\nabla_{\Gamma}\bW_0=\B{0}$), \eqref{e-1} and $\partial_{z}\B{\nu}=\B{0}$ to deduce
\begin{align}
\label{EQ:CPHI:1}
\begin{aligned}
	&
		\int_{-\infty}^{\infty}
        \overline{\mathbb{C}}^{\prime}(\B{\Phi}_0)\partial_z\B{\Phi}_0
		\left(\nabla_{\Gamma}\bW_0+\partial_{z}\bW_1\otimes\B{\nu}\right)^{\text{sym}}:
		\left(\nabla_{\Gamma}\bW_0+\partial_{z}\bW_1\otimes\B{\nu}\right)^{\text{sym}}
	\mathrm{\,d}z\\
	&\quad=
		\int_{-\infty}^{\infty}
		\left[\frac{\text{d}}{\text{d}z}     
            \overline{\mathbb{C}}(\B{\Phi}_0)
        \right]   
		\left(\nabla_{\Gamma}\bW_0+\partial_{z}\bW_1\otimes\B{\nu}\right)^{\text{sym}}:
		\left(\nabla_{\Gamma}\bW_0+\partial_{z}\bW_1\otimes\B{\nu}\right)^{\text{sym}}
	\mathrm{\,d}z\\
	&\quad=
		\int_{-\infty}^{\infty}
			\frac{\text{d}}{\text{d}z}
				\left[
					\overline{\mathbb{C}}(\B{\Phi}_0)
					\left(\nabla_{\Gamma}\bW_0+\partial_{z}\bW_1\otimes\B{\nu}\right)^{\text{sym}}:
					\left(\nabla_{\Gamma}\bW_0+\partial_{z}\bW_1\otimes\B{\nu}\right)^{\text{sym}}
				\right]
		\mathrm{\,d}z\\
	&\qquad
		-2\int_{-\infty}^{\infty}
			\frac{\text{d}}{\text{d}z}
				\left[
					\overline{\mathbb{C}}(\B{\Phi}_0)
					\left(\nabla_{\Gamma}\bW_0+\partial_{z}\bW_1\otimes\B{\nu}\right)^{\text{sym}}
					\B{\nu}\cdot\partial_{z}\bW_1
				\right]
			\mathrm{\,d}z
\end{aligned}
\end{align}
by means of the product rule and integration by parts.
\pagebreak[2]

Collecting \eqref{EQ:LAM:1}--\eqref{EQ:CPHI:1} and recalling \eqref{Phrel}, we eventually obtain
\begin{alignat}{2}\label{O1InnerInt}
\begin{aligned}
	& \B{\vartheta}_1\cdot\left(\B{e}_j-\B{e}_i\right)
	\\
	&\quad +\int_{-\infty}^{\infty}
			\left(
				\partial_{zz}\left(\partial_{z}\B{\Phi}_0\right)-\psi_0^{\prime\prime}\left(\B{\Phi}_0\right)\partial_{z}\B{\Phi}_0
			\right)\cdot \B{\Phi}_1\mathrm{\,d}z
			-\int_{-\infty}^{\infty}\B{\mu}_0\cdot\partial_{z}\B{\Phi}_1\mathrm{\,d}z\\
	&=
		-\frac{1}{\gamma}
			\sum_{r=1}^{l}
                    [\partial_{\lambda_{n_r}}\hspace{-0.7ex}\Psi]
					\left(\lambda_{0,n_1},\dots,\lambda_{0,n_l}\right)
                    \lambda_{0,n_r}
						\left(
                            \int_{-\infty}^{\infty}
                                \frac{\text{d}}{\text{d}z}
                                \left(
                                    \overline{\rho}(\B{\Phi}_0)\abs{\bW_{0,n_r}}^2
                                \right)
						\mathrm{\,d}z\right)
	\\
	&\quad
		+ \sigma_{ij}\hat{\kappa} + \frac{1}{\gamma}
			\sum_{r=1}^{l}
			\Bigg\{
                    [\partial_{\lambda_{n_r}}\hspace{-0.7ex}\Psi]
				    \left(\lambda_{0,n_1},\dots,\lambda_{0,n_l}\right)
    \\
    &\hspace{15ex}                
                    \cdot\Bigg[
                    \int_{-\infty}^{\infty}
                    \frac{\text{d}}{\text{d}z}
								\Big(
									\overline{\mathbb{C}}(\B{\Phi}_0)
										(...)^{\text{sym}}
                                        :(...)^{\text{sym}}
								\Big)
                                \mathrm{\,d}z
	\\
	&\hspace{15ex}\quad  
                            -2 \int_{-\infty}^{\infty}
                            \frac{\text{d}}{\text{d}z}
									\left(
										\overline{\mathbb{C}}(\B{\Phi}_0)
											(...)^{\text{sym}}
											\B{\nu}\cdot\partial_{z}\bW_{1,n_r}
									\right)
                                    \mathrm{\,d}z
                    \Bigg]
			\Bigg\}
\end{aligned}
\end{alignat}
on $\Gamma_{ij}$, where $(...)^{\text{sym}}$ abbreviates $\big(\nabla_{\Gamma}\bW_{0,n_r}+\partial_{z}\bW_{1,n_r}\otimes\B{\nu}\big)^{\text{sym}}$.
Next, we want to show that
\begin{align}\label{intvan}
	\int_{-\infty}^{\infty}
		\big(
			\partial_{zz}(\partial_{z}\B{\Phi}_0)-\psi_0^{\prime\prime}(\B{\Phi}_0)\partial_{z}\B{\Phi}_0
		\big)\cdot \B{\Phi}_1\mathrm{\,d}z
		-\int_{-\infty}^{\infty}\B{\mu}_0\cdot\partial_{z}\B{\Phi}_1\mathrm{\,d}z 
	= 0.
\end{align}
Differentiating \eqref{Pheq} with respect to $z$, multiplying by $\B{\Phi}_1$ and integrating the resulting equation with respect to $z$, we deduce
\begin{align*}
        \int_{-\infty}^{\infty}
		\big(
			\partial_{zz}(\partial_{z}\B{\Phi}_0)-\psi_0^{\prime\prime}(\B{\Phi}_0)\partial_{z}\B{\Phi}_0
		\big)\cdot \B{\Phi}_1\mathrm{\,d}z=
		-\int_{-\infty}^{\infty}
		\left[\partial_z(\B{\Lambda}_0+\B{\mu}_0)\right]\cdot\B{\Phi}_1\mathrm{\,d}z.
\end{align*}
Thus, in order to prove \eqref{intvan}, it suffices to show
\begin{align}\label{intvan2}
  		\int_{-\infty}^{\infty}
		\left[\partial_z(\B{\Lambda}_0+\B{\mu}_0)\right]\cdot\B{\Phi}_1\mathrm{\,d}z+
		\int_{-\infty}^{\infty}\B{\mu}_0\cdot\partial_{z}\B{\Phi}_1\mathrm{\,d}z=0.
\end{align}
By means of the product rule, the left-hand side can be reformulated as
\begin{align*}
    \int_{-\infty}^{\infty}
		\left[\partial_z(\B{\Lambda}_0+\B{\mu}_0)\right]\cdot\B{\Phi}_1\mathrm{\,d}z
	&=
    -\int_{-\infty}^{\infty}
	(\B{\Lambda}_0+\B{\mu}_0)\cdot\partial_z\B{\Phi}_1\mathrm{\,d}z
	\\
	&\qquad+
	\int_{-\infty}^{\infty}
	\frac{\text{d}}{\text{d}z}\left[(\B{\Lambda}_0+\B{\mu}_0)\cdot\B{\Phi}_1\right]\mathrm{\,d}z.
\end{align*}
Now as $\B{\Phi}_1,\partial_z\B{\Phi}_1 \in T\Sigma^N$ pointwise, we know as in \eqref{LamPhi0}
\begin{align*}
    \B{\Lambda}_0\cdot\partial_z\B{\Phi}_1=
    \B{\Lambda}_0\cdot\B{\Phi}_1=
    0.
\end{align*}
Thus, as we want to prove \eqref{intvan2}, it remains to show
\begin{align}\label{dzmu0}
   	\int_{-\infty}^{+\infty}
	\frac{\text{d}}{\text{d}z}\left[\B{\mu}_0\cdot\B{\Phi}_1\right]\mathrm{\,d}z=0. 
\end{align}
Recalling once more formula \eqref{muphiVan0}, we infer 
\begin{align*}
    [\mu^{\eps}]^i [\varphi^{\eps}]^i =0
    % \quad\text{in $\Omega$}
\end{align*}
for all $i\in \{1,\dots,N\}$.
Hence, for the $\mathcal{O}(1)$-contribution and the $\mathcal{O}(\eps)$-contribution of the inner expansions, we obtain the relations
\begin{align}\label{muPhiO}
    \mu_0^i(z)\,\Phi^i_0(z)=0
    \quad\text{and}\quad
    \mu_1^i(z)\,\Phi^i_0(z)=-\mu_0^i(z)\,\Phi^i_1(z),
\end{align}
respectively, for all $i\in\{1,\dots,N\}$ and $z\in\R$.
Now, the first equation in \eqref{muPhiO} implies that for any $z\in\R$ with $\Phi_0^i(z)\neq 0$, we have $\mu_0^i(z)=0$ and thus also $\mu_0^i(z)\Phi_1^i(z)=0$. 
On the other, for all $z\in\R$ with $\Phi_0^i(z)= 0$, we infer from the second equation in \eqref{muPhiO} that $\mu_0^i(z)\Phi_1^i(z)=0$. 
Combining both statements, we conclude
\begin{align*}
    \mu_0^i(z)\, \Phi_1^i(z)=0
    \quad\text{for all $i\in\{1,\dots,N\}$ and $z\in\R$.}
\end{align*}
This proves \eqref{dzmu0}. By the above considerations, this verifies \eqref{intvan2} which in turn implies equation \eqref{intvan}.
 
To conclude this section, we recall the definition of the \textit{jump}, see \eqref{DEF:JUMP}. Moreover, we recall from \eqref{RepCurv} that the mean curvature of $\Gamma_{ij}$ is given by $\kappa_{ij} = - \nabla_{\Gamma_{ij}} \cdot \B{n}_{\Gamma_{ij}}$.
Using the matching conditions \eqref{dzP1}, \eqref{wzer} and \eqref{W0W1}, we finally infer from \eqref{O1InnerInt} that
\begin{align}\label{SCond}
\begin{aligned}
    \big(\vartheta_1^j-\vartheta_1^i\big)
	&= \sigma_{ij}\kappa_{ij}
		-\frac{1}{\gamma}\sum_{j=1}^{l}
        [\partial_{\lambda_{n_r}}\hspace{-0.7ex}\Psi]
			\left(\lambda_{0,n_1},\dots,\lambda_{0,n_l}\right)
					\lambda_{0,n_r}
					\left[\overline{\rho\,}\abs{\B{w}_{0,n_r}}^2\right]_i^{j}
	\\
	&\quad
		+\frac{1}{\gamma}
			\sum_{r=1}^{l}
            \Bigg\{ [\partial_{\lambda_{n_r}}\hspace{-0.7ex}\Psi]
				\left(\lambda_{0,n_1},\dots,\lambda_{0,n_l}\right)
	\\
	&\qquad\cdot 
			\Big(\big[
					\overline{\mathbb{C}}
                    \mathcal{E}(\B{w}_{0,n_r}):\mathcal{E}(\B{w}_{0,n_r})\big]_i^j
			-2
					\big[\overline{\mathbb{C}}\mathcal{E}(\B{w}_{0,n_r})\B{\nu}\cdot\nabla\B{w}_{0,n_r}\B{\nu}\big]_i^j
			\Big)
			\Bigg\}
    \end{aligned}
\end{align}
on $\Gamma_{ij}$ for all $i,j\in\{1,\dots,N-1\}$. 
In the case $j=N$ and $i\neq N$, equation \eqref{O1InnerInt} simplifies to
\begin{alignat}{2}\label{TP}
    \begin{aligned}
            \big(\vartheta_1^j-\vartheta_1^i\big)
        &=&&
            \sigma_{ij}\kappa_{ij}
            +\frac{1}{\gamma}\sum_{r=1}^{l}[\partial_{\lambda_{n_r}}\hspace{-0.7ex}\Psi]
            \left(\lambda_{0,n_1},\dots,\lambda_{0,n_l}\right)
                \lambda_{0,n_r}
                \; \rho^i\abs{(\B{w}_{0,n_r})_i}^2\\
        &\phantom{=}&&
            -\frac{1}{\gamma}
            \sum_{r=1}^{l}[\partial_{\lambda_{n_r}}\hspace{-0.7ex}\Psi]
                \left(\lambda_{0,n_1},\dots,\lambda_{0,n_l}\right)
                    \mathbb{C}^i\mathcal{E}_i(\B{w}_{0,n_r}):\mathcal{E}_i(\B{w}_{0,n_r})
    \end{aligned}
\end{alignat}
on $\Gamma_{iN}$ by the matching in \eqref{CVaneN} and \eqref{CEnu}.

To conclude this section we note that according to \cite[Section 5.3]{Blank} or \cite[Section 2.4]{Bronsard} equation \eqref{Pheq} induces a further solvability condition, namely an angle condition for triple junctions. To see this, let us assume that the regions $\Omega_i$, $\Omega_j$, $\Omega_k$ (with $i,j,k$ pairwise different) meet at a triple point $m_{ijk}$ in the 2-dimensional case or on a triple curve
$m_{ijk}$ in the 3-dimensional case.
Then the angle condition is expressed via the normals of the three meeting interfaces as follows
\begin{align}\label{Angle}
	\sigma_{ij}\B{n}_{\Gamma_{ij}}+\sigma_{jk}\B{n}_{\Gamma_{jk}}+\sigma_{ki}\B{n}_{\Gamma_{ki}}=0\quad\text{in }m_{ijk}.
\end{align}
This relation immediately implies the angle condition
\begin{align*}
	\frac{\sin(\theta_{ij})}{\sigma_{ij}}=
	\frac{\sin(\theta_{jk})}{\sigma_{jk}}=
	\frac{\sin(\theta_{ki})}{\sigma_{ki}},
\end{align*}
where $\theta_{ij}$ denotes the angle between $\B{n}_{\Gamma_{jk}}$ and $\B{n}_{\Gamma_{ki}}$.
For the choice $\psi_0(\bphi)=\frac{1}{2}(1-\bphi\cdot\bphi)$ we deduce that the transition energies denoted by $\sigma_{ij},\sigma_{jk},\sigma_{ki}$ are always equal. This follows by exploiting the symmetry of this potential in \eqref{DEF:SIGMAIJ}. Thus, in this case, we know that triple junctions always occur at a $120^\circ$ contact angle.

%%%%%%%%%%%%%%%%%%%%%%%%%%%%%%%%%%%%
%%%%%%%%%%%%LIMIT PROBLEM%%%%%%%%%%%
%%%%%%%%%%%%%%%%%%%%%%%%%%%%%%%%%%%%

\section{The sharp-interface problem}\label{SEC:limitpr}
Now we are in a position to state the complete problem that is obtained from \eqref{state} and \eqref{GIp} in the sharp-interface situation.

\subsection{The sharp-interface limit of the state equation}
Therefore, we recall that the domain $\Omega$ is partitioned into $N$ regions $\Omega_i$ for $i=1,\dots,N$ representing the presence of the $i$-th material ($i<N$) or void ($i=N$) in its pure form. Those regions are separated by interfaces $\Gamma_{ij}$. Furthermore we have chosen $\B{\eta}_{\Gamma_{ij}}$ to be the unit normal vector field on $\Gamma_{ij}$ pointing from $\Omega_i$ into the region $\Omega_j$.
This means that
\begin{align*}
    x+\delta\B{\eta}_{\Gamma_{ij}}(x) \in \Omega_j
    \quad\text{and}\quad
    x-\delta\B{\eta}_{\Gamma_{ij}}(x)\in \Omega_i
    \quad\text{$x\in \Gamma_{ij}$ and $\delta>0$}.
\end{align*}
To capture the behavior of a function $\B{v}$ across the interface $\Gamma_{ij}$, we defined its \textit{jump} by
\begin{align*}
    [\B{v}]_i^j(x) \coloneqq \underset{\delta\searrow 0}{\lim}\; 
        \Big(
            \B{v}\big(x+\delta\B{\eta}_{\Gamma_{ij}}(x) \big)
            - \B{v}\big(x-\delta\B{\eta}_{\Gamma_{ij}}(x) \big) 
        \Big),
\end{align*}
for all $x\in\Gamma_{ij}$, see \eqref{DEF:JUMP}.

Combining the equations \eqref{state0} derived in Claim~\ref{Thm:Lam-1} and the jump conditions obtained in \eqref{wzer} and \eqref{CEnu}, we obtain the system
\begin{align}\tag{${SE}^{ij}_r$}\label{SeLim}
	\begin{cases}
		\begin{array}{rll}
				-\nabla\cdot\big(\mathbb{C}^{i}\mathcal{E}(\B{w}_{0,n_r})\big)
			&=
				\lambda_{0,n_r}
				\rho^i\B{w}_{0,n_r}
			&\quad
				\text{in }\Omega_i,
			\\[1ex]
			\left[\mathbb{C}\mathcal{E}(\B{w}_{0,n_r}) \B{n}_{\Gamma_{ij}} \right]_i^j
			&=
			\B{0}
			&\quad \text{on }\Gamma_{ij},
			\\[1ex]
    				[\B{w}_{0,n_r}]_i^j
			&=
				\B{0}
			&\quad \text{on }\Gamma_{ij},
			\\[1ex]
			    \mathbb{C}^i\mathcal{E}_i(\B{w}_{0,n_r}) \B{n}_{\Gamma_{iN}}
			&=
			    \B{0}
			&\quad
			    \text{on }\Gamma_{iN},
			\\[1ex]
					\mathbb{C}^{i}\mathcal{E}(\B{w}_{0,n_r}) \B{n}
			&=
				\B{0}
			&\quad
				\text{on }\Gamma_0\cap \partial\Omega_i,
			\\[1ex]
				\B{w}_{0,n_r}
			&=
				\B{0}
			&\quad
				\text{on }\Gamma_D\cap \partial\Omega_i,
		\end{array}
	\end{cases}
\end{align}
for $i,j=1,\dots,N-1$ and $r=1,\dots,l$, as the \textit{sharp-interface limit of the state equation}~\eqref{state}.
Here, $\bw_{0,n_r}$ is normalized in the material regions, i.e.,
\begin{align}\label{wFinalNorm}
    1
    =\sum_{i=1}^{N-1}\int_{\Omega_i}\rho^i\abs{\bw_{0,n_r}}^2\mathrm{\,d}x.
\end{align}
Furthermore, we infer from \eqref{CEnu} that
\begin{align}\label{NeuVoid}
    [\bw_{0,n_r}]_i^N=\B{0}
    \quad\text{on } \Gamma_{iN}
\end{align}
for all $i\in\{1,\dots,N-1\}$ and each $r\in\{1,\dots,l\}$. However, this condition does not provide any additional information as we do not know how $\bw_{0,n_r}$ behaves in the void region. 
In particular, we see that by interpreting \eqref{SeLim} as one system of PDEs in the material region $\bigcup_{i=1}^{N-1}\Omega_i$, the homogeneous Neumann boundary condition in the fourth line of \eqref{SeLim} is enough to obtain a closed system.

Combining the Neumann type jump condition on $\Gamma_{ij}$ stated in the second line of \eqref{SeLim} with the normality condition \eqref{wFinalNorm}, we are able to obtain the relation 
\begin{align}
    \label{EQ:SIEIG}
        \int_{\Omega^M} \C^M \, \E(\bw_{0,n_r}):\E(\bw_{0,n_r})\mathrm{\,d}x = \lambda_{0,n_r},
\end{align}
with
\begin{align*}
    \Omega^M := \bigcup_{i=1}^{N-1}\Omega_i
    \quad\text{and}\quad
    \C^M := \left( \sum_{i=1}^{N-1} \C^i\; \mathds{1}_{\Omega_i} \right),
\end{align*}
where $\mathds{1}_{\Omega_i}$ denotes the characteristic function on $\Omega_i$.
This means that the eigenvalue $\lambda_{0,n_r}$ in the sharp-interface setting is indeed solely determined by an eigenvalue equation on the material region $\Omega^M$ but does not have any contribution from the void region.

To verify \eqref{EQ:SIEIG}, 
we test \eqref{SeLim} with $\bw_{0,n_r}$ and integrate by parts. This yields
\begin{align}
    \label{EQ:OMI}
    \begin{aligned}
    &\int_{\Omega_i}\C^i\E(\bw_{0,n_r}):\E(\bw_{0,n_r})\mathrm{\,d}x-
    \int_{\partial\Omega_i}\C^i\E(\bw_{0,n_r})\B{n}_{\Gamma_i} 
        \cdot \bw_{0,n_r}\mathrm{\,d}\Gamma=
    \lambda_{0,n_r}
    \\
    &\quad =\int_{\Omega_i}\rho^i\abs{\bw_{0,n_r}}^2\mathrm{\,d}x,
    \end{aligned}
\end{align}
for all $i\in\{1,\dots,N-1\}$, where $\B{n}_{\Gamma_i}$ stands for the outer unit normal vector field of $\partial\Omega_i$. Noticing that the outer unit normal vector simply switches its sign on neighboring boundaries, we now use the second and the fourth line of \eqref{SeLim} to infer
\begin{align*}
    \sum_{i=1}^{N-1}\int_{\partial\Omega_i}\C^i\E(\bw_{0,n_r})\B{n}_{\Gamma_i}\cdot \bw_{0,n_r}\mathrm{\,d}\Gamma=0.
\end{align*}
Thus, summing the equations \eqref{EQ:OMI} from $i=1$ to $N-1$ and using property \eqref{wFinalNorm}, we conclude
\begin{align*}
    \sum_{i=1}^{N-1}    \int_{\Omega_i}\C^i\E(\bw_{0,n_r}):\E(\bw_{0,n_r})\mathrm{\,d}x=
    \lambda_{0,n_r}.
\end{align*}
By the linearity of the integral, this directly proves \eqref{EQ:SIEIG}.

\begin{Rem}\label{REM:Korn2}
    As a refinement of Remark~\ref{REM:Korn}~\ref{KornA}, we now see that as long as at least one of the material regions $\Omega_{1},\dots, \Omega_{N-1}$ shares a sufficiently nice part of its boundary with $\Gamma_D$, we can apply Korn's inequality in order to deduce that all $\lambda_{0,n_r}$ are strictly positive. 
    From a physical point of view, this is reasonable since if the material region $\Omega^M$ of the structure is not attached to some fixed boundary the shape can freely move within the design domain just by translation without exhibiting any vibrations.
\end{Rem}

\subsection{The sharp-interface limit of the first-order optimality condition}

Now let us turn to the limit of the gradient inequality \eqref{GIp}. 
For the sake of completeness, let us restate our final results from the previous section, i.e., \eqref{SCond} and \eqref{TP}. We have
\begin{alignat}{2}\label{GIsia}
    &0
		= \gamma \sigma_{ij}\kappa_{ij}
		-\sum_{j=1}^{l}
        [\partial_{\lambda_{n_r}}\hspace{-0.7ex}\Psi]
			\left(\lambda_{0,n_1},\dots,\lambda_{0,n_l}\right)
					\lambda_{0,n_r}
					\left[\overline{\rho\,}\abs{\B{w}_{0,n_r}}^2\right]_i^{j}
		+ \gamma \big(\vartheta_1^i-\vartheta_1^j\big)
	\notag\\[-1ex]
	&\qquad\quad
	\begin{aligned}
		+\sum_{r=1}^{l}
            &\Big\{ [\partial_{\lambda_{n_r}}\hspace{-0.7ex}\Psi]
				\left(\lambda_{0,n_1},\dots,\lambda_{0,n_l}\right)
		\\
			&\quad\cdot 
			\Big(\big[
					\overline{\mathbb{C}}
                    \mathcal{E}(\B{w}_{0,n_r}):\mathcal{E}(\B{w}_{0,n_r})\big]_i^j
			-2
					\big[\overline{\mathbb{C}}\mathcal{E}(\B{w}_{0,n_r})\B{\nu}\cdot\nabla\B{w}_{0,n_r}\B{\nu}\big]_i^j
			\Big)
			\Big\}
    \end{aligned}
\end{alignat}
on $\Gamma_{ij}$ for all $i,j=1\dots,N-1$, and 
\begin{alignat}{2}\label{GIsia:N}
    \begin{aligned}
        0
        &= \gamma \sigma_{iN}\kappa_{iN}
            +\sum_{r=1}^{l}[\partial_{\lambda_{n_r}}\hspace{-0.7ex}\Psi]
                \left(\lambda_{0,n_1},\dots,\lambda_{0,n_l}\right)
                \lambda_{0,n_r}
                \; \rho^i\abs{(\B{w}_{0,n_r})_i}^2
            + \gamma \big(\vartheta_1^i-\vartheta_1^N\big)
        \\
        &\quad
            -\sum_{r=1}^{l}[\partial_{\lambda_{n_r}}\hspace{-0.7ex}\Psi]
                \left(\lambda_{0,n_1},\dots,\lambda_{0,n_l}\right)
                    \mathbb{C}^i\mathcal{E}_i(\B{w}_{0,n_r}):\mathcal{E}_i(\B{w}_{0,n_r})
    \end{aligned}
\end{alignat}
on $\Gamma_{iN}$ for all $i=1\dots,N-1$ if $j=N$.
\pagebreak[2]

Here $\sigma_{ij}$ is defined as in \eqref{DEF:SIGMAIJ} and stands for the total energy of a transition across the interface $\Gamma_{ij}$.
The vector $\B{\vartheta}_1\in \R^N$ denotes the $\mathcal{O}(\eps)$-contribution of the Lagrange-multiplier 
resulting from the integral constraint $\fint_{\Omega} \bphi^\eps \mathrm{\,d}x=\B{m}$ 
that is hidden in the condition $\bphi^\eps\in \B{\mathcal{G}}^{\B{m}}$ (cf.~Theorem~\ref{Thm:Lagr}).\\

Recalling \eqref{Angle}, we additionally have the triple junction condition at any junction 
%point
$m_{ijk}$ with pairwise disjoint $i,j,k\{1,\dots,N\}$
\begin{align*}
	\sigma_{ij}\B{n}_{\Gamma_{ij}}+\sigma_{jk}\B{n}_{\Gamma_{jk}}+\sigma_{ki}\B{n}_{\Gamma_{ki}}=0\quad\text{in }m_{ijk}.
\end{align*}

\subsection{The sharp-interface optimality system in the case of only one material} \label{SEC:OS}
We now want to state above equations for the simplest case of only one single material (i.e., $N=2$) as this is the scenario we further study in the subsequent sections. 

In this case, we have $\Omega=\Omega^{M}\cup \Omega^{V}$, where $\Omega^{M}$ and $\Omega^{V}$ denote the material and the void parts of the domain, respectively. We now denote the interface separating the two phases by $\Gamma_{MV}$, its outer unit normal vector field by $\B{n}_{\Gamma_{MV}}$ and its mean curvature by $\kappa_{MV} = - \nabla_{\Gamma_{MV}}\cdot \B{n}_{\Gamma_{MV}}$. Using the notation $\Gamma_D^{M}\coloneqq \Gamma_D\cap \partial\Omega^M$ and $\Gamma_0^{M}\coloneqq \Gamma_0\cap \partial\Omega^M$, we obtain from \eqref{SeLim}, \eqref{wFinalNorm} and \eqref{NeuVoid} the state equation
\begin{align}\tag{${SE}^{MV}_r$}\label{SErMV}
	\begin{cases}
		\begin{array}{rll}
				-\nabla\cdot\left(\mathbb{C}^{M}\mathcal{E}(\B{w}_{0,n_r})\right)
			&=
				\lambda_{0,n_r}
				\rho^M\B{w}_{0,n_r}
			&\quad
				\text{in }\Omega^M,\\[1ex]
				\mathbb{C}^M\mathcal{E}_M(\B{w}_{0,n_r})\; \B{n}_{\Gamma_{MV}}
			&=
				\B{0}
			&\quad \text{on }\Gamma_{MV},\\[1ex]
				\B{w}_{0,n_r}
			&=
				\B{0}
			&\quad
				\text{on }\Gamma_D^M,\\[1ex]
					\mathbb{C}^{M}\mathcal{E}(\B{w}_{0,n_r})\; \B{n}
			&=
				\B{0}
			&\quad
				\text{on }\Gamma_0^M,
		\end{array}
	\end{cases}
\end{align}
for $r=1,\dots,l$, along with the first-order necessary optimality condition
\begin{alignat}{2}\label{GMV}
	\begin{aligned}
			0
		&=
			\gamma\,\sigma_{MV}\,\kappa_{MV}
			+\sum_{r=1}^{l}[\partial_{\lambda_{n_r}}\hspace{-0.7ex}\Psi]
				\left(\lambda_{0,{n_1}},\dots,\lambda_{0,{n_l}}\right)
				\lambda_0^{n_r}
				\rho^M\abs{(\B{w}_{0,n_r})_M}^2\\
		&\quad
			-\sum_{r=1}^{l}[\partial_{\lambda_{n_r}}\hspace{-0.7ex}\Psi]
			\left(\lambda_{0,n_1},\dots,\lambda_{0,n_l}\right)
			\mathbb{C}^M\mathcal{E}_M(\B{w}_{0,n_r}):\mathcal{E}_M(\B{w}_{0,n_r})
			+\gamma\left(\vartheta_1^1-\vartheta_1^2\right)
% 			\quad\text{on }\Gamma_{MV}.
	\end{aligned}
\end{alignat}
on $\Gamma_{MV}$.
This means that the functions $\B{w}_{0,n_r}$ are eigenfunctions to the eigenvalues $\lambda_{0,n_r}$ which essentially solve the eigenvalue problem for the elasticity equation subject to a homogeneous Neumann boundary condition on the shape $\Omega^M$. 

\begin{Rem}\label{Rem:Smooth}
    Note that, in general, one cannot predict the behavior of solutions to \eqref{SErMV}. If $\Omega^M$ is merely a set of finite perimeter that does not have a Lipschitz boundary or if $\Gamma_{MV}\cap\Gamma_D^M=\emptyset$, the classical spectral theory (as applied in Section~\ref{SUB:Se}) does not provide us with an infinite sequence of positive eigenvalues. Nevertheless, as we want to consider a well posed minimization problem and want to calculate shape derivatives associated to this problem, we assume that these issues do not occur. In particular, we  always assume $\Omega^M$ to be sufficiently smooth and $\partial\Omega^M$ to have a suitably nice intersection with $\Gamma^M_D$ such that an infinite sequence of positive eigenvalues actually exists (see also Remark~\ref{REM:Korn2}). 
\end{Rem}

%%%%%%%%%%%%%%%%%%%%%%%%%%%%%%%%%%%
%%%%%%%%%RELATING SHAPE CALC%%%%%%%
%%%%%%%%%%%%%%%%%%%%%%%%%%%%%%%%%%%

\section[Relating the first-order optimality condition to classical shape calculus]{Relating the first-order optimality condition\\ to classical shape calculus}\label{SEC:Rel}
We now want to compare the above results, especially \eqref{GMV}, to the results in \cite{Allaire}, which were obtained using shape calculus. Our goal is to justify that the gradient equality \eqref{GMV} is indeed the first-order condition of a sharp-interface eigenvalue optimization problem, which is formally the limit of the diffuse-interface problem we started with.
Therefore, we need to fit the notation of \cite{Allaire} to our setting.

As above consider the situation $N=2$, i.e., $\Omega=\Omega^M\cup \Omega^V$. Denote with $P_{\Omega}(\Omega^M)$ the perimeter of the shape $\Omega^M$ \textit{within} the design domain $\Omega$, which is given by the Hausdorff measure $\mathcal{H}^{d-1}(\partial\Omega^M\cap\Omega)$ provided that $\Omega^M$ is non-empty and sufficiently smooth. Furthermore, we consider a prescribed mass $m=\big|\Omega^M\big|<|\Omega|$. 
In order to be consistent with the notation used in the previous chapters, we choose $\B{m}=(m_1,m_2)^T\in\Sigma^2$ with $m_1 = m|\Omega|^{-1}$ and $m_2=1-m_1$.
Then the sharp-interface structural optimization problem that we intend to approximate via our diffuse-interface problem \eqref{Pepsla} reads as
\begin{alignat}{2}\tag{$\mathcal{P}_l^0$}\label{Pz}
	\begin{cases}
		\begin{aligned}
		 	&\min&&
		 		 J(\Omega^M)\coloneqq \Psi(\lambda_{n_1},\dots,\lambda_{n_l})+\gamma\, \sigma_{MV}\, P_{\Omega}(\Omega^M),
		    \\[1ex]
		 	&\text{over}&&
		 			\mathcal{U}^{\text{ad}}
		 		=
		 			\left\{
		 				\Omega^M\subset \Omega :\big| \Omega^M \big| = m
		 			\right\},
		 	\\[1ex]
		 	&\text{s.t.}&&
		 		\eqref{SErMV}
		 			\begin{cases}
		 				\begin{array}{rll}
		 						-\nabla\cdot\left(\mathbb{C}^{M}\mathcal{E}(\B{w}_{n_r})\right)
		 					&=
		 						\lambda_{n_r}
		 						\rho^M\B{w}_{n_r}
		 					&\quad
		 						\text{in }\Omega^M,
		 						\\[1ex]
		 						\mathbb{C}^M\mathcal{E}_M(\B{w}_{n_r})\; \B{n}_{\Gamma_{MV}}
		 					&=
		 						\B{0}
		 					&\quad \text{on }\Gamma_{MV},
		 					\\[1ex]
		 							\mathbb{C}^{M}\mathcal{E}(\B{w}_{n_r})\; \B{n}
		 					&=
		 						\B{0}
		 					&\quad
		 						\text{on }\Gamma_0^M,
		 					\\[1ex]
		 						\B{w}_{n_r}
		 					&=
		 						\B{0}
		 					&\quad
		 						\text{on }\Gamma_D^M,
		 				\end{array}
		 			\end{cases}
		 	\\[1ex]
            &&& \text{for all}\; r\in\{1,\dots,l\}.     
		 \end{aligned}
	\end{cases}	 
\end{alignat}
This system is the sharp-interface limit problem associated to the diffuse-interface problem \eqref{Pepsla}, where the side condition is exactly the sharp-interface state equation \eqref{SErMV} and the perimeter $\sigma_{MV}P_{\Omega}(\Omega^M)$ is the rigorous $\Gamma$-limit of the Ginzburg--Landau energy, see \cite{Baldo}. We recall that the constant $\sigma_{MV}$ we obtained in \eqref{DEF:SIGMAIJ} is exactly the one obtained in \cite{Baldo} in terms of the rigorous $\Gamma$-limit, which is denoted by $d(\B{e}_i,\B{e}_j)$ there. In particular, $\sigma_{MV}$ is independent of the shape $\Omega^M$.

In case an ambiguity might arise, we indicate the shape dependency explicitly in the eigenfunctions and eigenvalues, i.e., we write $(\lambda_{n_r}(\Omega^M),\bw_{n_r}(\Omega^M))$ for $r=1,\dots,l$. 
Now, we want to apply the calculus of shape derivatives from \cite[Thm. 2.5]{Allaire} to our situation. We obtain the following statement.

\begin{Thm}
	Let $\Omega^M$ be a smooth bounded open set and let $\B{\theta}\in W^{1,\infty}(\mathbb{R}^d,\mathbb{R}^d)$ with ${\B{\theta}\cdot\B{n}_{\Gamma_{\partial\Omega^M}}=0}$ on $\partial\Omega^M\backslash\Gamma_{MV}$.
    We further assume that for $r=1,\dots,l$, the eigenfunctions $\B{w}_{n_r}(\Omega^M)$ in $({SE}^{MV}_r)$ are sufficiently smooth,
    say $\B{w}_{n_r}(\Omega^M)\in H^2(\Omega^M;\mathbb{R}^d)$.
      
	Then, if the involved eigenvalues $\lambda_{n_r}$ for $r=1,\dots,l$ are all simple, the shape derivative of $J$ at the shape $\Omega^M$ in the direction $\B{\theta}$ fulfills the equation
	\begin{alignat}{2}\label{WShape}
		\begin{aligned}
			J^{\prime}(\Omega^M)(\B{\theta})
			&= \sum_{r=1}^{l}\Bigg\{ [\partial_{\lambda_{n_r}}\hspace{-0.7ex}\Psi](\lambda_{n_1}(\Omega^M),\dots,\lambda_{n_l}(\Omega^{M}))
			\\[-1ex]
			&\qquad\qquad \cdot\Bigg[
						\int_{\Gamma_{MV}}
				    	\mathbb{C}^M\mathcal{E}(\B{w}_{n_r}(\Omega^M)):\mathcal{E}(\B{w}_{n_r}(\Omega^M))
							\B{\theta}\cdot\B{n}_{\Gamma_{MV}}
						\textup{\,d}\HH^{d-1}\;
			\\[-1ex]
			&\qquad\qquad\qquad       
						- \lambda_{n_r}(\Omega^M)
							\int_{\Gamma_{MV}}
								\rho^M\big|\B{w}_{n_r}(\Omega^M)\big|^2
								\B{\theta}\cdot\B{n}_{\Gamma_{MV}}
						\textup{\,d}\HH^{d-1}
			\Bigg]
		    \Bigg\}
		    \\
			&\qquad -\int_{\Gamma_{MV}}
					\gamma\sigma_{MV}\,\kappa_{MV}\,
				\B{\theta}\cdot\B{n}_{\Gamma_{MV}}
		    \textup{\,d}\HH^{d-1}.
		\end{aligned}
	\end{alignat}
	Here, the shape derivative of $J$ at a shape $\Omega^M$ is defined as the Fréchet derivative of the functional
	\begin{align*}
	    W^{1,\infty}(\R^d;\R^d)\to \R, \quad
	    \B{\zeta}\mapsto J\big((\mathrm{Id}+\B{\zeta})\Omega^M\big)
	\end{align*}
	evaluated at $\B{\zeta}=\B{0}$.
\end{Thm}

\begin{Rem}\label{REM:DifShape}
    \begin{enumerate}[label = (\alph*), leftmargin=*]
        \item Note that the simplicity of eigenvalues is crucial here. Only then it is guaranteed that the eigenvalues and eigenfunctions depend on the domain $\Omega^M$ in a differentiable way. For a comprehensive overview over the differentiablitiy of spectral quantities with repsect to the domain we refer to \cite[Section 5.7]{HenrotPierre}.
        \item For $\B{\zeta}\in W^{1,\infty}(\R^d;\R^d)$ the application
                \begin{align*}
                    T_{\B{\zeta}}:\R^d\to \R^d, \quad
                    x \mapsto (\mathrm{Id}+\B{\zeta})(x),
                \end{align*}
                is invertible if $\norm{\B{\zeta}}_{W^{1,\infty}}<1$, and it holds $(\mathrm{Id}+\B{\zeta})^{-1}-\mathrm{Id}\in W^{1,\infty}(\R^d;\R^d)$ with 
                \begin{align*}
                    \norm{(\mathrm{Id}+\B{\zeta})^{-1}-\mathrm{Id}}_{W^{1,\infty}}\le
                    \norm{\B{\zeta}}_{W^{1,\infty}}(1-\norm{\B{\zeta}}_{W^{1,\infty}})^{-1}.
                \end{align*}
                This means the family $(T_{\B{\zeta}})_{\B{\zeta}\in W^{1,\infty}}$ describes diffeomorphic perturbations of $\Omega^M$ ``close'' to $\Omega^M$ if $\norm{\B{\zeta}}_{W^{1,\infty}}$ is small, motivating the definition of the shape derivative above. For a detailed discussion of this concept, we refer to \cite[Section 5.2]{HenrotPierre}.
    \end{enumerate}
    
\end{Rem}
\begin{proof}
    We proceed analogously to \cite[Theorem 2.5]{Allaire}.
    In the following, $\Omega_{\B{\zeta}}=(\mathrm{Id}+\B{\zeta})(\Omega^M)$ denotes the perturbation of $\Omega^M$ associated with a sufficiently small $\B{\zeta}\in W^{1,\infty}(\R^d;\R^d)$.
	First of all, for $\B{v}_{n_r}\in H^1(\R^d;\R^d)$ with $r=1,\dots,l$, we introduce the Lagrangian
	\begin{alignat*}{2}
		&\mathcal{L}(\Omega_{\B{\zeta}},\B{v}_{n_1},\dots,\B{v}_{n_l})
		\\[1ex]
		&\quad=
		\Psi
		\left(
		\frac{
			\int_{\Omega_{\B{\zeta}}}
                \mathbb{C}^M\mathcal{E}({\B{v}_{n_1}}):\mathcal{E}({\B{v}_{n_1}}) 
            \mathrm{\,d}x		
		}{\int_{\Omega_{\B{\zeta}}}\rho^M\abs{\B{v}_{n_1}}^2\mathrm{\,d}x},\dots,
		\frac{
			\int_{\Omega_{\B{\zeta}}}
                \mathbb{C}^M\mathcal{E}({\B{v}_{n_l}}):\mathcal{E}({\B{v}_{n_l}}) 
            \mathrm{\,d}x		
		}{\int_{\Omega_{\B{\zeta}}}\rho^M\abs{\B{v}_{n_l}}^2\mathrm{\,d}x}	
		\right)
		\\[1ex]
		&\qquad\quad +\gamma\sigma_{MV} P(\Omega_{\B{\zeta}})\mathrm{\,d}s.
	\end{alignat*}

For the partial Fréchet derivatives of the Lagrangian with respect to $\B{v}_{n_r}$ for $r=1,\dots,l$ at the point $(\Omega_{\B{\zeta}},\B{w}_{n_1}(\Omega_{\B{\zeta}}),\dots, \B{w}_{n_l}(\Omega_{\B{\zeta}}))$, we obtain
\begin{align}\label{ParLVan}
    \partial_{\B{v}_{n_r}} \mathcal{L}\big(\Omega_{\B{\zeta}},\B{w}_{n_1}(\Omega_{\B{\zeta}}),\dots, \B{w}_{n_l}(\Omega_{\B{\zeta}})\big)=0.
\end{align}
This is simply due to the fact, that the derivative of the Rayleigh quotient 
\begin{align*}
    \mathcal{R}_{\B{\zeta}}: H^1(\R^d;\R^d)\to \R, \quad
    \B{v}\mapsto \frac{
			\int_{\Omega_{\B{\zeta}}}
                \mathbb{C}^M\mathcal{E}({\B{v}}):\mathcal{E}({\B{v}}) 
            \mathrm{\,d}x		
		}{\int_{\Omega_{\B{\zeta}}}\rho^M\abs{\B{v}}^2\mathrm{\,d}x},
\end{align*}
evaluated at an eigenfunction $\B{w}_n=\B{w}_n(\Omega_{\B{\zeta}})$ reads as
\begin{align*}
    \mathcal{R}^{\prime}_{\B{\zeta}}(\bw_n)\B{v}
    &=\frac{
        2\int_{\Omega_{\B{\zeta}}}\C^M\E(\B{w}_n):\E(\B{v})\mathrm{\,d}x
        \int_{\Omega_{\B{\zeta}}}\rho^M\abs{\B{w}_n}^2\mathrm{\,d}x}
        {(\int_{\Omega_{\B{\zeta}}}\rho^M\abs{\B{w}_n}^2\mathrm{\,d}x)^2}\\
    &\phantom{=}-\frac{    
        2\int_{\Omega_{\B{\zeta}}}\C^M\E(\B{w}_n):\E(\B{w}_n)\mathrm{\,d}x
        \int_{\Omega_{\B{\zeta}}}\rho^M\B{w}_n\cdot\B{v}\mathrm{\,d}x
    }{(\int_{\Omega_{\B{\zeta}}}\rho^M\abs{\B{w}_n}^2\mathrm{\,d}x)^2}
\end{align*}
and this vanishes due to \eqref{SErMV}.
% \pagebreak[2]

On the other hand, recalling the definition of $J$ in \eqref{Pz}, we obviously have
\begin{align*}
    J(\Omega_{\B{\zeta}})=\mathcal{L}(\Omega_{\B{\zeta}},\B{w}_{n_1}(\Omega_{\B{\zeta}}),\dots,\B{w}_{n_l}(\Omega_{\B{\zeta}}))
\end{align*}
as the eigenvalues can be expressed by the corresponding Rayleigh quotients.
Note that due to the differentiability of eigenfunctions as discussed in Remark~\ref{REM:DifShape}, we can now apply the chain rule. Thus, using \eqref{ParLVan} we infer that the shape derivative is given by
\begin{align*}
    J^{\prime}(\Omega^M)&=
    \frac{\text{d}}{\text{d}\B{\zeta}}[J((\mathrm{Id}+\B{\zeta})(\Omega^M))]_{\B{\zeta}=\B{0}}\\&=
    \frac{\text{d}}{\text{d}\B{\zeta}}[\mathcal{L}((\mathrm{Id}+\B{\zeta})(\Omega^M),\bw_{n_1}(\Omega^M),\dots,\bw_{n_l}(\Omega^M))]_{\B{\zeta}=\B{0}}
\end{align*}
Applying the formulas for shape derivatives in \cite[Lemma 2.3]{Allaire}, we deduce
\begin{align*}
    \begin{aligned}
		J^{\prime}(\Omega^M)(\B{\theta})
		&= \sum_{r=1}^{l}\Bigg\{ [\partial_{\lambda_{n_r}}\hspace{-0.7ex}\Psi](\lambda_{n_1}(\Omega^M),\dots,\lambda_{n_l}(\Omega^{M}))
		\\[-1ex]
		&\qquad\qquad \cdot\Bigg[
					\int_{\partial\Omega^M}
			    	\mathbb{C}^M\mathcal{E}(\B{w}_{n_r}(\Omega^M)):\mathcal{E}(\B{w}_{n_r}(\Omega^M))
						\B{\theta}\cdot\B{n}_{\partial\Omega^M}
					\textup{\,d}\HH^{d-1}\;
		\\[-1ex]
		&\qquad\qquad\qquad       
					- \lambda_{n_r}(\Omega^M)
						\int_{\partial\Omega^M}
							\rho^M\big|\B{w}_{n_r}(\Omega^M)\big|^2
							\B{\theta}\cdot\B{n}_{\partial\Omega^M}
					\textup{\,d}\HH^{d-1}
		\Bigg]
	    \Bigg\}
	    \\
		&\qquad -\int_{\partial\Omega^M}
				\gamma\sigma_{MV}\,\kappa_{M}\,
			\B{\theta}\cdot\B{n}_{\partial\Omega^M}
	    \textup{\,d}\HH^{d-1},
	\end{aligned}
\end{align*}
where $\kappa_M$ denotes the mean curvature of $\partial\Omega^M$.
By the assumption ${\B{\theta}\cdot\B{n}_{\partial\Omega^M}=0}$ on $\partial\Omega^M\backslash\Gamma_{MV}$, the boundary integrals vanish on $\partial\Omega^M\backslash\Gamma_{MV}$ and we thus arrive at \eqref{WShape}.  
Note that in \cite{Allaire}, the mean curvature is defined as $\kappa=\nabla_{\partial\Omega^M}\cdot \B{n}_{\partial\Omega^M}$, whereas (in accordance with \eqref{RepCurv}) our mean curvature is given by $\kappa=-\nabla_{\partial\Omega^M}\cdot \B{n}_{\partial\Omega^M}$. This explains the negative sign of our term involving $\kappa_M$.
\end{proof}

\begin{Rem}
The preceding theorem shows that using the approach of classical shape calculus and additionally taking the volume constraint $\big|\Omega^M\big|=m$ into account, we recover the gradient equality \eqref{GMV} since the volume constraint produces a Lagrange multiplier as in our previous analysis.
This justifies our formal approach from the viewpoint of classical shape calculus since \eqref{GMV} can be interpreted as the first-order necessary optimality condition of the shape optimization problem \eqref{Pz}.
 \end{Rem}
 
%%%%%%%%%%%%%%%%%%%%%%%%%%%%%%%%%%%
%%%%%%%NUMERICAL EXAMPLES%%%%%%%%%%
%%%%%%%%%%%%%%%%%%%%%%%%%%%%%%%%%%%

\section{Numerical Examples}
\label{sec:num}

In the following, we present numerical results that illustrate the applicability of our approach to find optimal topologies.
After a brief introduction of the numerical method, we investigate the dependence of solutions on the parameter $\eps$ in Section~\ref{ssec:num:eps}. Therefore, we study a particular setting of an elastic beam that is known from literature (cf.~\cite{Allaire}).
In Section~\ref{ssec:num:beam}, we consider a joint optimization of $\lambda_1$ and $\lambda_2$ for this beam setup,
and in Section~\ref{ssec:num:beam-compl}, we investigate an extended optimization problem to not only optimize the shape and topology of this beam with respect to its first eigenvalue but also its compliance.

As in Subsection~\ref{SEC:OS} and Section~\ref{SEC:Rel}, we restrict ourselves to the case of only two phases, i.e., material and void. 
In this situation, the vector-valued phase-field $\bphi=(\varphi^1,\varphi^2)$ can be represented by a scalar order parameter 
\begin{align*}
    \varphi:=\varphi^1 - \varphi^2 \in H^1(\Omega)\cap L^\infty(\Omega),
\end{align*}
and in the Ginzburg--Landau energy $E^\eps(\varphi) = \int_\Omega \frac{\eps}{2}|\nabla \varphi|^2 +\psi(\varphi) \mathrm{\,d}x$, we choose
\begin{align*}
    \psi(\varphi) := \psi_0(s) + I_{[-1,1]}(\varphi) 
    \quad\text{with}\quad
    \psi_0(\varphi) = \tfrac 1{2\eps} (1-\varphi^2),
\end{align*}
where $I_{[-1,1]}$ is the indicator functional of the interval $[-1,1]$.
This means that $\varphi$ attains its values in $[-1,1]$, where \enquote{$1$} represents the material and \enquote{$-1$} represents the void.
The elastic tensor $\mathbb C(\varphi)$ now is defined as
\begin{align}
    \mathbb C(\varphi) \mathcal E(w) := \alpha(\varphi) \big(2 \mu\, \mathcal E(w) + \ell \trace\big(\mathcal E(w)\big)\, \mathcal I \big)
\end{align}
for Lam\'e parameters $\mu, \ell >0$ and the quadratic interpolation function $\alpha(\varphi)$ satisfying $\alpha(1) = 1$, $\alpha(-1) = \underline \alpha \eps$, and $\alpha^\prime(-1) = 0$ for some constant $\underline \alpha$.
The eigenvalue equation is given  by
\begin{align}\label{NumScale}
    -\nabla \cdot \left[\mathbb C(\varphi) \mathcal E(w)\right]= \lambda\, \beta(\varphi) \rho\, w,
\end{align}
with the quadratic interpolation function $\beta(\varphi)$ satisfying $\beta(1) = 1$, $\beta(-1) = \underline \beta \eps^2$ and ${\beta^\prime(-1) = 0}$ as well as 
an additional
density function $\rho$ that might depend on the spatial variable. If not stated differently, we use $\underline \alpha = 2\cdot10^{-4}$ and $\underline \beta = 10^{-4}$. Note that this choice of interpolation functions is exactly reflected by the choice \eqref{eq:num_choice} and the scaling choice of $k=1, l=2$ as discussed in Section~\ref{intermezzo}. More precisely we have
\begin{align}
    \begin{aligned}\label{eq:num_quad}
    \beta(\varphi)&=\beta_M\left(1-\frac{1-\varphi}{2}\right)+\underline{\beta}\eps^2\beta_V\left(\frac{1-\varphi}{2}\right)\\
    &=
    \left(1-\frac{1-\varphi}{2}\right)^2+
    \underline{\beta}\eps^2\left[-\left(\frac{1-\varphi}{2}-1\right)^2+1\right],
    \end{aligned}
\end{align}
and analogously for $\alpha$.

\paragraph{Numerical Solution Method.}
The numerical implementation is based on linear finite elements for all functions provided by the finite element package FEniCs~\cite{fenics1,fenics_book}
together with the PETSc linear algebra backend~\cite{petsc-user-ref,petsc-efficient}.
For the eigenvalue problem, we use the package SLEPc~\cite{slepc:2005}.
The optimization problem is solved by the VMPT method that is proposed in \cite{BlankRupprecht}. In our case, it can be understood as an extension of the  projected gradient method into the space $H^1(\Omega) \cap L^\infty(\Omega)$.
We refer to \cite{BlankRupprecht,GarHKL_OptNS,GarHKK_OptLapEV} for more details.

\subsection{Numerical investigation of the sharp-interface limit \texorpdfstring{$\eps \to 0$}{}}
\label{ssec:num:eps}
In this section, to illustrate the sharp-interface limit, we present numerical results for a sequence of decreasing values of $\eps$. 

We use the setup from \cite[Sec.~7.1]{Allaire} to find a cantilever beam with maximal first eigenvalue, i.e., we choose $\Psi(\lambda_1) = -\lambda_1$.
Our computational domain is given by $\Omega = (0,2)\times (0,1)$. 
The Young's modulus is $E = 1$ and Poisson's ratio is $\nu = 0.3$ leading to $\mu \approx 0.38$ and $\ell \approx 0.58$. 
We define the subset $\Omega_{\rho} = (1.9,2.0)\times (0.45,0.55)$ and set $\rho(x) = 1$ if $x \not \in \Omega_\rho$ and $\rho(x) = 100$ if $x \in \Omega_\rho$. We also fix $\varphi(x) = 1$ for all $x \in \Omega_\rho$.
The beam is supposed to be attached to the wall at the left boundary of $\Omega$, i.e., at $\Gamma_D = \{ (0,\eta)\mid \eta \in (0,1)\} \subset \partial\Omega$. This leads to the boundary condition $w = 0$ on $\Gamma_D$.
We further set $\Gamma_0 = \partial \Omega \setminus \Gamma_D$ and we fix $\gamma = 10^{-4}$ and  $\int_\Omega \varphi = 0$.

Similar as in \cite[Sec.~7.1]{Allaire}, we start our optimization process with a checkerboard type initial function given by $\varphi_0(x) = \sign\big(v(x)\big)\, \big|v\big(x\big)\big|^{0.3}$ with
$v(x) = \cos(3\pi x_1)\cos(4\pi x_2)$ for all $x\in\Omega$. 
We want to emphasize that this problem is expected to have many local minima and thus, the choice of initial function can significantly influence the shape and topology of the local minimizer found by our numerical method.

We now solve the optimization problem for a decreasing sequence of values of $\eps$. In Table~\ref{tab:num:InterfaceEps}, we present the values of $\eps$ together with
the corresponding value of the Ginzburg--Landau energy $E^\eps(\varphi) = \int_\Omega \frac{\eps}{2}|\nabla \varphi|^2 + \frac{1}{2\eps}(1-\varphi^2) \mathrm{\,d}x$ and the eigenvalue $\lambda_1$. Recall here that the values of the Ginzburg--Landau energy converge to a weighted perimeter of the shape in the sharp interface limit $\eps\to 0$.
In Figure~\ref{fig:num:beamEps}, we present the zero level lines of the (locally optimal) shapes we obtain for different values of $\eps$. Here we started with $\eps=0.08$ and used the local optimum as initial value for subsequent simulations.

\begin{table}%[]
	\centering\small
    \begin{tabular}{c*{8}{|c}}
         $\eps$ & $80 \cdot 10^{-3}$ &  $40\cdot 10^{-3}$ & $ 20\cdot 10^{-3}$ &  $10\cdot 10^{-3}$ &  $5\cdot 10^{-3}$ &  $2.5\cdot 10^{-3}$ & $1.25 \cdot 10^{-3}$\\
        \hline
        $\gamma E^\eps(\varphi)$ & 0.00117 & 0.00120 & 0.00117 & 0.00115 & 0.00114 & 0.00114 & 0.00114\\
        $\lambda_1$ & 0.01574 & 0.01626 & 0.01658 & 0.01678 & 0.01692 & 0.01699 & 0.01703
    \end{tabular}
    \caption{Scaled Ginzburg--Landau energy $\gamma E^\eps(\varphi)$ and principal eigenvalue $\lambda_1$ of the optimal beam shape for decreasing  values of $\eps$. 
    This indicates that the values $E^\eps(\varphi)$ and $\lambda_1$ converge as $\eps$ decreases.}
    \label{tab:num:InterfaceEps}
\end{table}
 
 \begin{figure}
     \centering
     \includegraphics[width=0.5\textwidth]{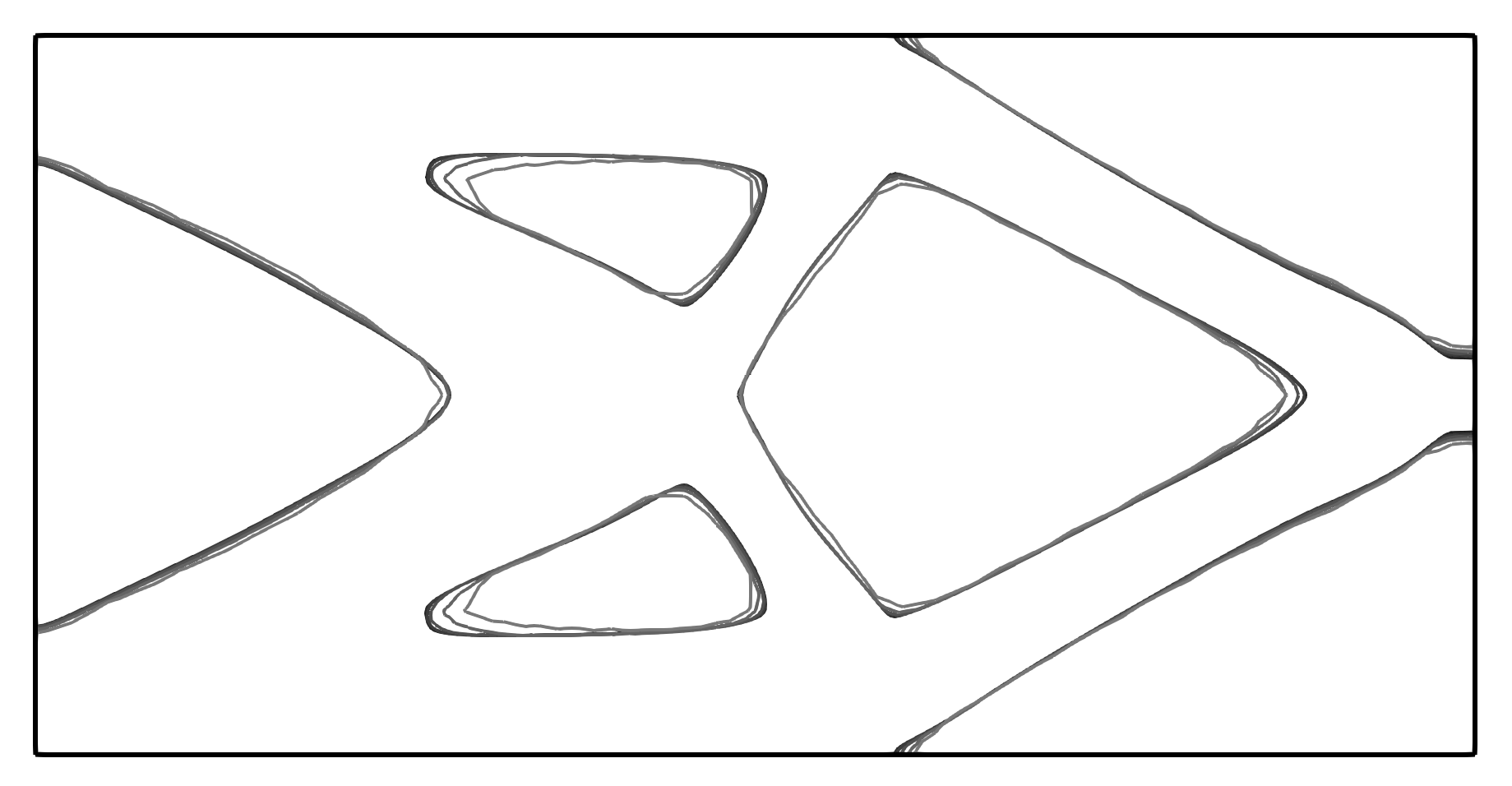}
     \caption{The zero level lines of the beam for the $\eps\to 0$ test for all tested $\eps$. 
     The darker the line is, the smaller is $\eps$.
     We observe that the interface seems to stabilize  with decreasing values of $\eps$ and that it only mildly depends on $\eps$.}
     \label{fig:num:beamEps}
 \end{figure}

\subsection{Optimization of a beam}
\label{ssec:num:beam}

As a first test, we illustrate the influence of the regularization strength $\gamma$ on the found structure.
The parameter $\gamma$ acts as a weight for the penalization of the length of the interface between void and material. 
Thus a smaller value of $\gamma$ is expected to lead to thinner structures which contain more braces. 
Using the same setup as before, we solve again the optimization problem for the cantilever beam, but this time we fix $\eps=0.02$.
We perform two simulations with $\gamma\in\{10^{-4},10^{-5}\}$.
The smaller $\gamma$ is chosen, the finer structures we expect. 
We also expect that we reach a larger value for $\lambda_1$, because less regularization is used.

In Figure~\ref{fig:num:beamGamma}, we present the found structures for these parameters. 
On the left we present the result for $\gamma=10^{-4}$ and on the right for $\gamma=10^{-5}$. 
As expected, it is clearly visible that the structure obtained for the smaller value of $\gamma$ is finer and contains more braces.
Additionally, decreasing $\gamma$ also leads to sharper corners. 

\begin{figure}
    \centering
    \fbox{
    \includegraphics[trim=200 80 200 80, clip,width=0.3\textwidth]{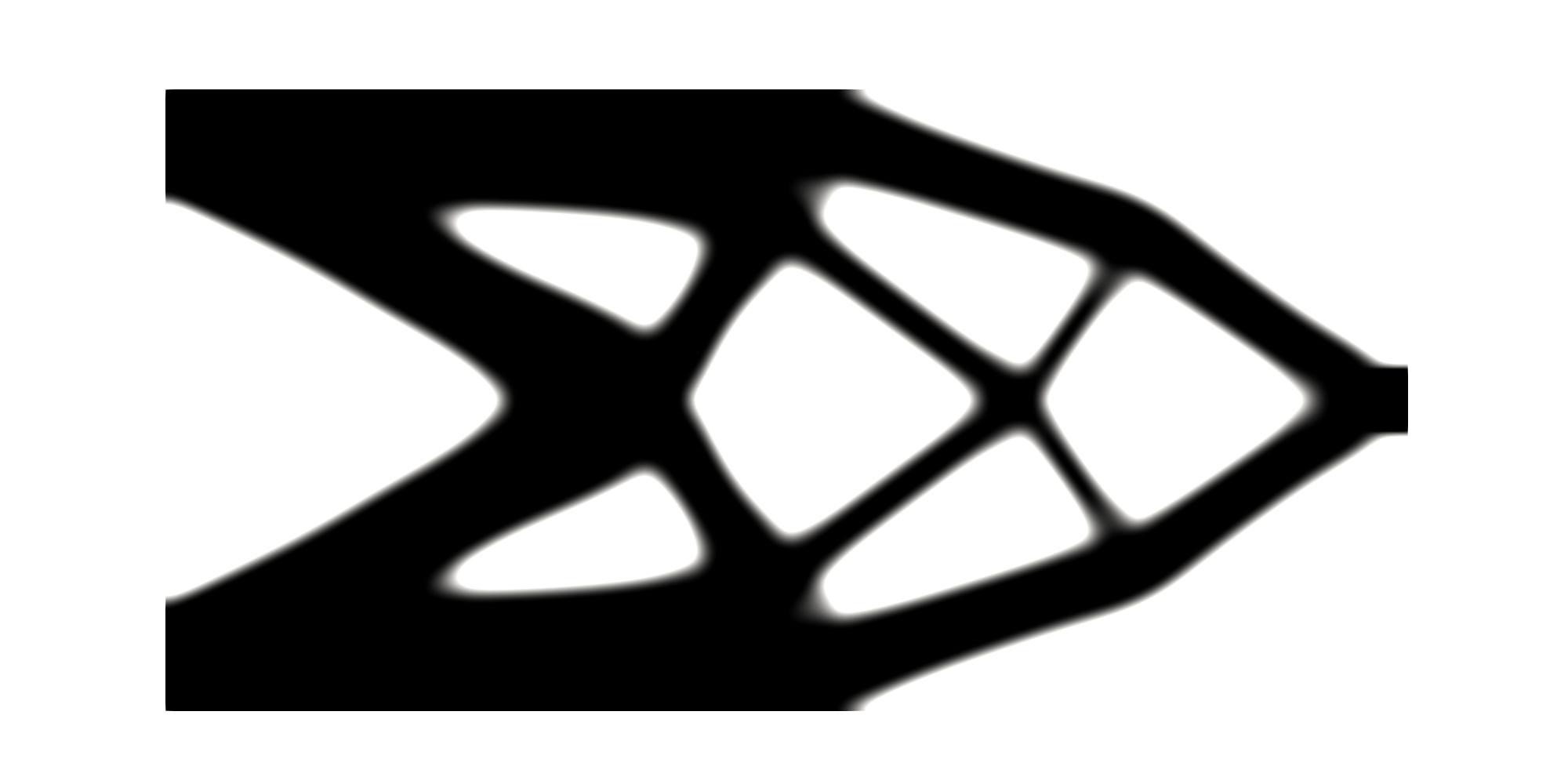}
    }
    $\qquad$
    \fbox{
    \includegraphics[trim=200 80 200 80, clip,width=0.3\textwidth]{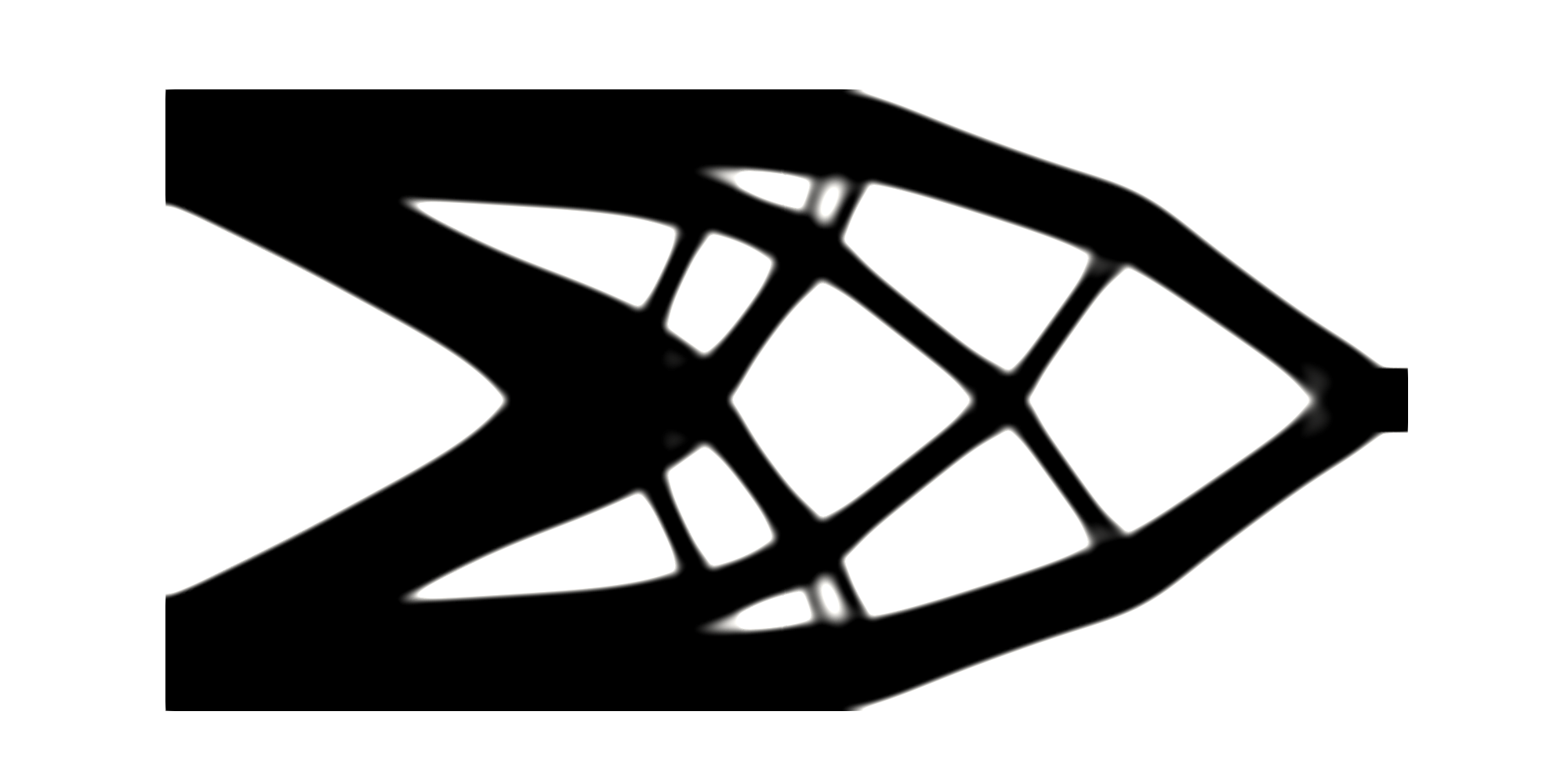}
    }
    \caption{The optimal beam for $\Psi(\lambda_1) = -\lambda_1$, i.e., maximization of the principal eigenvalue for $\gamma=10^{-4}$ (left) and $\gamma=10^{-5}$ (right) with $\eps = 0.02$. 
    We clearly observe finer structures for smaller $\gamma$. We obtain $\lambda_1 = 1.68\cdot 10^{-2}$ for $\gamma=10^{-4}$ 
    and $\lambda_1 =1.72\cdot10^{-2}$ for $\gamma=10^{-5}$. Thus, as expected, with less regularization we reach a larger value for $\lambda_1$.
    }
    \label{fig:num:beamGamma}
\end{figure}

In a second test for the beam setup, we compare the numerical results for different choices of $\Psi(\lambda_1,\lambda_2)$ as a linear combination of $\lambda_1$ and $\lambda_2$.
We set $\gamma=10^{-4}$ and use the solution shown in Figure~\ref{fig:num:beamGamma} as the initialization of the optimization method.
In Figure~\ref{fig:num:beamLam12}, we present numerical results for this setting with
the choice $\Psi(\lambda_1,\lambda_2) = -\lambda_1 - \alpha\lambda_2$ for $\alpha \in \{10^{-2},2\cdot10^{-2},6\cdot10^{-2},10^{-1}\}$.
Moreover, in Table~\ref{tab:num:beamLam12} we list the corresponding values of $\lambda_1$ and $\lambda_2$. Here, $\alpha=0$ corresponds to the result shown in Figure~\ref{fig:num:beamGamma} on the left.

\begin{figure}
    \centering
    \fbox{\includegraphics[trim=0 0 0 0,clip,width=0.23\textwidth]{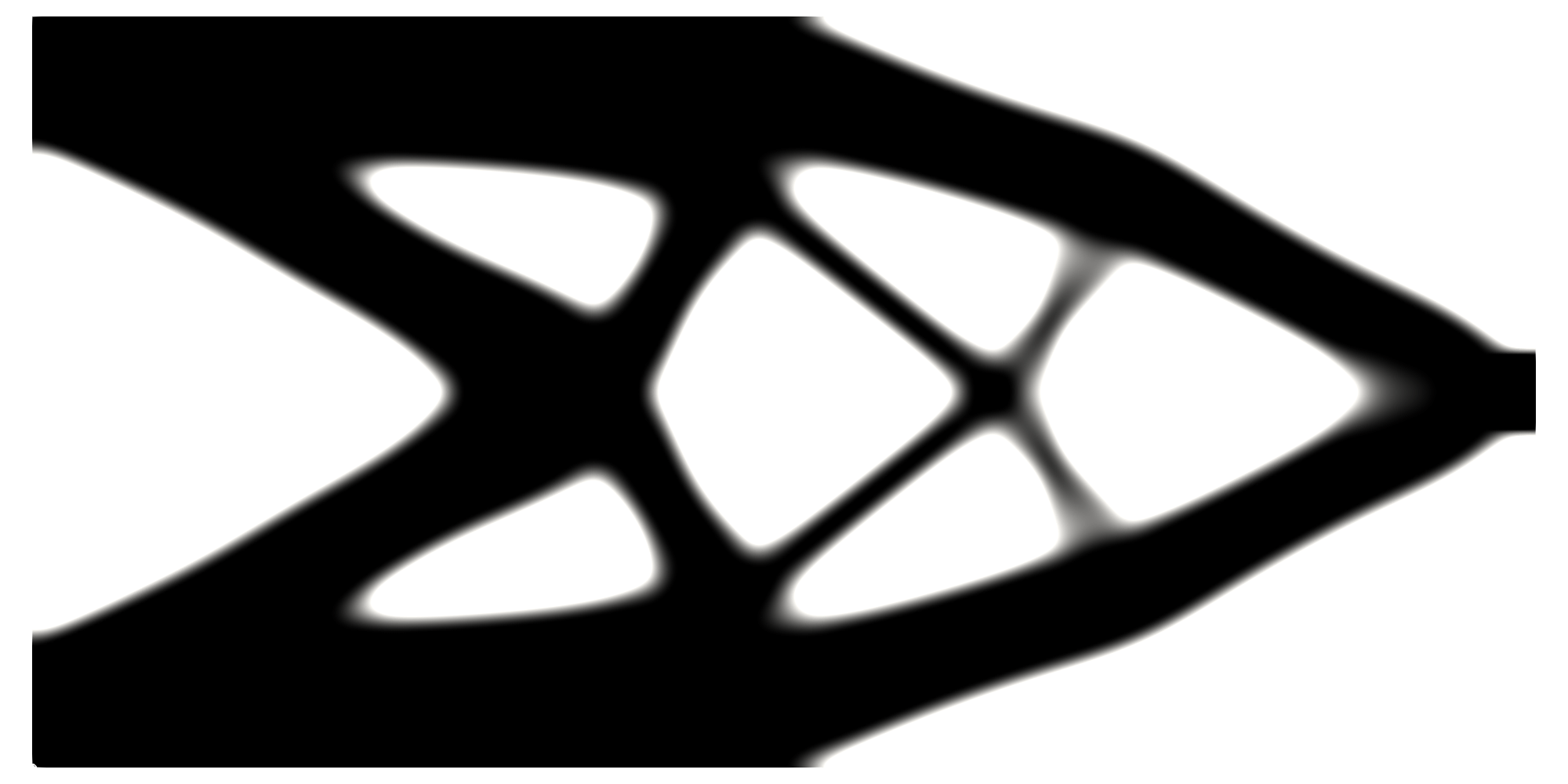}}\hfill
    \fbox{\includegraphics[trim=0 0 0 0,clip,width=0.23\textwidth]{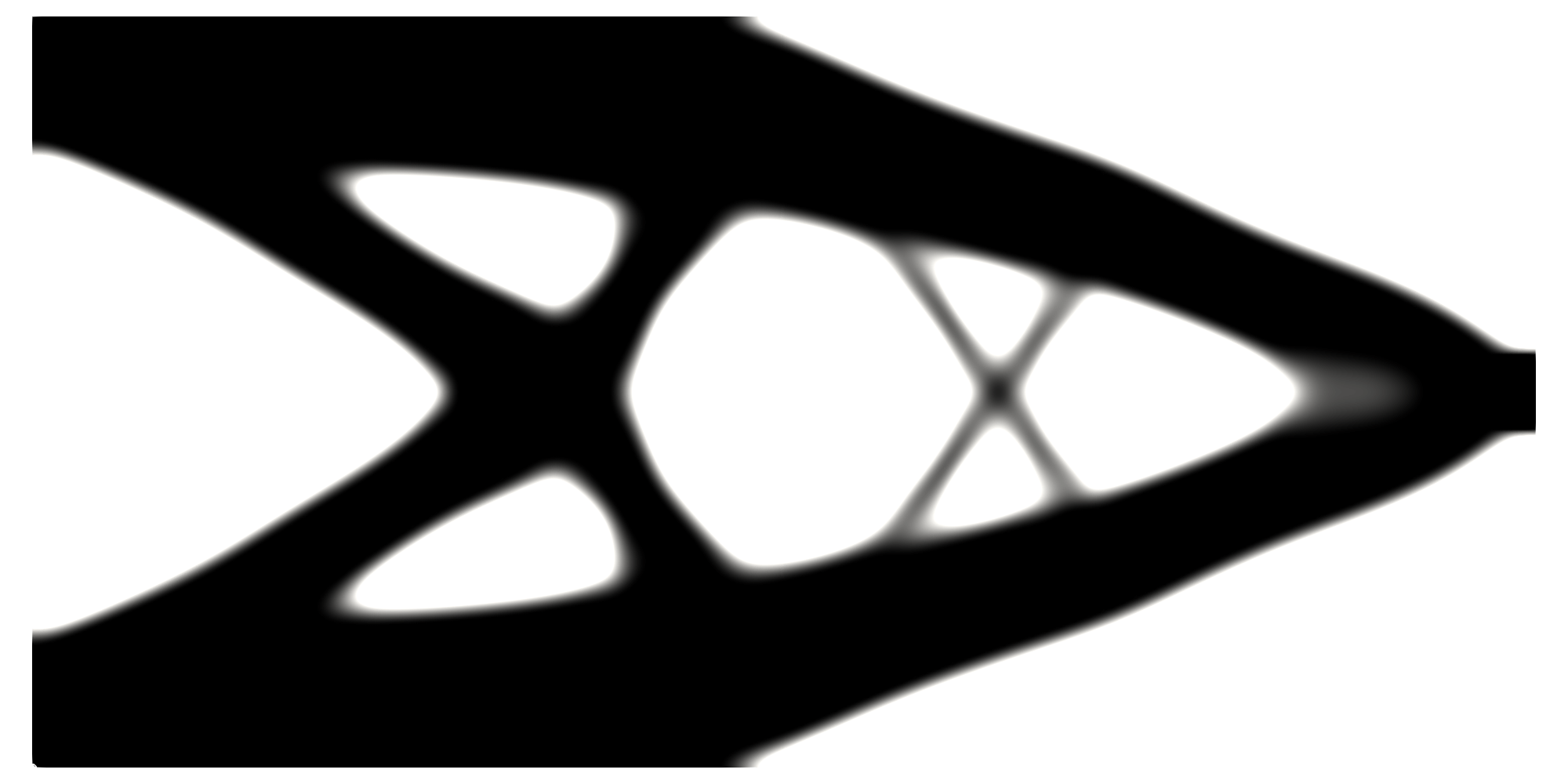}}\hfill
    \fbox{\includegraphics[trim=0 0 0 0,clip,width=0.23\textwidth]{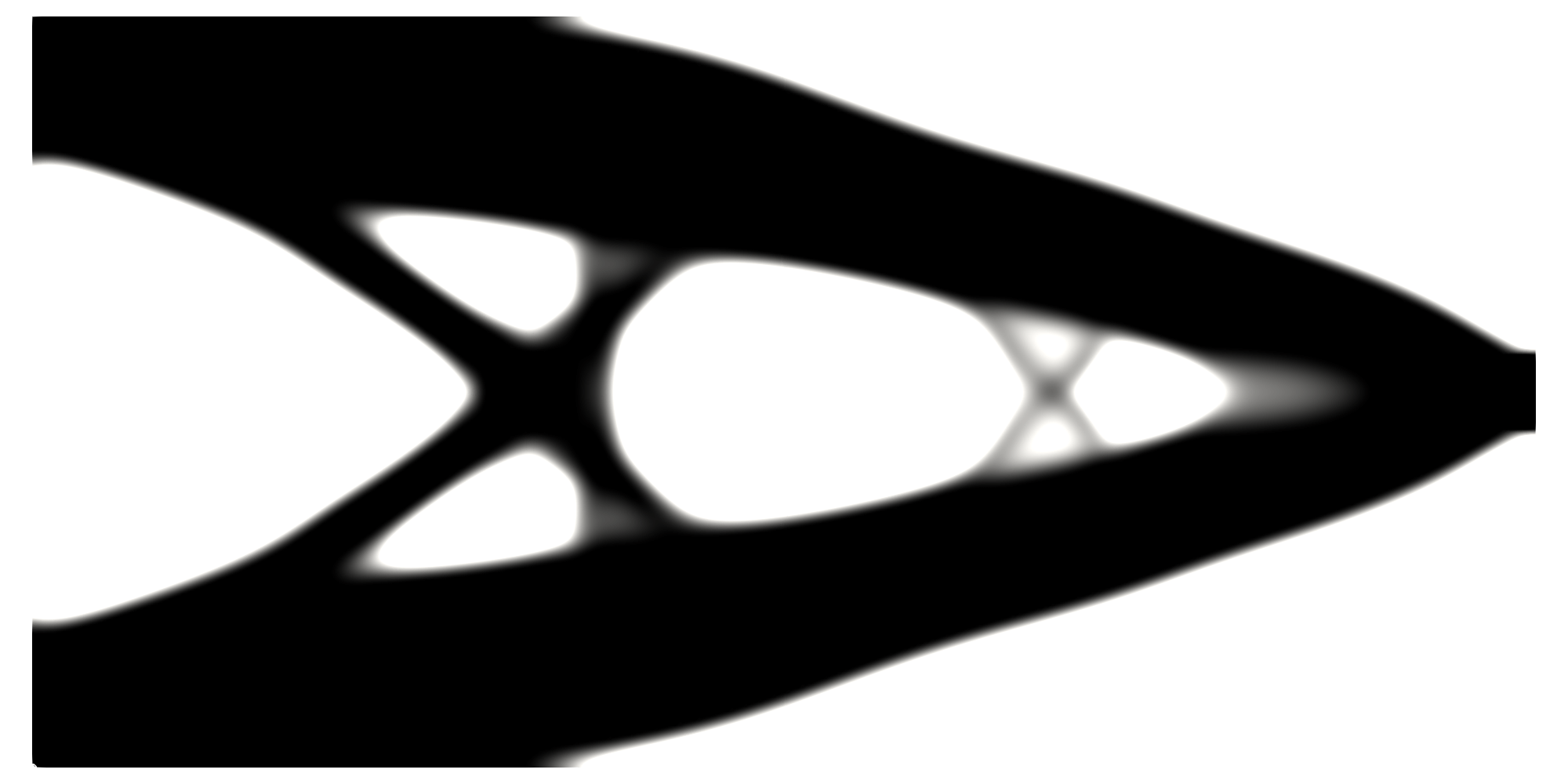}}\hfill
    \fbox{\includegraphics[trim=0 0 0 0,clip,width=0.23\textwidth]{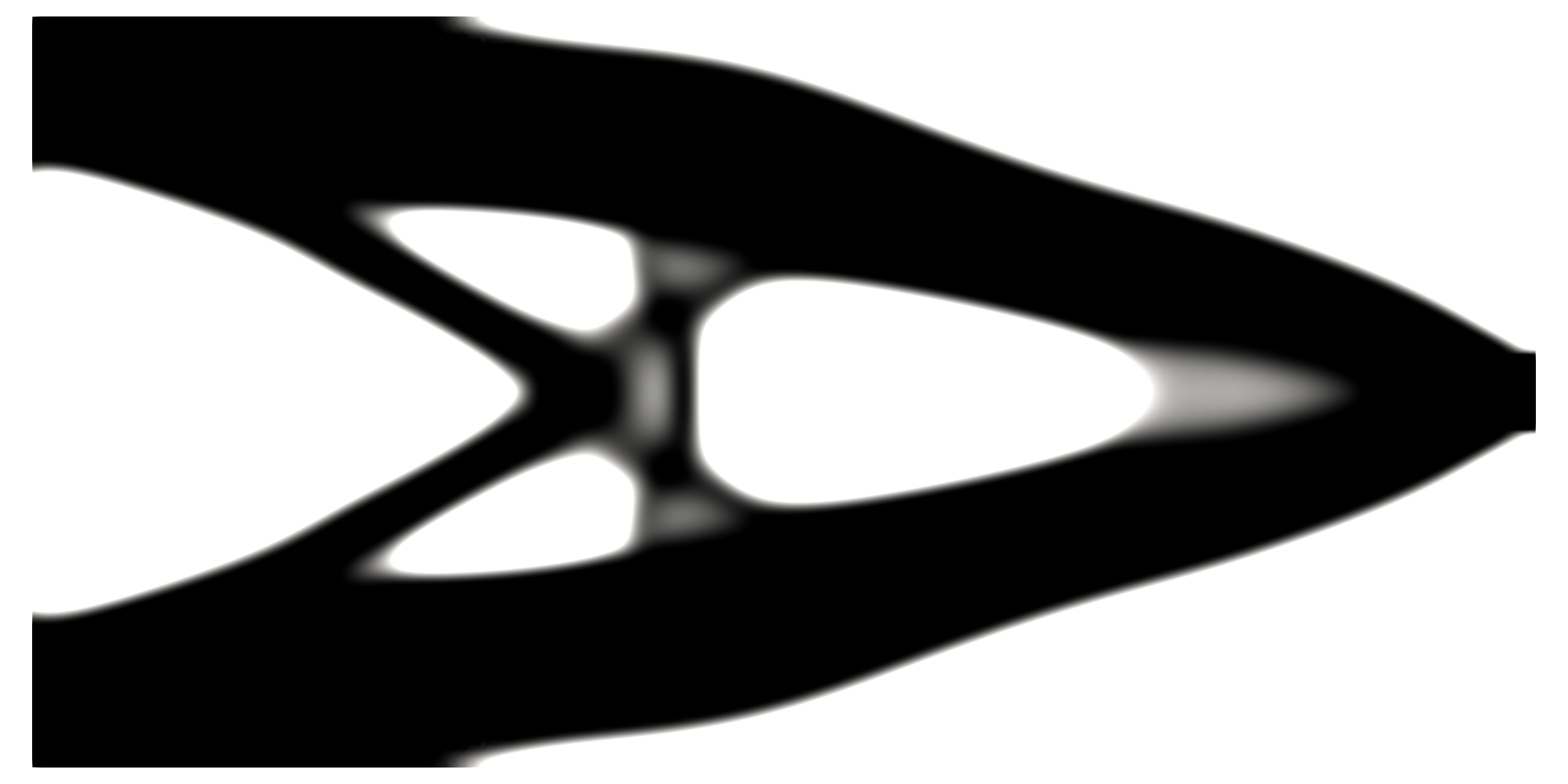}}
    
    \caption{Optimization of the cantilever beam for $\gamma = 10^{-4}$ and
    $\Psi(\lambda_1,\lambda_2) = -\lambda_1 - \alpha\lambda_2$, where
    $\alpha \in  \{ 10^{-2}, {2\cdot10^{-2}}, {6\cdot10^{-2}}, 10^{-1}\}$ (left to right). 
    We observe that increasing the weight of $\lambda_2$ above a certain value reduces the amount of fine structures. 
    }
    \label{fig:num:beamLam12}
\end{figure}

\begin{table}%[]
    \centering\small
    \begin{tabular}{c|rrrrrr}
        $\alpha$ & 0  & $1\cdot10^{-2}$  & $2\cdot10^{-2}$ & $6\cdot10^{-2}$ & $1\cdot10^{-1}$ \\
        \hline
        $\lambda_1$  & $1.677\cdot 10^{-2}$  & $1.662\cdot 10^{-2}$  & $1.606\cdot10^{-2}$ & $1.521\cdot10^{-2}$ & $1.508\cdot 10^{-2}$\\
        $\lambda_2$  & $9.181\cdot 10^{-2}$  & $11.874\cdot 10^{-2}$ & $15.178\cdot10^{-2}$ & $17.663\cdot10^{-2}$ & $18.047\cdot 10^{-2}$
    \end{tabular}
    \caption{The first and second eigenvalue ($\lambda_1$ and $\lambda_2$) for the optimal topologies for the beam example and 
    $\Psi(\lambda_1,\lambda_2) = -\lambda_1 -\alpha \lambda_2 $. 
    As expected, for larger weights $\alpha$ we reach a lower value for $\lambda_1$ 
    and a larger value for $\lambda_2$.
    }
    \label{tab:num:beamLam12}
\end{table}

\subsection{Joint optimization of compliance and principal eigenvalue}
\label{ssec:num:beam-compl}
In this subsection, we extend the problem by using a linear combination of compliance and the first eigenvalue as objective. 
For any given $\varphi \in H^1(\Omega)\cap L^\infty(\Omega)$, the compliance problem is to find a displacement field $\B{u}_c^\varphi \in H^1(\Omega;\R^d)$ satisfying
\begin{equation}
\label{eq:num:compliance}
    \begin{aligned}
    -\nabla\cdot(\mathbb C(\varphi)\mathcal E(\B{u}_c^\varphi)) &= \B{0}\quad \mbox{ in } \Omega,\\
    \B{u}_c^\varphi &= \B{0} \quad \mbox{ on } \Gamma_D \subset \partial\Omega,\\
    \left[\mathbb C(\varphi)\mathcal E(\B{u}_c^\varphi)\right] \cdot \B{n} &= \B{g} \quad \mbox{ on } \Gamma_g \subset \partial\Omega,\\
    \left[\mathbb C(\varphi)\mathcal E(\B{u}_c^\varphi)\right] \cdot \B{n} &= \B{0} \quad \mbox{ on } \Gamma_0\subset \partial\Omega,
\end{aligned}
\end{equation}
which minimizes the objective $\int_{\Gamma_g}\B{g}\cdot \B{u}_c^\varphi $. 

Combining this with our eigenvalue optimization problem for $\Psi(\lambda_1) = -\alpha\lambda_1$ for some $\alpha > 0$, we arrive at
\begin{equation}
\label{OP:COMP}
\left\{
\begin{aligned}
    &\min &&J(\varphi) =-\alpha\lambda_1^{\varphi} + \int_{\Gamma_g}\B{g}\cdot \B{u}_{\B{c}}^{\varphi} + \gamma E^\eps(\varphi)  \\
    &\mbox{over } 
    &&\varphi  \in H^1(\Omega)\cap L^\infty(\Omega),\\
    &\mbox{s.t. } 
    &&\B{u}_{\B{c}}^{\varphi} \mbox{ solves the compliance equation \eqref{eq:num:compliance}},\\
    &&&\lambda_1^{\varphi} \mbox{ is the first eigenvalue of \eqref{state}}.
\end{aligned}
\right.
\end{equation}
This means that we are looking for a structure that simultaneously minimizes the compliance with respect to a given force $\B{g}$ and maximizes 
the first eigenvalue $\lambda_1$. 
The optimization problem \eqref{OP:COMP} is actually a special case of the compliance and eigenvalue optimization problems studied in \cite{Garcke} (with the quantities therein being chosen as $N=2$, $\alpha=1$, $\beta=0$, $\B{f} \equiv \mathbf{0}$, $J_0 \equiv 0$ and $U_c = H^1(\Omega)$, i.e., $S_0=S_1=\emptyset$). More details about the formulation of the problem \eqref{OP:COMP} can be found in \cite[Section~2.7]{Garcke}. For the optimality system of \eqref{OP:COMP}, which we are going to solve numerically, we refer to \cite[Theorem~7.2]{Garcke}.

For the sharp-interface limit of compliance optimization problems \textit{without} eigenvalue optimization (such as \eqref{OP:COMP}, where $\alpha$ is set to zero), we refer to \cite{Blank}. As our sharp-interface analysis for eigenvalue optimization problems without compliance optimization relies on the same expansions as in \cite{Blank}, both approaches can be combined to formally derive the sharp-interface limit of problem \eqref{OP:COMP}.

For our numerical computations, we use the same setup as in Section~\ref{ssec:num:beam} for the beam example and fix $\gamma=1\cdot 10^{-3}$.
Moreover, the exterior force is $\B{g} = (0,-1)^T$ and acts on $\Gamma_g = \{ (2.0,y)\mid y \in (0.45,0.55)\}$. Note that $\Gamma_g$ belongs to the boundary of the domain $\Omega_\rho$ on which we assume a higher value of the density $\rho$.

In Figure~\ref{fig:num:beam-compl}, we show numerical result for this setting for different values of $\alpha$. We observe that the structures become finer when we increase the influence of the principal eigenvalue.
In Table~\ref{tab:num:compl_m_lam1}, we present the corresponding values for compliance and $\lambda_1$ for these shapes.
As expected, we achieve a larger compliance when we increase the weight $\alpha$ of the principal eigenvalue. 
Simultaneously, we also obtain larger values for the principal eigenvalue. It is worth mentioning that these results compare very well with the ones obtained in \cite{Allaire}, where a level-set method was used to directly tackle the sharp-interface problem (see especially Fig.~2 and Fig.~5 in \cite{Allaire}).

\begin{figure}
    \centering 
    \includegraphics[width=0.275\textwidth]{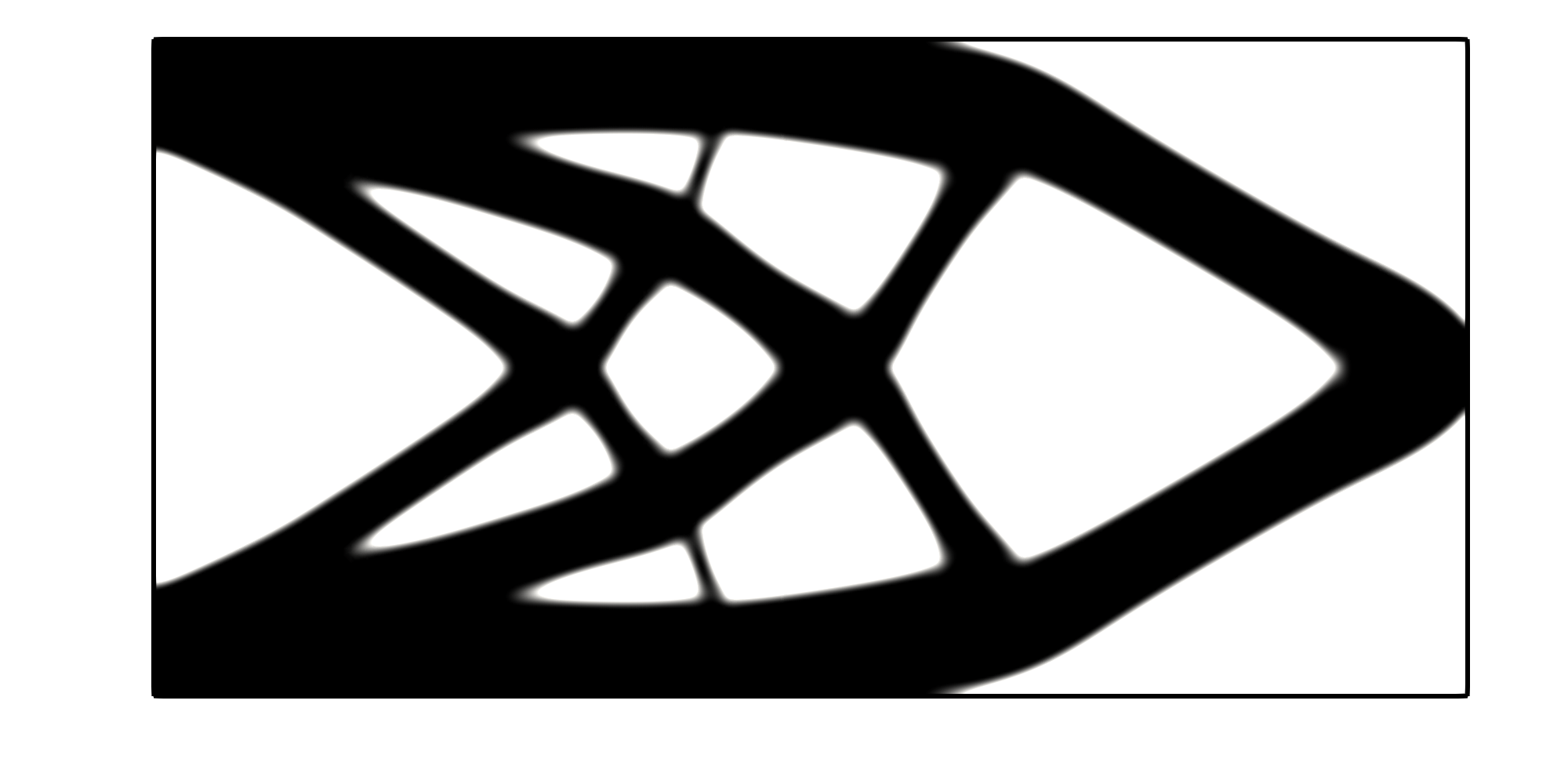}
    \includegraphics[width=0.275\textwidth]{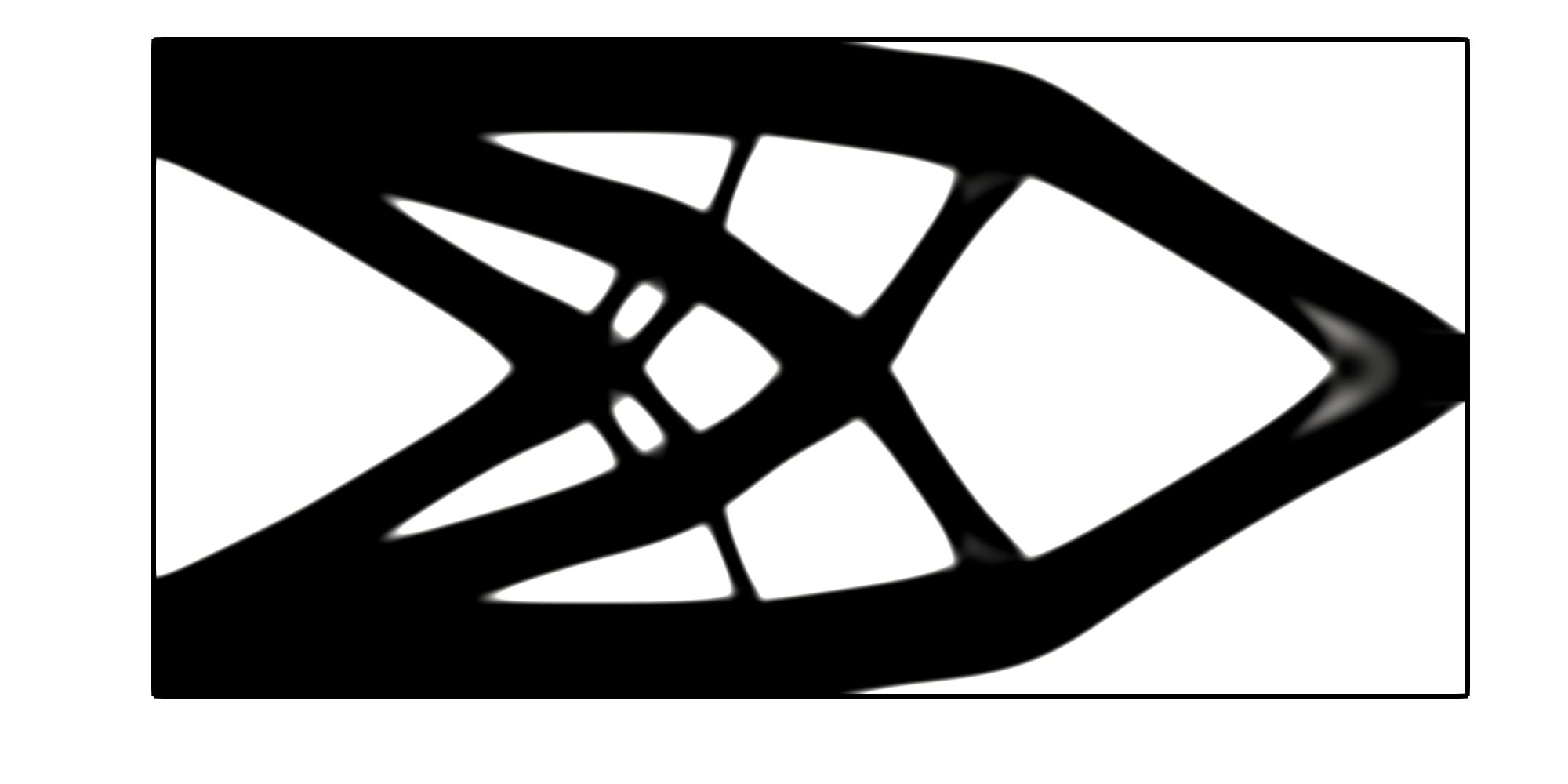}
    \includegraphics[width=0.275\textwidth]{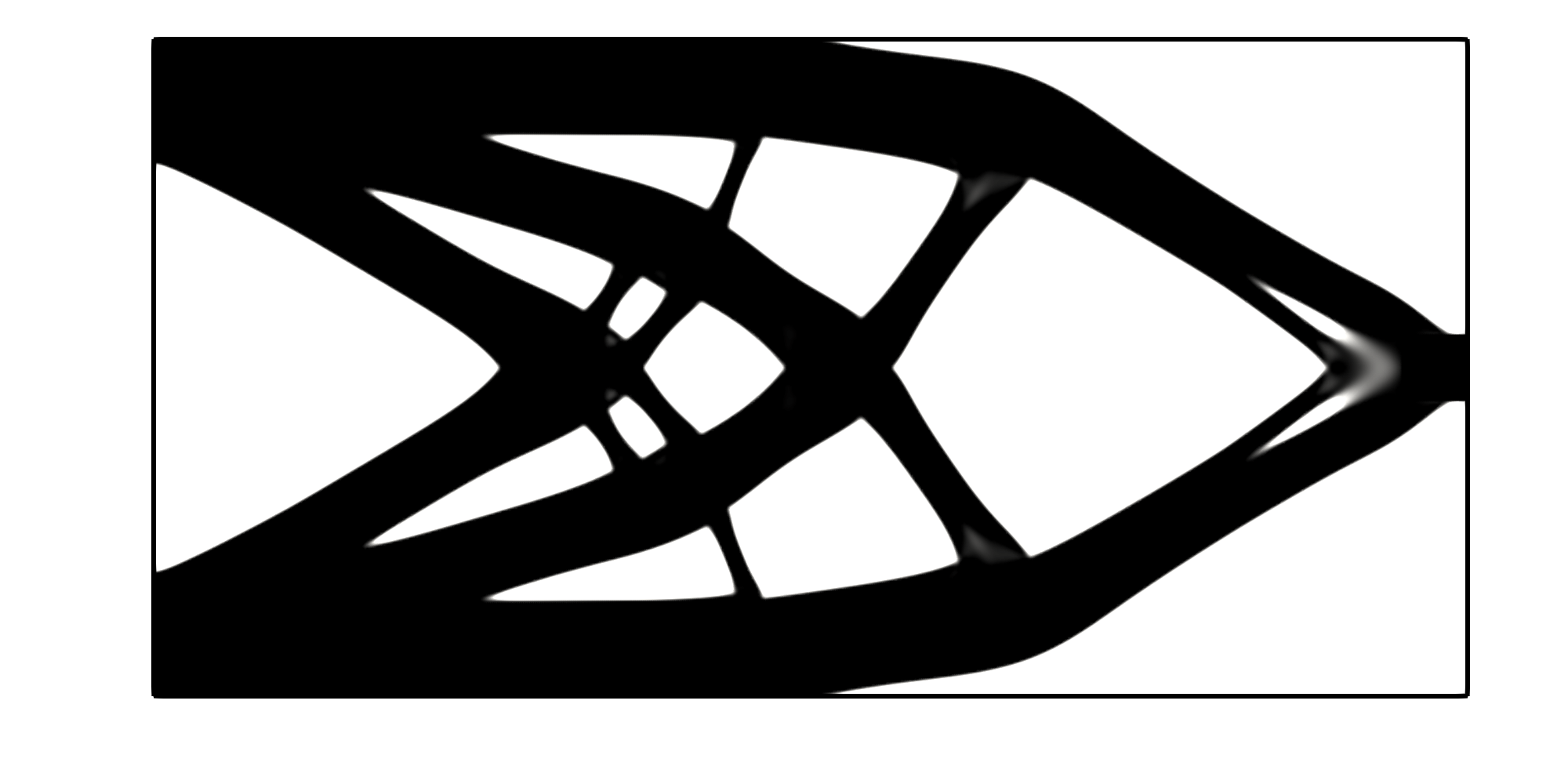}
    \caption{Numerical results for joint optimization of compliance and principal eigenvalue with weight $\alpha\in\{10,100,500\}$ (left to right).
    We observe that increasing the weight $\alpha$ of the first eigenvalue leads locally to a finer structure.}
    \label{fig:num:beam-compl}
\end{figure}

\begin{table}[h]
    \centering\small
    \begin{tabular}{c|ccccc}
        $\alpha$ &   10 & 100 & 200 & 500\\
        \hline
        Compliance  & 0.5507 & 0.5629 & 0.5676 & 0.5769\\
        \hline 
        $\lambda_1$ &  0.0164 & 0.0170 & 0.0172 & 0.0173
    \end{tabular}
    \caption{Values of compliance and principal eigenvalue $\lambda_1$ for joint optimization of compliance and principal eigenvalue 
    with weight $\alpha \in\{10,100,200,500\}$.
    We observe that increasing $\alpha$ leads, as expected, to larger values of the principal eigenvalue and larger values for the compliance.
    }
    \label{tab:num:compl_m_lam1}
\end{table}

%%%%%%%%%%%%%%%%%%%%%%%%%%%%%%%%%%%%
%%%%%%%% ACKNOWLEDGEMENT %%%%%%%%%%%
%%%%%%%%%%%%%%%%%%%%%%%%%%%%%%%%%%%%
 
\section*{Acknowledgment}
Harald Garcke, Paul H\"uttl and Patrik Knopf were partially supported by the RTG 2339 ``Interfaces, Complex Structures, and Singular Limits''
of the German Science Foundation (DFG). The support is gratefully acknowledged.
 
%%%%%%%%%%%%%%%%%%%%%%%%%%%%%%%%%%%%
%%%%%%%%%% BIBLIOGRAPHY %%%%%%%%%%%%
%%%%%%%%%%%%%%%%%%%%%%%%%%%%%%%%%%%%

\renewcommand{\sc}{\scshape}
\scriptsize
\setlength{\parskip}{0pt}
\bibliography{SIA_literature}

\end{document}